\documentclass[hidelinks]{article}
\usepackage[left=3cm,right=3cm,top=2cm,bottom=2cm]{geometry} 
\usepackage{amsmath}
\usepackage{amsthm}
\usepackage{amssymb}
\usepackage{dsfont}
\usepackage{stackengine}
\usepackage{graphicx}
\usepackage{appendix}
\stackMath
\usepackage{hyperref}
\usepackage[svgnames]{xcolor}
\usepackage[style=numeric,backend=bibtex,natbib=true,maxnames=5]{biblatex}
\numberwithin{equation}{section}

\hypersetup{
	colorlinks = true,
	citecolor = teal, 
	linkcolor=red,
	urlcolor=blue}

\DeclareMathOperator{\Tr}{Tr}

\newcommand{\norm}[1]{\left\lVert#1\right\rVert}
\newcommand{\opnormtwo}[1]{\norm{#1}_{\ell^2\to\ell^2}}
\newcommand{\opnorminf}[1]{\norm{#1}_{\ell^\infty\to\ell^\infty}}

\newcommand{\m}{\vect{m}}
\newcommand{\supp}[1]{\operatorname{supp}(#1)}
\DeclareMathOperator{\im}{Im}
\DeclareMathOperator{\re}{Re}
\DeclareMathOperator{\prob}{\mathbb{P}}
\DeclareMathOperator{\Expv}{\mathbb{E}}
\DeclareMathOperator{\dist}{dist}
\DeclareMathOperator{\sign}{sign}

\newcommand{\diag}[1]{\operatorname{diag}\left(#1\right)}
\newcommand{\bigO}[1]{\mathcal{O}\left(#1\right)}

\newcommand{\Oprec}{\mathcal{O}_\prec}

\newcommand{\E}[1]{\Expv\left[#1 \right]}
\newcommand{\Prob}[1]{\prob\left[#1 \right]}
\newcommand{\other}[1]{\widetilde{#1}}

\newcommand{\X}{X}
\newcommand{\Y}{Y}
\newcommand{\s}{\vect{s}}

\newcommand{\vv}{\vect{v}}

\newcommand{\eigB}{\beta}

\newcommand{\alp}{{\varepsilon_0/2}} %This exists because I think alpha is no longer needed
\newcommand{\twoalp}{\varepsilon_0} %This exists because I think alpha is no longer needed
\newcommand{\dom}{\Omega_0}
\newcommand{\vect}[1]{\mathbf{#1}}

\newcommand{\I}{i}
\newcommand{\stab}{B}

\newcommand{\Cmlnt}{\mathcal{C}}
\DeclareFontFamily{OT1}{pzc}{}
\DeclareFontShape{OT1}{pzc}{m}{it}{ <-> s*[1.1] pzcmi7t }{}
\DeclareMathAlphabet{\mathpzc}{OT1}{pzc}{m}{it}

\newcommand{\sng}{E_0}
\newcommand{\Sng}{\mathfrak{S}}
\newcommand{\D}{\mathbb{D}_{E_0,c_0}}
\newcommand{\dM}{\xi}

\newcommand{\dis}{\delta}
\newcommand{\M}{M}
\newcommand{\sol}{\mathpzc{h}}
\newcommand{\gapsol}{\mathpzc{k}}
\newcommand{\gapx}{\mathpzc{x}}

\newcommand{\diff}{\text{\scriptsize$\Delta$}}

\newtheorem{theorem}{Theorem} %
\numberwithin{theorem}{section} 
\newtheorem{lemma}[theorem]{Lemma} %
\newtheorem{Def}[theorem]{Definition} %
\newtheorem{prop}[theorem]{Proposition} %
\newtheorem{claim}[theorem]{Claim} %
\newtheorem{remark}[theorem]{Remark} %
\DeclareFieldFormat*{title}{#1}
\renewbibmacro{in:}{}
\DeclareFieldFormat*{eprint}{\href{http://arxiv.org/abs/#1}{#1}}
\begin{document}
	\begin{minipage}{0.85\textwidth}
		\vspace{2.5cm}
	\end{minipage}
	\begin{center}
		\large\bf Linear Eigenvalue Statistics at the Cusp
		%Edge and Cusp Mesoscopic Eigenvalue Statistics for Wigner-type Matrices
	\end{center}
	
	\vspace{0.5cm}
	
	\renewcommand{\thefootnote}{\fnsymbol{footnote}}	
	\begin{center}
		Volodymyr Riabov\footnotemark[1]\\
		\footnotesize 
		{Institute of Science and Technology Austria}\\
		{\it volodymyr.riabov@ist.ac.at}
	\end{center}
	
	\bigskip

	\footnotetext[1]{Supported by the ERC Advanced Grant "RMTBeyond" No.~101020331}
	\renewcommand*{\thefootnote}{\arabic{footnote}}
	\vspace{0.5cm}
	
	\begin{center}
		\begin{minipage}{0.91\textwidth}\footnotesize{
				{\bf Abstract.}}
			%We establish universal Gaussian fluctuations for the mesoscopic linear eigenvalue statistics of Wigner-type matrices at the edge and cusp singularities, and the local minima of the limiting spectral density. Subsequently, we identify a continuous one-parameter family of functionals that govern the limiting bias and variance. We show that the variance is equivalent to a suitable weighted $\dot{H}^{1/2}$-norm of the test function. Our analysis entirely covers both transitionary regimes: between the square-root singularity at a regular edge and the cubic-root singularity at a cusp, and between a sharp cusp and a non-zero local minimum in the bulk of the limiting spectral density.  Prior to this work, the linear eigenvalue statistics were only studied in the bulk and regular edge regimes.  %TODO:: give this more thought
			%
			We establish universal Gaussian fluctuations for the mesoscopic linear eigenvalue statistics in the vicinity of the cusp-like singularities of the limiting spectral density for Wigner-type random matrices.
			Prior to this work, the linear eigenvalue statistics at the cusp-like singularities were not studied in any ensemble. 
			Our analysis covers not only the exact cusps but the entire transitionary regime from the square-root singularity at a regular edge through the sharp cusp to the bulk. 
			%Prior to this work, only the bulk regime was known for Wigner-type matrices, and the linear eigenvalue statistics at the cusp-like singularities were not studied in any ensemble.
			We identify a new one-parameter family of functionals that govern the limiting bias and variance, continuously interpolating between the previously known formulas in the bulk and at a regular edge. 
			Since cusps are the only possible singularities besides the regular edges, our result gives a complete description of the linear eigenvalue statistics in all regimes.  
		\end{minipage}
	\end{center}
	
	\vspace{0.8cm}
	
	{\small
		\footnotesize{\noindent\textit{Date}: \today}
	}
	
	\vspace{10mm}
	
	\thispagestyle{headings}

\section{Introduction}
%LES in the bulk
Linear eigenvalue statistics are an important tool in studying asymptotic properties of the joint eigenvalue distribution of random matrices. In particular, if $(\lambda_j)_{j=1}^N$ denote the eigenvalues of a standard Wigner matrix $H$, then for any sufficiently regular test function $f$, the centered linear statistics 
\begin{equation} \label{LES}
	\sum_{j=1}^N f(\lambda_j) - \Expv \bigl[\sum_{j=1}^N f(\lambda_j)\bigr]= \Tr f(H) - \Expv\bigl[\Tr f(H)\bigr]
\end{equation}
are asymptotically Gaussian. This result, referred to as the Central Limit Theorem (CLT) for linear eigenvalue statistics, first appeared in the work of Khorunzhy, Khoruzhenko, and Pastur \cite{Khorunzhy1996} in the special case of the resolvent, $f(x):= (x-z)^{-1}, \,\im z > 0 $, and later in a paper by Lytova and Pastur \cite{Lytova2009indCLT} for general $f$. The moment conditions on the matrix elements and the regularity assumptions on the test function $f$ were optimized in \cite{Bao2016, Landon2022, Shcherbina2011}. Remarkably, CLT holds without the
customary $N^{-1/2}$ normalization in \eqref{LES}, indicating a strong correlation among eigenvalues.    

For certain classes of test functions $f$, the variance of the limiting normal distribution can be explicitly calculated. 
In particular, 
one can introduce an $N$-dependent scaling around a fixed reference energy $\sng$ in the spectrum, and consider a \textit{scaled test function} $f(x) := g(\eta_0^{-1}(x-\sng))$, where $g$ is a sufficiently smooth compactly supported $N$-independent function, and $\eta_0\equiv\eta_0(N) \ll 1$ is a scaling parameter above the local eigenvalue spacing at $\sng$. For such \textit{scaled test functions}, the limiting variance is universal, i.e., it does not depend on the precise distribution of the entries but is influenced significantly by whether $\sng$ is located in the bulk of the spectrum or at the singular points, e.g., the spectral edges. 

To highlight this phenomenon, consider standard Wigner matrices with the limiting spectral density given by the semicircle law $\rho(x) =(2\pi)^{-1} \sqrt{(4-x^2)_+}$, the bulk corresponds to $(-2,2)$, and the spectral edges are located at $\pm2$. 
%
%
%
%The precise expression is universal in the sense that it is independent of the precise distribution of the matrix entries, but it depends heavily on whether $f$ is supported in \textit{the bulk}, i.e., where the eigenvalue density $\rho$ is bounded from below and analytic, or near singular points, e.g., square root singularities at the spectral edges. For example, for standard Wigner matrices, the density is the semicircle law $\rho(x) = \frac{1}{2\pi} \sqrt{(4-x^2)_+}$, the bulk regime corresponds to $(-2,2)$, and the spectral edges are located at $\pm 2$. 
%
%To highlight the dependence on the support of $f$, 
% 
If the reference energy $\sng$ lies in $(-2,2)$, then in the \textit{mesoscopic} regime $N^{-1}\ll\eta_0\ll 1$, the limiting variance of the linear eigenvalue statistics for the scaled test functions is given by
\begin{equation} \label{bulkCLT}
	\frac{1}{2\boldsymbol{\beta}\pi^2}\iint_{\mathbb{R}^2}\frac{\bigl(g(x)-g(y)\bigr)^2}{(x-y)^2}\mathrm{d}x\mathrm{d}y.
\end{equation}
Here the symmetry parameter $\boldsymbol{\beta} = 1,2$ corresponds to $H$ being real symmetric and complex Hermitian, respectively. 
For standard Wigner matrices, the \textit{mesoscopic CLT} in the bulk was %first obtained for Gaussian Orthogonal $H$ in \cite{BoutetKhoruI}, with subsequent extensions to Wigner matrices in \cite{BoutetKhoruII} for $N^{-1/8}\ll \eta_0\ll 1$. 
obtained in \cite{He2017WignerCLT}. %, the result was proved down to the optimal $N^{-1}$ scale and was since extended to ensembles of greater generality in the recent works \cite{BaoSchnelliXu, LiSchnelliXu, LiXu, Landon2021Wignertype}.
For a more comprehensive account of mesoscopic CLT in the bulk of the spectrum, we refer to \cite{Landon2020applCLT, Li2021genWigner} and references therein.

%LES near singular points: edges, but no cusps

CLT also holds for mesoscopic linear eigenvalue statistics at the spectral edges $\sng = \pm2$, where the density of states exhibits a regular square-root singularity. In contrast to the bulk regime, however,  the mesoscopic range is limited to $N^{-2/3}\ll \eta_0 \ll 1$, and the universal limiting variance is given by 
\begin{equation} \label{edgeCLT}
	\frac{1}{4\boldsymbol{\beta}\pi^2}\iint_{\mathbb{R}^2}\frac{\bigl(g(\mp x^2)-g(\mp y^2)\bigr)^2}{(x-y)^2}\mathrm{d}x\mathrm{d}y =  \frac{1}{4\boldsymbol{\beta}\pi^2}\iint_{x,y>0}\frac{\bigl(g(\mp x)-g(\mp y)\bigr)^2}{(x-y)^2}\frac{x+y}{\sqrt{xy}}\mathrm{d}x\mathrm{d}y,
\end{equation}
where the signs $+$ and $-$ correspond to the left and right edges of the spectrum, respectively. This result was first established by Basor and Widom in \cite{BasorWidom} for Gaussian Unitary matrices and was extended to the case of Gaussian Orthogonal Ensemble by Min and Chen in \cite{MIN2020114836}. In subsequent works, the mesoscopic CLT was obtained in the settings of Dyson Brownian motion \cite{AdhikariHuang}, and deformed Wigner\footnote{Deformed Wigner matrices are random matrices of the form $H=W+A$, where	$W$ is a standard Wigner matrix and $A$ is deterministic. } and sample covariance matrices \cite{Li2021deformed} at the regular edges. Note that the kernel function $\mathcal{K}_{\mathrm{edge}}(x,y) = \frac{x+y}{\sqrt{xy}}\mathds{1}_{x,y>0}$ is universal at the edge. 

While the spectral density of standard Wigner matrices exhibits only regular edge singularities,  more general ensembles, such as deformed Wigner matrices and Wigner-type matrices, possess a richer singularity structure. In particular, special \textit{cusp-like} singularities may occur. %The support of the density may consist of multiple intervals, and special \textit{cusp-like} singularities may occur when a gap between two support intervals is very small. 
These singularities require much more delicate analysis; in fact, all prior research on linear eigenvalue statistics was limited to the bulk and regular edge regimes. In particular, cusp-like singularities were explicitly ruled out by the assumptions of \cite{Li2021deformed}. In the present paper, we give the first complete description of  the mesoscopic linear eigenvalue statistics at all types of local minima and singular points of the limiting eigenvalue density, including the cusps, in the setting of Wigner-type matrices. We note that for Wigner-type matrices, even the regular edge regime for the mesoscopic CLT has not been considered before. 

%Singularities beyond regular edges
%TODO:: rewrite this, clearly separate deformed Wigner and Wigner type matrices.
We chose the Wigner-type ensembles as our model for analyzing the linear eigenvalue statistics near the cusp-like points because of the rich and well-understood singularity structure of their associated spectral densities. Wigner-type matrices were introduced in \cite{Ajanki2016Univ}; they consist of centered entries $H_{jk}$ independent up to the symmetry constraint $H = H^*$, with the matrix of variances $S_{jk} := \Expv |H_{jk}|^2$ satisfying the \textit{flatness} assumption $S_{jk} \sim N^{-1}$. They are a natural generalization of the standard Wigner matrices that correspond to the case $S_{jk} = N^{-1}$,  while the so-called \textit{generalized Wigner matrices} correspond to the special case of  $S$ being stochastic.

%The deformed Wigner matrices $W := H+A$ are obtained by introducing an additive deformation, given by a deterministic self-adjoint matrix $A$. In the setting of \cite{Li2021deformed}, the matrix $A := \diag{(a_j)_{j=1}^N}$ is assumed to be diagonal. 
The spectral density  $\rho$  of Wigner-type matrices is  typically  not semicircular and instead is recovered from the unique solution $\m:= (m_j)_{j=1}^N$ of the \textit{vector Dyson equation}
\begin{equation} \label{VDE}
	-\frac{1}{m_j(z)} = z + \sum\limits_{k=1}^N S_{jk} m_{k}(z), \quad \im m_j(z) \im z>0,
\end{equation}
via the Stieltjes inversion formula applied to $N^{-1}\sum_{j=1}^N\im m_j(E+\I\eta)$. 

%Densities that arise from VDE
The spectral densities which arise from \eqref{VDE} under the flatness condition were fully characterized in the works \cite{Ajanki2017SingularitiesQVE, Ajanki2019QVE} of Ajanki, Erd\H{o}s, and Kr\"{u}ger. In particular, the resulting measure may have multiple support intervals of order one length with square-root singularities at the edges, possible cubic-root cusp in the interior of the support intervals, and no other types of singularities occur. The gaps between support intervals can be small, leading to a non-regular behavior of density near the adjacent edges. The spectral density can also have small but non-zero local minima that exhibit cusp-like behavior. The structure of the set of small local minima was described by Alt, Erd\H{o}s, and Kr\"{u}ger in \cite{Alt2020energy}, and the universality of local eigenvalue statistics in the near-cusp regime was proved for the complex Hermitian case by Erd\H{o}s, Kr\"{u}ger, and Schr\"{o}der in \cite{Erdos2018CuspUF}, and for the real symmetric case by Cipolloni et al. in \cite{Cipolloni2018CuspU}.

%Main result
In our main result, Theorems \ref{th_main} and \ref{th:rho>0}, we  prove the universal Gaussianity of the linear eigenvalue statistics and precisely identify a continuous one-parameter family of its variance and bias functionals at all singularities and local minima in the support of the spectral density in the large $N$ limit. 
The parameter $\alpha$ characterizing the bias and variance describes the deviation of the %local minimum at $\sng$ 
local density profile near $\sng$ %resolved on the scale $\eta_0$
from an exact cusp singularity, see Figure \ref{fig:rho_fig}. More precisely, $\alpha$ is given by the large $N$ limit of the ratio $\ell/\eta_0$, between the \textit{length-scale} $\ell$ of the density profile at $\sng$ and the mesoscopic scaling parameter $\eta_0$.
When $\sng$ is an endpoint of a small gap in the support, then the length-scale $\ell$ is proportional to the size of the gap $\Delta$ (see Definition \ref{def_sng} below), while for non-zero local minimum at $\sng$, the parameter $\ell$ is negative and proportional to $\rho(\sng)^3$.  
\begin{figure}
	\centering
	\includegraphics[width=.3\textwidth]{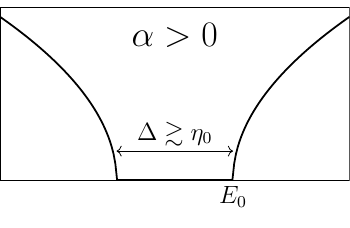}\hfill
	\includegraphics[width=.3\textwidth]{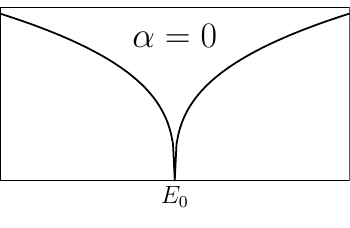}\hfill
	\includegraphics[width=.3\textwidth]{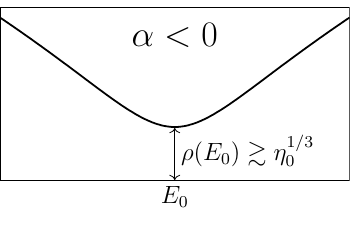}\hfill
	\caption{The three variants of a cusp-like singularity at the point $\sng$, and the corresponding ranges for the parameter $\alpha$.}
	\label{fig:rho_fig}
\end{figure}

We derive explicit formulas for the limiting variance and bias functionals in terms of universal kernel functions, which depend only on the  shape of the density $\rho$ near  $\sng$ resolved on the scale $\eta_0$. In particular, if the gap adjacent to $\sng$ is much larger than $\eta_0$, the limiting distribution of the linear eigenvalue statistics matches that near a regular edge, with the variance given by \eqref{edgeCLT}.  Similarly, if the density $\rho(\sng)$ is much larger than $\eta_0^{1/3}$, then the linear eigenvalue statistics exhibit bulk-like behavior, with the limiting variance given by \eqref{bulkCLT}. Therefore, all previously known formulas for the bias and variance of the linear eigenvalue statistics are only applicable in the regime where $|\ell|$ is much larger than the mesoscopic scale $\eta_0$.

The novel cusp-like statistics emerge in the complementary regime, when $|\ell|\lesssim \eta_0$,  and are characterized by three new universal kernel functions $\mathcal{K}_{\mathrm{gap}}$, $\mathcal{K}_{\mathrm{min}}$, and $\mathcal{K}_{\mathrm{cusp}}$, which play a universal role analogous to $\mathcal{K}_{\mathrm{edge}}$ in \eqref{edgeCLT}. In the small gap regime, if $\Delta$ and $\eta_0$ are comparable in size, the limiting variance is given by a weighted $\dot{H}^{1/2}$-norm of an $\alpha$-dependent rescaling of the test function $g$, with a universal kernel $\mathcal{K}_{\mathrm{gap}}$ that does not depend on $\alpha$ (see \eqref{main_gap} below for the bias and variance formulas).  
Similarly, if $\sng$ is a non-zero local minimum, and $\rho(\sng) \sim \eta_0^{1/3}$, then the variance is governed by an $\alpha$-independent kernel $\mathcal{K}_{\mathrm{min}}$ (defined in \eqref{main_min}). Finally, if $|\ell|$ is much smaller than $\eta_0$,  then the third kernel $\mathcal{K}_{\mathrm{cusp}}$, given in \eqref{main_cups} below, arises.   

These regimes fully describe the limiting distribution of the linear eigenvalue statistics at the singular points and the local minima in the self-consistent spectrum in the entire mesoscopic range above the local eigenvalue spacing at $\sng$.  
Under appropriate $\alpha$-dependent rescaling, the kernel functions $\mathcal{K}_{\mathrm{min}}$, $\mathcal{K}_{\mathrm{cusp}}$, and $\mathcal{K}_{\mathrm{gap}}$ interpolate non-trivially between their previously known bulk and regular-edge counterparts, $\mathcal{K}_{\mathrm{bulk}}(x,y) = 2$ and $\mathcal{K}_{\mathrm{edge}}(x,y) = \frac{x+y}{\sqrt{xy}}\mathds{1}_{x,y>0}$, see \eqref{bulkCLT} and \eqref{edgeCLT}, with the transitions recovered by taking the limits $\alpha \to \mp\infty$, respectively. 
%The kernel functions $\mathcal{K}_{\mathrm{gap}}$, $\mathcal{K}_{\mathrm{min}}$, and $\mathcal{K}_{\mathrm{cusp}}$ are non-trivial and qualitatively different from their previously known bulk and regular-edge counterparts.

%Proof strategy and ingredients
Our proof strategy consists of two main steps: analyzing the characteristic function of the linear eigenvalue statistics \eqref{LES} for a Wigner-type matrix $H$, and identifying the novel kernels. 

Working with the characteristic function is a standard method, see, e.g., \cite{Landon2020applCLT, Lytova2009indCLT, R2023bulk}, and requires understanding a special two-point function of the resolvents 
\begin{equation} \label{Txy}
	T_{xy}(z,\zeta) := \sum_{a\neq y} S_{xa}G_{ay}(z)G_{ya}(\zeta),
\end{equation}
where $G(z) := (H-z)^{-1}$. Note that in the standard Wigner case, $S_{xa} = N^{-1}$, the sum \eqref{Txy} can be linearized using the resolvent identity. However, no such algebraic simplification is available for Wigner-type matrices. 
Therefore, the key probabilistic ingredient is a local law (Theorem \ref{th:Tlocallaw}) for the two-point function $T(z,\zeta)$, which we prove by analyzing the corresponding two-body \textit{stability operator} $\stab(z,\zeta):=1-\m(z)\m(\zeta)S$ (defined in \eqref{stab_def} below). The local law for $T(z,\zeta)$ was already obtained in \cite{R2023bulk} in this way, but only in the bulk of the spectrum.

The main difficulty compared to \cite{R2023bulk} arises from the decay of the smallest (in absolute value) eigenvalue  of the one-body stability operator $\stab(z,z)$ near singular points in the spectrum. As was established in \cite{Alt2020energy}, near the cusp-like points, the destabilizing eigenvalue vanishes like $\rho(z)^2$, in contrast to being order one in the bulk, and order $\rho(z)$ at a regular edge. To obtain error estimates that are sufficiently strong to prove the Central Limit Theorem in the full mesoscopic range, we need to counteract the instability of $\stab(z,z)^{-1}$. The improvement comes from the local symmetry of the density near its cusp-like points, which leads to a intricate  cancellation in the error term of the local law for $T(z,\zeta)$. We capture these cancellations by employing the additional directional information encoded in the single resolvent local law for Wigner-type matrices proved by Erd\H{o}s, Kr\"{u}ger, and Schr\"{o}der in \cite{Erdos2018CuspUF}. The same mechanism is crucial for obtaining effective error estimates on other quantities involving two resolvents throughout the proof.

In the second step of the proof, we focus on computing the variance and bias and analyzing the associated integral kernels, which emerge in the first probabilistic step. The main novel ingredient in this step is the asymptotic expansion for the smallest eigenvalue $\eigB(z,\zeta)$ of the stability operator $\stab(z,\zeta)$ (Propositions \ref{prop_eigval} and \ref{prop:eigB}). Understanding the precise $z-\zeta$ behavior of the eigenvalue $\eigB$ is essential because it characterizes the singularities of the variance kernel. In particular, we find that in the perturbative regime $|\eigB(z,\zeta)|\cdot\norm{\m(z)-\m(\zeta)} \sim |z-\zeta|$. For example, near an exact cusp point, this relation is a combination of two delicate fractional power scalings $|\eigB(z,\zeta)| \sim |z-\zeta|^{2/3}$ and $\norm{\m(z)-\m(\zeta)} \sim |z-\zeta|^{1/3}$.  The non-analytic nature of $\rho$ near the singularities prevents the use of standard techniques applicable in the bulk, and in prior work, such expansions near cusp-like points were obtained only for the special case of $\zeta = z, \bar{z}$ in the work of Alt, Erd\H{o}s, and Kr\"{u}ger \cite{Alt2020energy}. 
Further analysis relies on the perturbative expansion of the solution vector $\m$ near the local minima of the spectral density. The extensive analysis in \cite{Ajanki2019QVE} and \cite{Alt2020energy} shows that the leading term of this expansion is governed by an explicit shape function, which, up to scaling, depends only on the type of the singularity at $\sng$. Expressing the kernels in terms of the universal shape functions allows us to obtain the desired explicit integral formulas for the limiting bias and variance. 

The paper is organized in the following way. Section \ref{sec:main_results} contains the assumptions on the model and our main result, Theorems \ref{th_main} and \ref{th:rho>0}. In Section \ref{sec:prelim}, we collect the preliminary results we refer to throughout the paper. Section \ref{sec:proof_main} contains the proof of Theorem \ref{th_main} based on Propositions \ref{prop:main1} and \ref{prop:main2}. Proposition \ref{prop:main1} contains an effective estimate on the characteristic function of the linear eigenvalue statistics \eqref{LES}, which we prove in Subsections \ref{sec:main1proof}. We prove the local law for the two-point function $T(z,\zeta)$ and other technical inputs for Proposition \ref{prop:main1} in Subsection \ref{sec:curlyT_proof}. Proposition \ref{prop:main2} provides explicit formulas for the bias and the variance obtained from Proposition \ref{prop:main1} and establishes the continuity of the bias and variance functionals defined in Theorem \ref{th_main}. We prove Proposition \ref{prop:main2} in Sections \ref{sec:main2proof}. Section \ref{sec:pert} contains the necessary perturbative estimate on the stability operator and related quantities. In Sections \ref{sec:variance_lemma_proof} and \ref{sec:bias}, we analyze the variance and expectation of the linear eigenvalue statistics and establish the intermediate technical steps leading to the proof of Proposition \ref{prop:main2}. Finally, in Section \ref{sect:smallrho}, we outline the ingredients necessary for extending our results to the case of non-zero local minima, stated in Remark \ref{th:rho>0}.

%The sharp cusps and the cusp-like singularities, produced by two support intervals coming close together and merging, result in new universal local statistics described by Pearcey kernel, as was established by Cipolloni, Erd\H{o}s, Kr\"{u}ger, and Schr\"{o}der in \cite{Cipolloni2018CuspU, Erdos2018CuspUF}.

%The local law for the resolvent, and the universality of local eigenvalue statistics was established in \cite{Ajanki2016Univ} and \cite{Erdos2018CuspUF}. 

\section{Main Result} \label{sec:main_results}
We begin with the definition of Wigner-type matrices originally introduced in Section 1.1 of \cite{Ajanki2016Univ}.
\begin{Def}[Wigner-type matrices] \label{WT_def}
	Let $H = \left(H_{jk} \right)_{j,k=1} ^N $ be an $N\times N$ matrix with independent entries up to the Hermitian symmetry condition $H = H^*$ satisfying $\E{H_{jk}} = 0$.
	
	We consider both real and complex Wigner-type matrices. In case the matrix $H$ is complex, we assume additionally that $\re H_{jk}$ and $\im H_{jk}$ are independent and $\Expv[H_{jk}^2] = 0$ for $k\neq j$.
	
	\textbf{Assumption (A).} Denote by $S$ the matrix of variances $S_{jk} := \Expv[|H_{jk}|^2]$. We assume that $S$ satisfies
	\begin{gather} \tag{A} \label{cond_A}
		\frac{c_{inf}}{N} \le S_{jk} \le \frac{C_{sup}}{N},
	\end{gather}
	for all $j,k \in \{1,\dots, N\}$ and some strictly positive constants $C_{sup}, c_{inf}$.
	
	\textbf{Assumption (B).} We assume a uniform bound on all other moments of $\sqrt{N}H_{jk}$, that is, for any $p \in \mathbb{N}$ there exists a positive constant $C_p$ such that
	\begin{gather} \tag{B} \label{cond_B}
		\E{|\sqrt{N}H_{jk}|^p} \le C_p
	\end{gather}
	holds for all $j,k \in \{1,\dots, N\}$.
	
	\textbf{Assumption (C).} We assume that the unique (Theorem 4.1 in \cite{Ajanki2016Univ}) solution $\m := (m_j)_{j=1}^N$ of the vector Dyson equation \eqref{VDE} satisfies the bound
	\begin{gather} \tag{C} \label{cond_C} %REV 1.4
		|m_j(z)| \le C_{\mathrm{m}}, \quad z\in \mathbb{C},\quad j \in \{1,\dots, N\}, 
	\end{gather}
	uniformly in $N$ for some positive constant $C_{\mathrm{m}}$. General sufficient conditions on the matrix of variances $S$ for the assumption \eqref{cond_C} to hold were obtained in Chapter 6 of \cite{Ajanki2019QVE}.
\end{Def}

The \textit{self-consistent density of states} $\rho(E)\equiv\rho_N(E)$ is defined by the Stieltjes inversion formula, 
\begin{equation} \label{rho_def}
	\rho(E) := \pi^{-1}\lim_{\eta\to+0}\im m(E+\I\eta),
\end{equation}
where $m(z) := N^{-1}\sum_{j=1}^N m_j(z)$.

\begin{Def}[Singularities in the self-consistent spectrum]	\label{def_sng}
	Let $\rho$ be the self-consistent density of states defined in \eqref{rho_def}. Let $\mathcal{I} \subset \mathbb{R}$ be the set on which $\rho$ is positive, then $\supp{\rho} = \overline{\mathcal{I}}$. A point $E\in \supp{\rho}$ is called a singularity in the self-consistent spectrum if and only if $\rho(E) = 0$. Let $\Sng \equiv \Sng_N$ be the set of all singularities, then $\Sng = \partial\mathcal{I}$.
	
	For any $E\in\mathfrak{S}$, define the size of the adjacent gap in the support $\Delta(E)$ by 
	\begin{equation}
		\Delta(E) := \sup\{|b-a| : E \in [a,b] \subset \mathbb{R}\backslash\mathcal{I}\}.
	\end{equation}
	A point $E\in \mathfrak{S}$ is called a cusp if $\Delta(E) = 0$, and an edge if $\Delta(E) > 0$. For all edge points $E$, we define the sign $\widehat{s}(E) \in \{-1,1\}$ in such a way that $E + \widehat{s}(E)x \in \supp{\rho}$ for all $x \in (0, \varepsilon)$ with some $\varepsilon>0$. For definiteness, let $\widehat{s}(E) := 0$ for all cusp points. %Furthermore, a point $E$ is called a simple edge if %$\Delta(E) = +\infty$.
	%$\Delta(E) \ge \Delta_*$ for some positive $N$-independent threshold $\Delta_*$.
\end{Def}

In addition to the singular points in the self-consistent spectrum, our methods allow us to study mesoscopic eigenvalue statistics at small but non-zero local minima of $\rho$.
\begin{Def}[Set of small local minima] 
	Let $\rho$ be the self-consistent density defined in \eqref{rho_def}. We consider points $\sng$, which lie in the set 
	\begin{equation}
		\mathcal{M}_{\rho_*} := \{E \in \supp{\rho}\backslash\partial\supp{\rho}\, :\, \rho(E) \le \rho_* \text{ and } E \text{ is a local minimum of } \rho\},
	\end{equation}
	where $\rho_*>0$ is a small positive threshold that depends only on the constants in Assumptions \eqref{cond_A}$-$\eqref{cond_C}.
\end{Def}
The structure of the set $\mathcal{M}_{\rho_*}$ was studied in \cite{Alt2020energy}. In particular, $\mathcal{M}_{\rho_*}$ consists of a finite number of points, which are an order one distance away from each other and from the end-points of the support intervals of $\rho$ (see Theorem 7.2 in \cite{Alt2020energy}).

We use the notion of the \textit{self-consistent fluctuation scale} introduced in \cite{Erdos2018CuspUF}, that corresponds to the typical eigenvalue spacing around the energy $E$. For a reference energy $E\in\supp{\rho}$ with $\rho(E) = 0$, the natural fluctuation scale $\eta_{\mathfrak{f}}(E)$ is given by
\begin{equation} \label{etaf}
	\eta_\mathfrak{f}(E) := \max\{N^{-3/4}, N^{-2/3}\min\{\Delta^{1/9}(E),1\}\}.
\end{equation}
For $E \in \supp{\rho}$ with $\rho(E)>0$, the natural fluctuation scale $\eta_{\mathfrak{f}}(E)$ is defined implicitly through
\begin{equation} \label{eq:etaf_supp}
	\int_{-\eta_\mathfrak{f}(E)}^{\eta_\mathfrak{f}(E)} \rho(E+x) \mathrm{d}x = \frac{1}{N}.
\end{equation}

We can now state our main  result; for convenience we separately formulate the small gap regime
and the small local minimum regime with all their transitionary behavior in the next two theorems.  
\begin{theorem}[Mesoscopic Central Limit Theorem at the Singularities of the Spectrum] \label{th_main}
	Let $g$ be a $C^2_c(\mathbb{R})$ test function, and let $E_0 \in \Sng$ be a singularity in the spectrum. Let $\varepsilon_0>0$ be a  fixed small constant, and let $\eta_0$ satisfy $N^{\varepsilon_0}\eta_{\mathfrak{f}}(E_0) \le \eta_0 \le N^{-\varepsilon_0}$. Assume that $\lim_{N\to\infty} \tfrac{1}{2}\eta_0^{-1}\Delta(\sng) = \alpha \in [0,+\infty]$, and if $\alpha \neq 0$, assume additionally that $\lim_{N\to\infty} \widehat{s}(\sng) = s$. Then the linear eigenvalue statistics on the scale $\eta_0$ around the point $\sng$ satisfy
	\begin{equation} \label{eq:main_clt}
		\Tr \bigl[g\bigl(\eta_0^{-1}(H-E_0)\bigr)\bigr] - N\int\limits_{\mathbb{R}}g\bigl(\eta_0^{-1}(x-E_0)\bigr)\rho(x)\mathrm{d}x \xrightarrow{d} \mathcal{N}\bigl((\tfrac{2}{\boldsymbol{\beta}}-1)\mathrm{Bias}_{\alpha}(g),\tfrac{1}{\boldsymbol{\beta}}\mathrm{Var}_{\alpha}(g)\bigr),
	\end{equation}
	where $\boldsymbol{\beta} = 1,2$ corresponds to real symmetric and complex Hermitian $H$, respectively, and the variance and bias functionals depend on $\alpha$ in the following way:
	\begin{itemize}
		\item[(i)] \textbf{Mesoscopically regular edge}. If $\alpha = \infty$, then
		\begin{equation} \label{main_edge}
			\mathrm{Var}_{\infty}(g) = \frac{1}{4\pi^2}\iint\limits_{x,y>0}\frac{\bigl(g(s\, x)-g(s\, y)\bigr)^2}{(x-y)^2}\mathcal{K}_{\mathrm{edge}}(x,y)\mathrm{d}x\mathrm{d}y,\quad \mathrm{Bias}_{\infty}(g) = \frac{g(0)}{4}, 
		\end{equation}
		where $s = +,-$ correspond to $\sng$ being the left and the right end-points of a support interval, respectively, and for all $x,y>0$, the kernel function $\mathcal{K}_{\mathrm{edge}}(x,y)$ is given by
		\begin{equation} \label{eq:K_edge}
			\mathcal{K}_{\mathrm{edge}}(x,y) := \frac{x+y}{\sqrt{xy}}\mathds{1}_{x,y>0}.
		\end{equation}
		More explicitly, $\mathrm{Var}_{\infty}$ admits the expression
		\begin{equation} \label{eq:norm_edge}
			\mathrm{Var}_{\infty}(g) = \frac{1}{4\pi^2}\norm{g(s\, x^2)}_{\dot {H}^{1/2}}^2.
		\end{equation}
		\item[(ii)] \textbf{Mesoscopic cusp}. If $\alpha = 0$, then
		\begin{equation}\label{main_cups}
			\begin{split}
				\mathrm{Var}_{0}(g) &= \frac{1}{4\pi^2}\iint\limits_{\mathbb{R}^2}\frac{\bigl(g(x^3)-g(y^3)\bigr)^2}{(x-y)^2}\mathcal{K}_{\mathrm{cusp}}(x,y)\mathrm{d}x\mathrm{d}y, \quad
				\mathrm{Bias}_{0}(g) = \frac{g(0)}{3}, 
			\end{split}
		\end{equation}
		where 
		\begin{equation} \label{eq:K_cusp}
			\mathcal{K}_{\mathrm{cusp}}(x,y) := 3\frac{\sign(xy)\bigl(x^{4}+y^{4}+2xy(x^{2}+y^{2})\bigr)}{(x^2+xy+y^2)^2}.
		\end{equation}
		Moreover, $\mathrm{Var}_{0}$ satisfies the relation
		\begin{equation} \label{eq:norm_cusp}
			\mathrm{Var}_{0}(g) \asymp \norm{g(|x|^3)-g(-|x|^3)}_{\dot{H}^{1/2}}^2 + \norm{g(|x|^{3/2})+g(-|x|^{3/2})}_{\dot{H}^{1/2}}^2,
		\end{equation}
		where $\asymp$ denotes equivalence up to universal constants.
		\item[(iii)] \textbf{Edge adjacent to a mesoscopic gap}. If $0 < \alpha < \infty$, then
		\begin{equation} \label{main_gap}
			\begin{split}
				\mathrm{Var}_{\alpha}(g) &= \frac{1}{4\pi^2}\iint\limits_{|x|,|y|>1}\frac{\bigl(g_{\alpha}(x)-g_{\alpha}(y)\bigr)^2}{(x-y)^2}\mathcal{K}_{\mathrm{gap}}(x,y)\mathrm{d}x\mathrm{d}y,\\
				\mathrm{Bias}_{\alpha}(g) &= \frac{g_{\alpha}(1)+g_{\alpha}(-1)}{4} - \frac{\sqrt{3}}{2\pi}\int\limits_{|x|>1} \frac{g_{\alpha}(x)}{(4x^2-1)\sqrt{x^2-1}}\mathrm{d}x,
			\end{split}
		\end{equation}
		where $g_{\alpha}(x) := g\bigl(s\,\alpha(4x^3-3x-1)\bigr)$, the value $s=+,-$ corresponds to the left and right end-points of a support interval, respectively, and the function $\mathcal{K}_{\mathrm{gap}}$ is defined as
		\begin{equation}
			\mathcal{K}_{\mathrm{gap}}(x,y) := 6\frac{(xy-1) (3 (4 x^2-1) (4 y^2-1)+8 (x^2 - y^2)^2) + 2xy (x - y)^2 (8xy+1)}{\sqrt{x^2-1}\sqrt{y^2-1}(4x^2+4xy+4y^2-3)^2}.
		\end{equation}
		Furthermore, $\mathrm{Var}_{\alpha}$ satisfies the relation
		\begin{equation} \label{eq:norm_gap}
			\mathrm{Var}_{\alpha}(g) \asymp \norm{g_{\alpha}^{\mathrm{o}}(\sqrt{x^2+1})}_{\dot{H}^{1/2}}^2 + \norm{g_{\alpha}^{\mathrm{e}}\bigl(\sqrt{\tfrac{1}{2}+\tfrac{1}{2}\sqrt{4x^2+1}}\bigr)\bigr)}_{\dot{H}^{1/2}}^2,
		\end{equation}
		where $g_{\alpha}^{\mathrm{o}}(x) := \tfrac{1}{2}g_{\alpha}(x)-\tfrac{1}{2}g_{\alpha}(-x)$, $g_{\alpha}^{\mathrm{e}}(x) := \tfrac{1}{2}g_{\alpha}(x)+\tfrac{1}{2}g_{\alpha}(-x)$.
	\end{itemize}
\end{theorem}

\begin{theorem}[Mesoscopic Central Limit Theorem at Small Local Minima] \label{th:rho>0} 
	Let $g$ be a $C^2_c(\mathbb{R})$ test function, let $\sng \in \mathcal{M}_{\rho_*}$ be a non-zero local minimum. Let $\varepsilon_0 > 0$ be a fixed small constant, and let $\eta_0$ satisfy $N^{\varepsilon_0}\eta_{\mathfrak{f}}(\sng) \le \eta_0 \le N^{-\varepsilon_0}$.
	Assume that $\lim_{N\to \infty} \frac{-2\psi(\sng)}{\sqrt{27}\pi}\eta_0^{-1}\rho(\sng)^3 = \alpha \in [-\infty, 0]$ (the auxiliary quantity $\psi(\sng)$ is defined in \eqref{psi_def} below), then the linear eigenvalue statistics satisfy the conclusion \eqref{eq:main_clt} of Theorem \ref{th_main} with bias and variance functionals defined for negative values $\alpha$ in the following way: %-2\psi(\sng)\rho_0^3/(\sqrt{27}\pi\eta_0)
	\begin{itemize}
		\item[(iv)] \textbf{Mesoscopic local minimum}. If $\alpha \in (-\infty,0)$, then
		\begin{equation} \label{main_min}
			\begin{split}
				\mathrm{Var}_{\alpha}(g) &= \frac{1}{4\pi^2}\iint\limits_{\mathbb{R}^2}\frac{\bigl(g\bigl(\alpha(4x^3+3x)\bigr)-g\bigl(\alpha(4y^3+3y)\bigr)\bigr)^2}{(x-y)^2}\mathcal{K}_{\mathrm{min}}(x,y)\mathrm{d}x\mathrm{d}y,\\
				\mathrm{Bias}_{\alpha}(g) &=
				\frac{\sqrt{3}}{2\pi}\int\limits_{\mathbb{R}} \frac{g\bigl(\alpha(4x^3+3x)\bigr)}{(4x^2+1)\sqrt{x^2+1}}\mathrm{d}x,
			\end{split}
		\end{equation}
		where
		\begin{equation}
			\mathcal{K}_{\mathrm{min}}(x,y) := 6\frac{(xy+1) (3 (4 x^2+1) (4 y^2+1)+8 (x^2 - y^2)^2)+2xy (x - y)^2 (8 x y-1)}{\sqrt{x^2+1}\sqrt{y^2+1}(4x^2+4xy+4y^2+3)^2}.
		\end{equation}
		Furthermore, $\mathrm{Var}_{\alpha}$ satisfies the relation
		\begin{equation} \label{eq:norm_min}
			\mathrm{Var}_{\alpha}(g) \asymp \norm{g^{\mathrm{o}}(\alpha(4x^3+3x))}_{\dot{H}^{1/2}}^2 + \norm{g^{\mathrm{e}}\biggl(\alpha(1+2\sqrt{4x^2+1})\sqrt{\tfrac{1}{2}\sqrt{4x^2+1}-\tfrac{1}{2}}\biggr)}_{\dot{H}^{1/2}}^2.
		\end{equation}
	\item[(v)] \textbf{Mesoscopically bulk-like local minimum}. If $\alpha = - \infty$, then
		\begin{equation} \label{main_bulk}
			\mathrm{Var}_{-\infty}(g) = \frac{1}{2\pi^2}\iint\limits_{\mathbb{R}^2}\frac{\bigl(g(x)-g(y)\bigr)^2}{(x-y)^2}\mathrm{d}x\mathrm{d}y = \frac{1}{2\pi^2}\norm{g}_{\dot{H}^{1/2}}^2, \quad \mathrm{Bias}_{-\infty}(g) = 0.
		\end{equation}	
	
	\item[(vi)] \textbf{Mesoscopic cusp}. If $\alpha = 0$, the bias and variance functionals coincide with $\mathrm{Bias}_0$ and $\mathrm{Var}_0$ defined in $(ii)$ of Theorem \ref{th_main}.
	\end{itemize}
\end{theorem}

Special cases of Theorems \ref{th_main} and \ref{th:rho>0} were obtained before. In particular, the case $(i)$ of Theorem \ref{th_main} includes all edge points $\sng$ with $\Delta(\sng) \gtrsim 1$, and the limiting bias and variance functionals in \eqref{main_edge} agree with regular edge asymptotics obtained in \cite{AdhikariHuang, Li2021genWigner}.
The case $(v)$ of Theorem \ref{th:rho>0} matches the limiting distribution of mesoscopic linear eigenvalue statistics in the bulk, see \cite{He2017WignerCLT,Li2021deformed, Li2021genWigner, Landon2021Wignertype, R2023bulk} for results on Wigner and related ensembles.
Note that for $\sng$ satisfying $\rho(\sng) \ge \rho_*$, the conclusion \eqref{eq:main_clt} is an immediate consequence of the Central Limit Theorem for mesoscopic linear eigenvalue statistics in the bulk, see Theorem 2.2 of \cite{R2023bulk}.

We further note that the cases $(ii)$ of Theorem \ref{th_main} and $(vi)$ of Theorem \ref{th:rho>0} include all physical cusps, which correspond to the points $\sng$ in the self-consistent spectrum that satisfy $0\le \Delta(\sng)\lesssim N^{-3/4}$, $0\le \rho(\sng) \lesssim N^{-1/4}$. Local eigenvalue statistics at such points were studied in \cite{Cipolloni2018CuspU} and \cite{Erdos2018CuspUF}.

\begin{remark}[Alternative variance formulas]	
	
	To highlight the analogy with \eqref{bulkCLT}, the integral formula for the variance in \eqref{main_gap} can be recast in a way that transfers the $\alpha$-dependence from the argument of the function $g$ in \eqref{main_gap} into the kernel. More precisely, performing the change of variable $x:=\gapx(t)$, $y:=\gapx(q)$, where $\gapx(t)$ is the unique real solution to $4\gapx^3-3\gapx-1 = t$ for $t \in \mathbb{R}\backslash[-2,0]$, yields the following expression
	\begin{equation}
		\begin{split}
			\mathrm{Var}_\alpha(g) &= 	\frac{1}{4\pi^2}\iint\limits_{(\mathbb{R}\backslash[-2\alpha,0])^2}\frac{\bigl(g(s\,x)-g(s\,y)\bigr)^2}{(x-y)^2}\widehat{\mathcal{K}}_{\mathrm{gap}}\bigl(\gapx(\alpha^{-1}x),\gapx(\alpha^{-1}y)\bigr)\mathrm{d}x\mathrm{d}y,\\ \widehat{\mathcal{K}}_{\mathrm{gap}}(x,y) &:= \frac{(4x^2+4xy+4y^2-3)^2}{(4x^2-1)(4y^2-1)}\mathcal{K}_{\mathrm{gap}}(x,y).
		\end{split}
	\end{equation}
	Similar expressions can be obtained for the variance in \eqref{main_cups} and \eqref{main_min} by changing the variable according to the unique real solutions to $4x^3+3x = t$ and $x^3=t$ for $t\in\mathbb{R}$, respectively. 
	
\end{remark}

\begin{remark}[Continuity of the variance]	
	The one-parameter families of functionals $\mathrm{Var}_{\alpha}$ and $\mathrm{Bias}_{\alpha}$ described in $(i)-(iii)$ of Theorem \ref{th_main} and $(iv)-(vi)$ of Theorem \ref{th:rho>0} are continuous in the parameter $\alpha$. The rigorous justification of this fact is contained in the proof of Proposition \ref{prop:main2}. However, for the variance formulas, the continuity can be seen via a simple heuristic argument. Indeed, performing a change of variable $x := (4\alpha)^{-1/3} t$, $y := (4\alpha)^{-1/3} q$ and taking the limit $\alpha \to 0$ in the expression for the variance in \eqref{main_gap} yields the expression in \eqref{main_cups}, while setting $x := 1 + (9\alpha)^{-1} t$, $y := 1 + (9\alpha)^{-1} q$ and taking the limit $\alpha \to +\infty$ yields the expression \eqref{main_edge}. Similar reasoning applies to variance formula in the transitionary regime \eqref{main_min} of a small local minimum and its two extremes in \eqref{main_cups} and \eqref{main_bulk}.   
\end{remark}

\begin{remark}[Additive deformation]
	The analysis laid out in \cite{Alt2020energy} applies to more general versions of \eqref{VDE}, in particular, the \textit{matrix Dyson equation}
	\begin{equation} \label{MDE}
		- M(z)^{-1} = z - A + S[M(z)],\quad A = A^*,\quad \im z \im M(z) > 0.
	\end{equation}
	The equation \eqref{MDE} can be used to study Wigner-type matrices with an additive Hermitian deformation $A$, which, in particular, generalize the deformed Wigner matrices considered in \cite{Li2021deformed}. Moreover, the behavior of the resulting self-consistent density of states near its singular points and local minima is described by the same universal shape functions as the density which arises from \eqref{VDE}. Therefore, the deterministic analysis presented in Sections \ref{sec:variance_lemma_proof} and \ref{sec:bias} applies to deformed Wigner-type matrices, for which the self-consistent resolvent satisfies the equation \eqref{MDE}. However, we do not pursue this generalization,  since in Section \ref{sec:proof_main} we rely on the local law for the resolvent of a Wigner-type matrix at a cusp \cite{Erdos2018CuspUF} that was only proved in the case of a diagonal deformation $A$. For such deformations, our results hold with no alteration to the proof.
\end{remark}

\section{Notations and Preliminaries} \label{sec:prelim}
\subsection{Notations}
For two vectors $\vect{x},\vect{y} \in \mathbb{C}^N$, we use the normalized definition of the scalar product and the corresponding $\ell^2$-norm, namely
\begin{equation}
	\langle\vect{x},\vect{y}\rangle := N^{-1}\sum_{j=1}^N \overline{x_j}y_j, \quad \norm{\vect{x}}_2 := \langle \vect{x}, \vect{x} \rangle^{1/2}.
\end{equation}

Let $\mathbb{H}$ denote the complex upper half-plane, $\mathbb{H} := \{z\in\mathbb{C}:\im z > 0\}$, and let $\overline{\mathbb{H}}$ denote its closure.

For a function $h : \mathbb{H} \to \mathbb{C}$, we define its continuation into the lower-half plane $\mathbb{H}^*$ by complex conjugation, i.e., $h(\overline{z}) := \overline{h(z)}$ for all $z\in\mathbb{H}$. Furthermore, if $h:\mathbb{H} \to \mathbb{C}$ admits a continuous extension to the closure of the complex upper-half plane, $\overline{\mathbb{H}}$, then we distinguish the two limiting values on $\mathbb{R}$ of its continuation to $\mathbb{H}\cup\mathbb{H}^*$ by defining 
\begin{equation} \label{up_dn_limit}
	h(E+\I0) := \lim\limits_{\eta \to +0} h(E+\I\eta), \quad h(E-\I0) := \lim\limits_{\eta \to-0} h(E+\I\eta) = \overline{h(E+\I0)},
\end{equation}
and if the two limiting values coincide, we write $h(E) := 	h(E+\I0) = 	h(E-\I0)$.

We use the following notion of stochastic domination.
\begin{Def} (Definition 2.1 in \cite{Erdos2013LocalLaw})  \label{stochdom_def} %REV 1.6
	Let $\mathcal{X} = \mathcal{X}^{(N)}(u)$ and $\mathcal{Y} = \mathcal{Y}^{(N)}(u)$ be two families of non-negative random variables possibly depending on a parameter $u \in U^{(N)}$. We say that $\mathcal{Y}$ stochastically dominates $\mathcal{X}$ uniformly in $u$ if for any $\varepsilon > 0$ and $D>0$ there exists $N_0(\varepsilon,D)$ such that for any $N \ge N_0(\varepsilon, D)$,
	\begin{equation*}
		\sup\limits_{u\in U^{(N)}}\Prob{\mathcal{X}^{(N)}(u) > N^\varepsilon \mathcal{Y}^{(N)}(u)} < N^{-D}.
	\end{equation*}
	We denote this relation by $\mathcal{X} \prec \mathcal{Y}$ or $\mathcal{X} = \Oprec(\mathcal{Y})$. For a complex-valued random variable $\mathcal{Z}$, the notation $\mathcal{Z} = \Oprec(\mathcal{Y})$ signifies that $|\mathcal{Z}|\prec\mathcal{Y}$.
\end{Def}

For two non-negative quantities $X, Y \in \mathbb{R}$ depending on $N$, we write $X \ll Y$ if there exist $\varepsilon, N_0 > 0$ such that $|X| \le N^{-\varepsilon}|Y|$ for all $N \ge N_0$.
Similarly, we write $X \lesssim Y$ if there exists a constant $C,N_0>0$ such that $|X| \le C|Y|$ for all $N \ge N_0$, and $X \sim Y$ if both $X\lesssim Y$ and $Y \lesssim X$ hold. The value of the constant $C$ can implicitly depend on the constants in Assumptions \eqref{cond_A}-\eqref{cond_C}.

We use $C$ and $c$ to denote constants, the precise value of which is irrelevant and may change from line to line.

\subsection{Properties of the Solution Vector and the Stability Operator}

Our work relies heavily on the various properties of the solution $\m$ of the vector Dyson equation \eqref{VDE}, and the corresponding stability operator. Here we summarize some of the results from \cite{Ajanki2016Univ, Alt2020energy} that we refer to throughout the proof.
\begin{lemma} (Proposition 4.1 in \cite{Ajanki2016Univ} and Lemma 5.4 in \cite{Alt2020energy}) \label{lemma_m}
	Let $\rho(z) := \pi^{-1}\langle \im \m(z) \rangle$ be the harmonic extension of $\rho$ into the upper-half plane $\mathbb{H}$.
	Under the assumptions \eqref{cond_A}-\eqref{cond_C}, the solution vector $\m$ admits a unique extension to $\overline{\mathbb{H}}$, and satisfies the following properties.
	\begin{enumerate}
		\item[(i)] Uniformly for all $z\in \overline{\mathbb{H}}$, and for all $j \in \{1,\dots,N\}$, the entries of $\m$ admit the estimate
		\begin{equation} \label{m_bound}
			|m_j(z)| \sim (1+|z|)^{-1},
		\end{equation}
		and the imaginary parts of the entries of $\m$ are comparable in size, that is,
		\begin{equation} \label{immrho}
		\im m_j(z) \sim \rho(z).
		\end{equation}
		\item[(ii)] For all $z,\zeta \in \overline{\mathbb{H}}$,
		\begin{equation} \label{dm_bound}
			\norm{\m(z) - \m(\zeta)}_\infty \lesssim |z -\zeta|^{1/3}.
		\end{equation}
		\item[(iii)] The functions $z \mapsto \rho(z)^{-1}\im z$ and $z\mapsto \rho(z)^{-1}\im\m(z)$ admit unique uniformly $1/3$-H\"{o}lder continuous extensions to $\{z\in\overline{\mathbb{H}} : |z|\le C \}$. In particular, 
		\begin{equation} \label{rhoeta_to0}
			\lim\limits_{\eta\to+0} \eta\rho(E+\I\eta)^{-1} = 0, \quad E\in \supp\rho.
		\end{equation}
		\item[(iv)] There exists a threshold $\rho_* \sim 1$, such that $\sign(\re\m(z))$ has a unique $1/3$-H\"{o}lder continuous extensions to $\{z\in\overline{\mathbb{H}} : |z|\le C, \rho(z) \le \rho_* \}$. 
	\end{enumerate}	
\end{lemma}

In the sequel, we use $\m$ (and functions of $\m$, such as $\im \m$, $|\m|$, $\m'$ etc) to denote both the vector $(m_j)_{j=1}^N$ and the corresponding multiplication operator $\diag{(m_j)_{j=1}^N}$.

An important ingredient of the proof is the \textit{stability operator}, originally introduced in \cite{Ajanki2019QVE}, defined for two spectral parameters $z,\zeta \in \mathbb{C}\backslash\mathbb{R}$ as
%For two spectral parameters $z,\zeta \in \mathbb{C}\backslash\mathbb{R}$, let $\stab(z,\zeta)$ denote the stability operator 
\begin{equation} \label{stab_def}
	\stab(z,\zeta) := 1-\m(z)\m(\zeta)S.
\end{equation}
We extend the definition \eqref{stab_def} to the real line, adhering to the convention \eqref{up_dn_limit}.

In the following lemma, we collect the necessary properties of the one-body saturated self-energy operator $F(z) := |\m(z)|S|\m(z)|$, introduced in \cite{Ajanki2019QVE}. Its two-body analogue $F(z,\zeta)$ was also studied in \cite{Landon2021Wignertype}. %We collect the main properties of $F$ in the following lemma.
\begin{lemma} (Lemma 4.3 in \cite{Alt2020energy} ) \label{lemma_F} 
	Let $z:=E+\I\eta \in \mathbb{H}$ satisfy $|z| \le C$. The operator $F:=F(z)$ has a unique largest eigenvalue $\opnormtwo{F}$, let $\vv := \vv(z)$ be the corresponding unique $\ell^2$-normalized eigenvector of $F$ with positive entries, then under the assumption \eqref{cond_A},
	\begin{equation} \label{eigF_1}
		1- \opnormtwo{F} = \eta \frac{\langle\vv, |\m|\rangle}{\langle\vv, \im\m/|\m|\rangle} \sim \rho^{-1}\eta.
	\end{equation} 
	Furthermore, eigenvector $\vv$ admits the estimate
	\begin{equation} \label{v_def}
		\vv = \vect{f}\norm{\vect{f}}_2^{-1} + \mathcal{O}(\rho^{-1}\eta),
	\end{equation}
	where the vector $\vect{f}:=\vect{f}(z)$ is defined as
	\begin{equation} \label{f_def}
		\vect{f} := \rho^{-1}\im\m/|\m| \sim 1.
	\end{equation}	
%	\begin{equation} \label{v_bound}
%		\vv \sim 1.
%	\end{equation}
	The operator $F$ has a spectral gap $\vartheta \sim 1$, that is
	\begin{equation}\label{gapF}
		\mathrm{spec}(F/\opnormtwo{F}) \subset [-1+\vartheta,1-\vartheta]\cup\{1\}.
	\end{equation}
	In particular, 
	\begin{equation}
		\opnorminf{(1-F)^{-1}(1-\vv\vv^*)} + \opnormtwo{(1-F)^{-1}(1-\vv\vv^*)} \lesssim 1.
	\end{equation}
\end{lemma}

%We define the following quantities
%\begin{equation} \label{sigmapsi_def}
%	\psi := \langle \vect{p}\vect{f}^2, (1+F)(1-F)^{-1}(1-\vv\vv^*)[\vect{p}\vect{f}^2] \rangle,
%\end{equation}
%where $\vect{p} = \sign(\re\m)$.

For the remainder of this section, we assume that $z$ lies in the domain 
\begin{equation} \label{dom_DEc}
	\D := \{z\in \mathbb{C}: |z-\sng| \le c_0N^{-\varepsilon_0/3}\},
\end{equation}
where $\sng\in \Sng$ and $\varepsilon_0>0$ satisfy the assumptions of Theorem \ref{th_main}, and $c_0$ is a positive $N$-independent constant. Note that for any $\rho_*\sim 1$, the constant $c_0$ can be chosen such that all $z\in\D$ satisfy the condition $\rho(z)^{-1}|\im z| + \rho(z) \le \rho_*$.

We collect the properties of the stability operator $\stab(z,z)$ for $z$ in the vicinity of $\sng$, which were proved in \cite{Alt2020energy}, and serve as the starting point for our perturbative analysis of $\stab(z,\zeta)$.
\begin{lemma} (Lemma 5.1 in \cite{Alt2020energy}) \label{lemma_stab} Let $\stab(z,z)$ be the stability operator defined in \eqref{stab_def}, then there exists $\epsilon \sim 1$, such that for all $z := E+ \I\eta \in\D\cap\mathbb{H}$,
	\begin{equation} \label{stab0_gap}
		\opnormtwo{\bigl(w - \stab(z,z)\bigr)^{-1}} + \opnorminf{\bigl(w - \stab(z,z)\bigr)^{-1}} \lesssim 1,
	\end{equation}
	for all $w \in \mathbb{C}$ such that $|w| \ge \epsilon$ and $|1-w| \le 1 - 2\epsilon$. Furthermore, $\stab$ has a single simple eigenvalue $\eigB(z,z)$ with $|\eigB(z)| \le \epsilon$, which is isolated and  satisfies
	\begin{equation} \label{beta_est}
		|\eigB(z,z)| \sim \rho(z)^{-1}\eta + \rho(z)\bigl(\rho(z) + \sigma(z)\bigr),
	\end{equation}
	where quantity $\sigma(z)$ is defined by\
	\begin{equation} \label{sigma_def}
		\sigma(z) := \langle\vect{p}(z)\vect{f}(z)^2,\vect{f}(z) \rangle, \quad \vect{p}(z) := \sign\bigl(\re\m(z)\bigr).
	\end{equation}

	Let $\Pi(z,z)$ be the eigenprojector corresponding to the eigenvalue $\eigB(z,z)$ defined by
	\begin{equation}
		\Pi(z,z) := \frac{1}{2\pi\I}\oint\limits_{|w|=\epsilon} \bigl(w-\stab(z,z)\bigr)^{-1}\mathrm{d}w,
	\end{equation}
	then for $Q(z,z):=1-\Pi(z,z)$, we have
	\begin{equation} \label{eq:invstab0Q}
		\opnorminf{\stab(z,z)^{-1}Q(z,z)} + \opnormtwo{\stab(z,z)^{-1}Q(z,z)} \lesssim 1.
	\end{equation}

	Let $\vv := \vv(z)$ be the eigenvector of $F := F(z)$ defined in Lemma \ref{lemma_F}, and let $\vect{f}:=\vect{f}(z)$ be the vector defined in \eqref{f_def}, then the definitions
	\begin{equation} \label{b_def}
		\vect{b}(z):=\Pi(z,z)[|\m(z)|\vv(z)]\langle\vv(z),\vect{f}(z)\rangle\quad\text{and}\quad\vect{b}^{\ell}(z) := \Pi(z,z)^*[|\m(z)|^{-1}\vv(z)] \langle \vv(z),\vect{f}(z)\rangle
	\end{equation} 
	yield right and left eigenvectors of $\stab(z,z)$ corresponding to $\eigB(z,z)$ which satisfy
	\begin{equation} \label{b_vect}
		\begin{split}
			\vect{b}(z) &= \langle\vv,\vect{f}\rangle |\m|\vv  + 2\I\rho |\m| (1-F)^{-1}(1-\vv\vv^*)[\vect{p}\vect{f}^2] + \mathcal{O}(\rho^2),\\
			\vect{b}^\ell(z) &= \langle\vv,\vect{f}\rangle |\m|^{-1}\vv  - 2\I\rho|\m|^{-1}(1-F)^{-1}F(1-\vv\vv^*)[\vect{p}\vect{f}^2] + \mathcal{O}(\rho^2),
		\end{split}
	\end{equation}
	where $\rho, \m, \vv,\vect{f}, \vect{p}$, and $F$ are evaluated at $z$.
\end{lemma}
In particular, the analysis of the stability operator $\stab(z,z)$ in \cite{Alt2020energy} yields the following estimate on the derivative of the solution vector $\m(z)$,
\begin{equation} \label{eq:m'_bound}
	\norm{\m'(z)}_{\infty} \lesssim 1+|\eigB(z,z)|^{-1}.
\end{equation}

We denote the other end of the adjacent gap in the support corresponding to $\sng$ by $\widehat{\sng}$. More explicitly, $\widehat{\sng} := \sng+\Delta(\sng)$, if $\sng$ is a right end-point of a support interval, and $\widehat{\sng} := \sng-\Delta(\sng)$ if $\sng$ is a left-end point. Note that $\widehat{\sng}=\sng$ if $\sng$ is a cusp.
\begin{prop} (Proposition 3.2 of \cite{Erdos2018CuspUF}) \label{prop_scaling}
	Define the quantity $\dis\equiv\dis(z) := \min\{ |z-E_0|, |z-\widehat{E_0}|\}$, and let $\Delta_0 := \min\{\Delta(\sng), 1\}$, then for all $z\in \D\cap\mathbb{H}$, the harmonic extension $\rho := \rho(z)$ of the self-consistent density of states satisfies
	\begin{equation}
		\begin{split} \label{eq:rho_comp}
			\rho &\sim \dis^{1/2}(\Delta_0 + \dis)^{-1/6}, \quad \re z \in \supp{\rho} \\
			\rho &\sim |\eta| \dis^{-1/2}(\Delta_0 + \dis)^{-1/6}, \quad \re z \notin \supp{\rho},
		\end{split}
	\end{equation}
	where $\eta := \im z$. 	The eigenvalue $\eigB:=\eigB(z,z)$ of $\stab(z,z)$ with smallest modulus is comparable with
	\begin{equation} \label{eq:eig_comp}
		|\eigB| \sim \rho^{-1}\eta +\rho(\rho + |\sigma|) \sim \dis^{1/2}(\Delta_0 + \dis)^{1/6},
	\end{equation}
	where $\sigma := \sigma(z)$.	Moreover, there is a positive threshold $c^*\sim 1$, such that for all $z\in\D\cap\mathbb{H}$,
	\begin{equation} \label{eq:rho+sigma_comp}
		\begin{split}
			\rho + |\sigma| &\sim (\Delta_0+\dis)^{1/3}, \quad \re z \in \supp{\rho} ~\text{or}~ \dis \le c^*\Delta_0,\\
			\rho + |\sigma| &\lesssim (\Delta_0+\dis)^{1/3}, \quad \re z\notin \supp{\rho} ~\text{and}~ \dis \ge c^*\Delta_0.
		\end{split}
	\end{equation}
\end{prop}

%\begin{equation} \label{sng_objects}
%	\m_0 := \m(\sng),\quad\stab_0 := \stab(\sng,\sng), \quad \Pi_0[\cdot]:= \langle \vect{b}^\ell, \vect{b}\rangle^{-1} \langle\vect{b}^\ell, \cdot\rangle \vect{b}, \quad Q_0 := 1 - \Pi_0,
%\end{equation}
%where $\vect{b}:= \vect{b}(\sng), \vect{b}^{\ell}:=\vect{b}^{\ell}(\sng)$ are defined in \eqref{b_vect}. In view of \eqref{beta_est} and \eqref{rhoeta_to0}, $\eigB_0 := \eigB(E_0,E_0) = 0$.

The solution vector $\m$ satisfies the following perturbation formula.
\begin{prop} (Proposition 6.1 and (6.20), Proposition 10.1 in \cite{Alt2020energy}) \label{prop_mpert}
	There exists a threshold $c_0\sim 1$, such that for all $z,\zeta \in \D\cap\mathbb{H}$,
	\begin{equation} \label{eq:dM_dir}
		\m(\zeta) - \m(z) = \dM(z,\zeta)\vect{b}(z) + \dM(z,\zeta)^2\vect{r}(z) + \mathcal{O}(|\zeta-z|), 
	\end{equation}
	where $\vect{b}(z)$ is defined in \eqref{b_def}, the function $\dM(z,\zeta)$ and the vector $\vect{r}(z)$ are given by
	\begin{equation} \label{eq:dM_def}
		\dM(z,\zeta) := \frac{\langle\vect{b}^\ell(z), \m(\zeta) -\m(z) \rangle}{\langle\vect{b}^\ell(z), \vect{b}(z) \rangle}, \quad \vect{r}(z) := (1-\eigB(z,z))\frac{Q(z,z)}{\stab(z,z)}\biggl[\frac{\vect{b}^2(z)}{\m(z)}\biggr].
	\end{equation}
	For all $z,\zeta \in \D\cap\mathbb{H}$, the function $\dM(z,\zeta)$ satisfies the cubic equation
%	\begin{equation}
%		|\dM(z,\zeta)| \lesssim |z-\zeta|^{1/3}.
%	\end{equation}
%	Moreover, for a fixed $z\in \D$ and for all $w\in\mathbb{R}$ such that $z+w\in\D$, the function $\dM(w):=\dM(z,z+w)$ satisfies the cubic equation 
	\begin{equation} \label{dm_cubic}
		\mu_3\dM(z,\zeta)^3 + \mu_2\dM(z,\zeta)^2 + \mu_1\dM(z,\zeta) + (\zeta-z)\mu_0 = 0,
	\end{equation}
	with complex coefficients $\mu_3:=\mu_3(z),\mu_2:=\mu_2(z),\mu_1:=\mu_1(z)$, and $\mu_0:=\mu_0(z,\zeta)$ in \eqref{dm_cubic} admit the estimates 
	\begin{equation} \label{eq:cubic_mus}
		\begin{split}
			\mu_3 &= \psi + \mathcal{O}(\rho + \rho^{-1}\eta ), \quad \mu_2 = \sigma + \I\rho \biggl(3\psi + \frac{\sigma^2}{\langle\vect{f}^2\rangle}\biggr) + \mathcal{O}(\rho^2+\rho^{-1}\eta),\\
			\mu_1 &= 2\I\rho\sigma - 2\rho^2\biggl(\psi + \frac{\sigma^2}{\langle \vect{f}^2 \rangle}\biggr) + \mathcal{O}(\rho^3 + \eta + \rho^{-2}\eta^2), \quad
			\mu_0 = \pi+\mathcal{O}(\rho+\rho^{-1}\eta+|\dM(z,\zeta)|)
		\end{split} 
	\end{equation}
	where $\eta := \im z$,  $\rho:=\rho(z)$, $\sigma := \sigma(z)$ is defined in \eqref{sigma_def}, $\vect{f} := \vect{f}(z)$ is defined in \eqref{f_def}, and the quantity $\psi := \psi(z)$ is given by
	\begin{equation} \label{psi_def}
		\psi(z) := \biggl\langle \vect{p}(z)\vect{f}(z)^2, \frac{1+F(z)}{1-F(z)}\bigl(1-\vv(z)\vv(z)^*\bigr)[\vect{p}(z)\vect{f}(z)^2] \biggr\rangle.
	\end{equation}
	The function $\dM(z,\zeta)$ is bounded by
	\begin{equation} \label{dM_bound}
		|\dM(z,\zeta)| \lesssim \min\biggl\{|\zeta-z|^{1/3}, %\frac{|\zeta-z|^{1/2}}{(\rho(z)+|\sigma(z)|)^{1/2}}, \frac{|\zeta-z|}{\rho(z)(\rho(z)+|\sigma(z)|)}
		\frac{|\zeta-z|}{\rho(z)^2}
		\biggr\}.
	\end{equation}
	Moreover, there exists a threshold $c_*$, such that for all $w \in [-c_*,c_*]$, the estimate \eqref{dM_bound} can be improved to 
	\begin{equation} \label{eq:better_dM_bound}
		|\dM(z,z+w)| \lesssim M(z,w) := \min\biggl\{|w|^{1/3}, \frac{|w|^{1/2}}{(\Delta_0+\dis(z))^{1/6}}, \frac{|w|}{\dis(z)^{1/2}(\Delta_0+\dis(z))^{1/3}}
		\biggr\}.
	\end{equation}
\end{prop} 

\begin{lemma} (Lemmas 5.5, 7.15, and 7.16 in \cite{Alt2020energy})
	For all singular points $E\in \Sng\cap\D$, the quantity $\sigma(E)$ satisfies the following estimate 
	\begin{equation} \label{eq:sigma_comp}
		|\sigma(E)| \sim \min\{1, \Delta(E)^{1/3}\}.
	\end{equation}
	Uniformly for all $z\in \D$, we have
	\begin{equation} \label{eq:psi_sim_1}
		\psi(z) + \sigma(z)^2 \sim 1.
	\end{equation}
	Moreover, there exists a threshold $\Delta_*$, such that if $\Delta(E) \le \Delta_*$, then
	\begin{equation} \label{def:delta_hat}
		\Delta(E) = \widehat{\Delta}(E)\bigl(1+\mathcal{O}(|\sigma(E)|)\bigr), \quad \widehat{\Delta}(E) := \frac{4|\sigma(E)|^3}{27\pi \psi(E)^2},
	\end{equation}
	where $\psi(E)$ is defined in \eqref{psi_def}.
\end{lemma}

\begin{lemma} (Lemma C.1 in \cite{Alt2020energy}) \label{lemma_ncpt}
	Let $z_0,\zeta_0 $ be two spectral parameters in $\D$, and let $\eigB_0 := \eigB(z_0,\zeta_0)$.
	Let $\Pi_0$ be the eigenprojector corresponding to eigenvalue $\eigB_0$ of $\stab_0 := \stab(z_0,\zeta_0)$, and let $Q_0 := 1-\Pi_0$.
	Then for all $z,\zeta \in \D\backslash\mathbb{R}$, the eigenprojector $\Pi:=\Pi(z,\zeta)$ defined in \eqref{eq:rankPi} admits the estimate
	\begin{equation} \label{Pi_pert}
		\Pi = \Pi_0 + (1-\eigB_0)\frac{Q_0}{\stab_0-\eigB_0}\frac{\diff\m}{\m(z_0)\m(\zeta_0)}\Pi_0 + \Pi_0\frac{\diff\m}{\m(z_0)\m(\zeta_0)}\frac{1-\stab_0}{\stab_0-\eigB_0}Q_0 + \mathcal{O}(\norm{\diff\m}_\infty^2),
	\end{equation}
	\begin{equation} \label{Q/B_pert}
		\stab(z,\zeta)^{-1}Q(z,\zeta) = \stab_0^{-1}Q_0 + \mathcal{O}(\norm{\diff\m}_\infty),
	\end{equation}
	where $\diff\m := \m(z)\m(\zeta) - \m(z_0)\m(\zeta_0)$, and all $\diff\m/(\m(z_0)\m(\zeta_0))$ are interpreted as diagonal matrices.
	
	The eigenvalue $\eigB:=\eigB(z,\zeta)$ of $\stab(z,\zeta)$ with the smallest modulus satisfies 
	\begin{equation} \label{eig_pert}
		\eigB = \eigB_0 - \Tr\bigl[\frac{(1-\eigB_0)\diff\m}{\m(z_0)\m(\zeta_0)}\Pi_0\bigr] - \Tr\bigl[\frac{(1-\eigB_0)\diff\m}{\m(z_0)\m(\zeta_0)}\frac{1-\stab_0}{\stab_0-\eigB_0}Q_0\frac{\diff\m}{\m(z_0)\m(\zeta_0)}\Pi_0\bigr]
		%			\\ &- \frac{Q_0}{\stab_0}\frac{\diff\m}{\m_0^2}\Pi_0\frac{\diff\m}{\m_0^2}\Pi_0 + \Pi_0\frac{\diff\m}{\m_0^2}\Pi_0\frac{\diff\m}{\m_0^2}\frac{\stab_0-1}{\stab_0}Q_0
		+ \mathcal{O}(\norm{\diff\m}_\infty^3).
	\end{equation}
\end{lemma}
\begin{proof}
	Let $\stab := \stab(z,\zeta)$.
	Observe that by definition \eqref{stab_def},  $\stab-\stab_0 = (\m(z_0)\m(\zeta_0))^{-1}\diff\m(\stab_0-1)$, and hence $\opnorminf{\stab-\stab_0} \lesssim \norm{\diff\m}_\infty$. 
	
	Estimates \eqref{Pi_pert} and \eqref{eig_pert} are immediate consequences of the non-Hermitian perturbation theory of Appendix C in \cite{Alt2020energy}.
	The estimate \eqref{Q/B_pert} follows immediately from the contour integral representation
	\begin{equation}
		\stab^{-1}Q = \frac{1}{2\pi \I} \oint_{\Gamma} \stab^{-1}(w-\stab)^{-1}\mathrm{d}w = \frac{1}{2\pi\I}\oint_{\Gamma} w^{-1}(w-\stab)^{-1}\mathrm{d}w,
	\end{equation}
	where $\Gamma$ is the circle $\{w\in \mathbb{C} : |w-1| = 1 - 2\epsilon\}$, with $\epsilon$ defined in Lemma \ref{lemma_stab}.
\end{proof}

\subsection{Local Law for the Resolvent}
For a small constant $\kappa>0$, let  $\mathcal{D}_{\kappa}$ denote the spectral domain 
\begin{equation}
	\mathcal{D}_{\kappa} := \{z\in\mathbb{C} : \dist(z,\supp{\rho}) \in [N^\kappa \eta_\mathfrak{f}(\re z), N^{100}] \},
\end{equation}
where the natural fluctuation scale $\eta_\mathfrak{f}(E)$ is defined in \eqref{etaf}.

We now state the optimal \textit{averaged} and \textit{isotropic} local laws for the resolvent of a Wigner-type matrix.

\begin{theorem} (Theorem 2.5 in \cite{Erdos2018CuspUF}) \label{th:local_law}
	For any $\kappa>0$ and for all $z := E + \I\eta \in \mathcal{D}_\kappa$, the resolvent $G(z)$ satisfies
	\begin{equation}\label{eq:local_law_dir}
		G(z) = \diag{\m(z)} + \theta(z)\diag{\vect{b}(z)} + D(z),
	\end{equation}
	where $\vect{b}(z)$ is defined in \eqref{b_def}, $\theta(z)% := \langle\vect{b}^\ell(z),\vect{b}(z)\rangle^{-1}\Tr\bigl[\diag{\vect{b}^\ell(z)}\bigl(G(z)-\diag{\m(z)}\bigr)\bigr]
	$ is a random variable that admits the estimate
	\begin{equation} \label{eq:theta_est}
		|\theta(z)| \prec \Theta(z),
	\end{equation}
	and $D(z)$ is a random matrix that satisfies
	\begin{equation} \label{eq:D_est}
		N^{-1} \bigl| \Tr[ \diag{\vect{w}} D(z)] \bigr| \prec \Psi(z)^2, \quad \bigl|\langle \vect{x}, D(z)\vect{y} \rangle \bigr| \prec \Psi(z),
	\end{equation}
	for all deterministic vectors $\vect{x},\vect{y},\vect{w} \in \mathbb{C}^N$ with $\norm{\vect{w}}_\infty, \norm{\vect{x}}_2,\norm{\vect{y}}_2 \lesssim 1$.
	The deterministic control parameters $\Theta(z)$ and $\Psi(z)$ are defined as 
	\begin{equation} \label{eq:control_parameters}
		\Psi(z) := \sqrt{\frac{\rho(z)}{N |\eta|}}, \quad \Theta(z) := \frac{1}{N \dist(z, \supp{\rho})}.
	\end{equation}
\end{theorem}

\begin{remark}
	The restatement of Theorem 2.5 in \cite{Erdos2018CuspUF} presented in Theorem \ref{th:local_law} above follows from the combination of Proposition 3.4, Lemma 3.8 in \cite{Erdos2018CuspUF}. In particular, estimates \eqref{eq:D_est} are an immediate consequence of (3.8a), (3.14), and (3.31). Note that the matrix $D$ in the notation of the present paper corresponds to $-\mathcal{B}^{-1}\mathcal{Q}[MD] + \Theta^2\mathcal{B}^{-1}\mathcal{Q}[M\mathcal{S}[V_r]V_r] + E$ in the notation of \cite{Erdos2018CuspUF}. We have additionally used the fact that for all $z\in \mathcal{D}_{\kappa}$, $\Psi(z) \gtrsim \Theta(z)$, which follows from the definition of the natural fluctuation scale \eqref{etaf} and \eqref{eq:etaf_supp}, the upper bound in \eqref{m_bound}, and the comparison relations for the harmonic extension of the density $\rho(z)$ \eqref{eq:rho_comp}.
\end{remark}

\section{Proof of the Main Result} \label{sec:proof_main}
To streamline the presentation, we only prove the more challenging Theorem \ref{th_main} in full detail. The specific alterations required for completing the proof of Theorem \ref{th:rho>0} are outlined in Section \ref{sect:smallrho} below.
\begin{proof}[Proof of Theorem \ref{th_main}]
	We deduce Theorem \ref{th_main} from the following two propositions.
	\begin{prop} \label{prop:main1}
		Let $\eta_0$, $\varepsilon_0 > 0$ and $E_0$ satisfy the assumptions of Theorem \ref{th_main}.
		Define the scaled test function $f$ to be
		\begin{equation} \label{scaled_f_def}
			f(x) := g\bigl(\eta_0^{-1}(x-E_0)\bigr),
		\end{equation}
		and	let $\phi(\lambda)$ be the characteristic function of $\Tr f(H) -\E{\Tr f(H)}$,
		\begin{equation} \label{phi_def}
			\phi(\lambda) := \E{\exp\{\I\lambda\left(\Tr f(H) - \E{\Tr f(H)}\right)  \}},\quad \lambda \in\mathbb{R}.
		\end{equation}
		Then its derivative $\phi'(\lambda)$ satisfies the following equation,
		\begin{equation} \label{main1_goal}
			\phi'(\lambda) = -\lambda\phi(\lambda)V(f) + \Oprec\bigl(N^{-1/2}\eta_0^{-3/4}(\Delta_0+\eta_0)^{1/12}(1+|\lambda|^4)+(1+|\lambda|)N^{-\alp}\bigr), \quad \lambda\in\mathbb{R}.
		\end{equation}
		%provided $c \le V(f) \le C$ for some positive $N$-independent constants $c$ and $C$.
		
		Here the variance $V(f)$ for a scaled test function $f$ is defined by
		\begin{equation} \label{variance_V}
			V(f) := \frac{1}{\pi^2}\int\limits_{\dom}\int\limits_{\dom'} \frac{\partial \other f(\zeta)}{\partial \bar \zeta}\frac{\partial \other f(z)}{\partial \bar z} \mathcal{K}(z,\zeta)\mathrm{d}\bar\zeta\mathrm{d}\zeta \mathrm{d}\bar z\mathrm{d}z,
		\end{equation}
		where for $z,\zeta \in \mathbb{C}/\mathbb{R}$ the kernel $\mathcal{K}(z,\zeta)$ is defined by
		\begin{equation} \label{kernel_K}
			\begin{split}
				\mathcal{K}(z,\zeta) :=& ~\frac{2}{\boldsymbol{\beta}}\frac{\partial }{\partial\zeta}\Tr\biggl[\frac{\m'(z)}{\m(z)}\bigl(1-S\m(z)\m(\zeta)\bigr)^{-1}\biggr]
				\\
				&+\biggl(1-\frac{2}{\boldsymbol{\beta}} \biggr) \Tr\left[S\m'(z)\m'(\zeta)\right] + \frac{1}{2}\frac{\partial^2}{\partial z\partial\zeta}\left\langle\overline{\m(z)\m(\zeta)}, \Cmlnt^{(4)}\m(z)\m(\zeta) \right\rangle,
			\end{split}
		\end{equation}
		with the matrix $\Cmlnt^{(4)}$ defined by 
		\begin{equation}
			\Cmlnt^{(4)}_{jk} := c^{(4)}(\re H_{jk}) + c^{(4)}(\im H_{jk}),
		\end{equation}
		where $c^{(4)}(h) := \partial^4_{t}\log \Expv[\exp\{\I t h\}]\rvert_{t=0}$ denotes the fourth cumulant of a real random variable.
		The integration domains $\dom, \dom'$ in \eqref{variance_V} are defined as
		\begin{equation} \label{Omega_0_defs}
			\dom := \{z\in \mathbb{C} : |\im{z}| > N^{-\alp}\eta_0 \},\quad	\dom' := \{z\in \mathbb{C} : |\im{z}| > 2N^{-\alp}\eta_0 \},
		\end{equation}
		and $\other{f}$ is the quasi-analytic extension of $f$, defined by
		\begin{equation} \label{QA_f}
			\other{f}(x+\I\eta) = \chi(\eta) \left( f(x) + \I\eta f'(x) \right),
		\end{equation}
		where $\chi :\mathbb{R} \to [0,1]$ is an even $C_c^\infty(\mathbb{R})$ function supported on  $[-c_1,c_1]$, satisfying $\chi(\eta) = 1$ for $|\eta|<c_1/2$, where the constant $1\ge c_1 \sim 1$ is chosen such that $\supp{\other{f}} \subset \D$.
%		Let $\eta_0$, $\varepsilon_0 > 0$ and $E_0$ satisfy the assumptions of Theorem \ref{th_main}, let $f$ be a scaled test function defined in \eqref{scaled_f_def}, 
%		and	let $\phi(\lambda)$ be the characteristic function of $\Tr f(H) -\E{\Tr f(H)}$,
%		\begin{equation} \label{phi_def}
%			\phi(\lambda) := \E{\exp\{\I\lambda\left(\Tr f(H) - \E{\Tr f(H)}\right)  \}},\quad \lambda \in\mathbb{R}.
%		\end{equation}
%		Then its derivative $\phi'(\lambda)$ satisfies the following equation,
%		\begin{equation} \label{main1_goal}
%			\phi'(\lambda) = -\lambda\phi(\lambda)V(f) + \Oprec\bigl(?\bigr), \quad \lambda\in\mathbb{R},
%		\end{equation}
%		provided $c \le V(f) \le C$ for some positive $N$-independent constants $c$ and $C$.
%		
%		Here the variance $V(f)$ for a scaled test function $f$ is defined by
%		\begin{equation} \label{V_def}
%			V(f) := \frac{1}{4\pi^2}\iint\limits_{[\sng-\eps,\sng+\eps]^2}(f(x)-f(y))^2 \other{\mathcal{K}}(x,y)\mathrm{d}x\mathrm{d}y,
%		\end{equation}
%		where for $x,y \in \mathbb{R}$ the kernel $\other{\mathcal{K}}(x,y)$ is defined by
%		\begin{equation} \label{kernel_K}
%			\begin{split}
%				\other{\mathcal{K}}(x,y) &:= \re\bigl[\mathcal{K}(x+\I 0, y+\I 0)-\mathcal{K}(x-\I0, y+\I0)\bigr],\\
%				\mathcal{K}(z,\zeta) &:= \frac{2}{\beta}\frac{\partial }{\partial\zeta}\Tr\biggl[\frac{\m'(z)}{\m(z)}\stab(z,\zeta)^{-1}\biggr] +\biggl(1-\frac{2}{\beta} \biggr) \Tr\left[1 - \stab(z,\zeta)\right].
%			\end{split}
%		\end{equation}
	\end{prop}

	\begin{prop} \label{prop:main2}
		Let $\sng$, $\eta_0$ satisfy the conditions of Theorem \ref{th_main}, and let $f(x) := g(\eta_0^{-1}(x-\sng))$ be the scaled test function as in \eqref{scaled_f_def}. Then the variance $V(f)$ defined in \eqref{variance_V}, and the expected value of the linear eigenvalue statistics satisfy
		\begin{equation} \label{eq:main2}
			V(f) = \mathrm{Var}_{\widehat{\alpha}}(g) + \mathcal{O}(N^{-\varepsilon_0/9} + \eta_0^{1/9}),
		\end{equation}
		\begin{equation} \label{eq:main2_bias}
			\Expv\bigl[\Tr f(H)\bigr] - N\int_\mathbb{R} f(x)\rho(x)\mathrm{d}x = \mathrm{Bias}_{\widehat{\alpha}}(g) + \mathcal{O}(N^{-\varepsilon_0/6}),
		\end{equation}
		where $\mathrm{Var}_{\widehat{\alpha}}$ and $\mathrm{Bias}_{\widehat{\alpha}}$ are the variance and bias functionals given by $(i)-(iii)$ of Theorem \ref{th_main}, with $\widehat{\alpha} := \tfrac{\widehat{\Delta}}{2\eta_0}$ if $\Delta(\sng) < \Delta_*$, and $\widehat{\alpha}:=\infty$ if $\Delta(\sng) \ge \Delta_*$.
		
		Furthermore, the variance and bias functionals defined in Theorem \ref{th_main} are continuous in $\alpha$, i.e.,
		\begin{equation} \label{eq:Var_converge}
			\lim\limits_{N\to\infty}\mathrm{Var}_{\widehat{\alpha}}(g) = \mathrm{Var}_{\alpha}(g), \quad \text{ given that } \lim_{N\to\infty}\widehat{\alpha} = \alpha.
		\end{equation}
		\begin{equation} \label{eq:Bias_converge}
			\lim\limits_{N\to\infty}\mathrm{Bias}_{\widehat{\alpha}}(g) = \mathrm{Bias}_{\alpha}(g), \quad \text{ given that } \lim_{N\to\infty}\widehat{\alpha} = \alpha.
		\end{equation}
	\end{prop}
We defer the proofs of Propositions \ref{prop:main1} and \ref{prop:main2} to Section \ref{sec:main1proof}, and Sections \ref{sec:main2proof}, \ref{sec:bias}, respectively.

Integrating the estimate \eqref{main1_goal} of Proposition \ref{prop:main1}, and applying L\'evy's continuity theorem, we prove the universal Gaussian fluctuations of \eqref{LES}. The expressions for the limiting bias and variance functionals follow by Proposition \ref{prop:main2} and the assumption that $\widehat{\alpha} \to \alpha$ as $N \to \infty$. This concludes the proof of Theorem \ref{th_main}.
\end{proof}

\begin{remark}
	For the sake of presentation, we only consider real symmetric matrices ($\boldsymbol{\beta} = 1$) in the proof of Propositions \ref{prop:main1} and \ref{prop:main2}, as the complex Hermitian ($\boldsymbol{\beta} = 2$) case differs only in using the complex analogue of cumulant expansion formula (see Section II in \cite{Khorunzhy1999}, Lemma 3.1 in \cite{He2017WignerCLT}, Lemma A.1 in \cite{Li2021deformed} for the statement of the cumulant expansion formula in the real and complex cases).
\end{remark}

\subsection{Characteristic Function Method. Proof of Proposition \ref{prop:main1}} \label{sec:main1proof}
First, we collect the results necessary to establish Proposition \ref{prop:main1}.  
\begin{lemma} (c.f. Lemma 5.4 in \cite{R2023bulk}) \label{lemma:standard_estimates}
	Let $\phi(\lambda)$ be the characteristic function %defined in \eqref{phi_def}, then, under the conditions of Theorem \ref{main_Th}, the following estimates hold
	of $\{1-\Expv\}[\Tr f(H)]$, then the following estimates hold.
	\begin{equation} \label{phi'_Omega_int}
		\begin{split}
			\phi(\lambda) &= \E{\other{e}(\lambda)} + \Oprec\bigl(N^{-\alp}\bigr),\\
			\phi'(\lambda) &= \frac{\I}{\pi}\int\limits_{\dom} \frac{\partial\other{f}}{\partial\bar{z}}
			\E{\other e(\lambda) \left\{1-\Expv\right\}\left[ \Tr G(z)\right] }\mathrm{d}\bar z \mathrm{d}z + \Oprec\bigl(|\lambda|N^{-\alp}\bigr),
		\end{split}
	\end{equation}
	where the integration domain $\dom$ is defined in \eqref{Omega_0_defs}, and the function $\other{e}(\lambda)$ is given by
	\begin{equation} \label{tilde e}
		\other{e}(\lambda) := \exp\biggl\{\frac{\I\lambda}{\pi} \int\limits_{\dom'}\frac{\partial \other{f}}{\partial \bar{z}}\{1-\Expv\}\left[\Tr G(z)\right] \mathrm{d}\bar z \mathrm{d}z \biggr\}.
	\end{equation}
	Furthermore, there exists a positive constant $c_0\sim 1$ such that for all $z\in \D\cap\mathcal{D}_\kappa$,
	\begin{subequations} \label{1/mj(z)}
		\begin{align}
			\Expv\bigl[\other e(\lambda)\{1-\Expv\}[\Tr G(z)]\bigr] =&~\Expv\bigl[\other e(\lambda)\left\{1-\Expv\right\}[\mathcal{P}(z)]\bigr] \label{eq:new_term} \\ 
			&+\Expv\bigl[\other e(\lambda)\left\{1-\Expv\right\}[\mathcal{T}(z,z)]\bigr]
			+\frac{2\I\lambda}{\pi} \Expv\biggl[\other e(\lambda) \int\limits_{\dom'} \frac{\partial\other{f}}{\partial\bar\zeta} \frac{\partial}{\partial\zeta}\mathcal{T}(z,\zeta) \mathrm{d}\bar \zeta \mathrm{d}\zeta \biggr] \label{eq:curlyT_terms} \\
			&+\frac{\I\lambda}{\pi} \E{\other e(\lambda)} \int\limits_{\dom'} \frac{\partial\other{f}}{\partial\bar\zeta} \Tr\left[S \m'(z)\m'(\zeta)\right] \mathrm{d}\bar \zeta \mathrm{d}\zeta\\
			&+\frac{\I\lambda}{2\pi}\E{\other e(\lambda) }\int\limits_{\dom'}\frac{\partial\other{f}}{\partial\bar\zeta}\frac{\partial^2 }{\partial  z\partial \zeta}\bigl\langle \overline{\m(z)\m(\zeta)}, \Cmlnt^{(4)}\m(z)\m(\zeta)\bigr\rangle \mathrm{d}\bar \zeta \mathrm{d}\zeta\\
			&+\mathcal{O}_{\prec}\bigl(|\eigB(z,z)|^{-1}(1+|\lambda|^4)\Psi(z)\eta_0^{-1/3}
			%\eta_0^{-1/4}(\Delta +\eta_0)^{-1/12}
			\bigr), 
		\end{align}
	\end{subequations}
	where the random functions $\mathcal{P}(z)$ and  $\mathcal{T}(z,\zeta)$ are defined as
	\begin{equation} \label{eq:Curly_P}
		\mathcal{P}(z) := \sum\limits_{a,b}\bigl[\frac{\m'(z)}{\m(z)}S\bigr]_{ja}\bigl(G_{aa}(z)-m_a(z)\bigr)\bigl(G_{bb}(z)-m_b(z)\bigl)
	\end{equation}
	\begin{equation} \label{Curly T def}
		\mathcal{T}(z,\zeta) := \sum_{b}\sum_{a\neq b} \bigl[\frac{\m'(z)}{\m(z)}S\bigr]_{ab} G_{ba}(z)G_{ab}(\zeta). 
	\end{equation} 
\end{lemma}

Note the appearance of the term on the right-hand side of \eqref{eq:new_term}. In prior works this term was bounded by $\mathcal{O}_\prec(N\Psi(z)\Theta(z))$. However, such error is not admissible in the almost-cusp regime under consideration. Therefore we estimate the right-hand side of \eqref{eq:new_term} separately by making use of the local law in the form \eqref{eq:local_law_dir} in the following lemma.
\begin{lemma} \label{lemma:new_term_est}
	Let $\mathcal{P}(z)$ be the random function defined in \eqref{eq:Curly_P}.
	There exists a positive threshold $c_0\sim 1$, such that for all $z\in \D\cap\mathcal{D}_\kappa$,
	\begin{equation} \label{eq:new_term_est}
		%\frac{1}{N}\biggl\lvert\sum\limits_{a,j}\frac{m_j'(z)}{m_j(z)} S_{ja}\bigl(G_{jj}(z)-m_j(z)\bigr)\bigl(G_{aa}(z)-m_a(z)\bigr)\biggr\rvert \prec \frac{\Psi(z)^3+(\Delta_0 + \dis(z))^{1/3}\Theta^2(z)}{|\beta(z,z)|},
		\bigl\lvert\mathcal{P}(z)\bigr\rvert \prec N|\beta(z,z)|^{-1}\bigl(\Psi(z)^3+(\Delta_0 + \dis(z))^{1/3}\Theta^2(z)\bigr),
	\end{equation}
	where $\eigB(z,z)$ is the smallest eigenvalue of $\stab(z,z)$, and $\Delta_0$ is defined in Proposition \ref{prop_scaling}. 
\end{lemma}
We postpone the proof of Lemma \ref{lemma:new_term_est} until the end of the section.

To estimates the terms \eqref{eq:curlyT_terms}, we prove the following local law for the two-point function $\mathcal{T}(z,\zeta)$.
\begin{theorem} \label{th:Tlocallaw}
	For all $z,\zeta \in \D\cap\mathcal{D}_\kappa$, we have the estimate
	\begin{equation} \label{eq:curlyTlaw}
		\mathcal{T}(z,\zeta) = \Tr\biggl[ \frac{\m'(z)}{\m(z)} \bigl(1 - S\m(z)\m(\zeta) \bigr)^{-1} \bigl(S\m(z)\m(\zeta) \bigr)^2\biggr] + \mathcal{E}(z,\zeta),
	\end{equation}
	where the error term $\mathcal{E}(z,\zeta)$ is analytic in both variables and admits the bound
	\begin{equation} \label{eq:curlyT_error}
		\begin{split}
			\bigl\lvert\mathcal{E}(z,\zeta)\bigr\rvert \prec N|\eigB(z,z)|^{-1}\bigl( \Psi^2(z) \Psi(\zeta) + \Psi(z)\Psi^2(\zeta) + ( \Delta_0 + \dis(z) + \dis(\zeta))^{1/3}\Theta(z)\Theta(\zeta)\bigr).
		\end{split}
	\end{equation} 	
\end{theorem}
The key improvement compared to the bulk counterpart of Theorem \ref{th:Tlocallaw} is the factor $(\Delta_0 + \dis(z)+\dis(\zeta))^{1/3}$ in front of $\Theta(z)\Theta(\zeta)$, which makes the error estimate \eqref{eq:curlyT_error} effective in the small $\Delta$ regime. We prove Theorem \ref{th:Tlocallaw} in Subsection \ref{sec:curlyT_proof}.

Finally, to bound the integrals of the error estimates against the partial derivative of the quasi-analytic extension $\partial_{\bar{z}}\other{f}(z)$, we use the following lemma.
\begin{lemma} (c.f. Lemma 5.6 in \cite{R2023bulk}, c.f. Lemma 4.4 in \cite{Landon2020applCLT}) \label{lemma:int}
	Let $f$ be the scaled test function defined in \eqref{scaled_f_def}. Let $\Omega$ be a domain of the form
	\begin{equation}
		\Omega := \{z\in\mathbb{C}: cN^{-\tau'}\eta_0 < |\im z| < 1, a < \re z  < b\},
	\end{equation} 
	such that $\supp{f} \subset (a,b)$ and $\tau',c$ are positive constants.
	Let $K(z)$ be a holomorphic function on $\Omega$ satisfying
	\begin{equation}
		|K(z)| \le C|\im z|^{-s}(\Delta_0+|\im z|)^{-q},\quad z \in \Omega, 
	\end{equation}
	for some $q,s\ge 0$ satisfying $0 \le s+q \le 2$. Then there exists a constant $C' > 0$ depending only on $g$%in \eqref{scaled_f}
	, $\chi$% in \eqref{QA_f}
	, $s$, and $q$, such that
	\begin{equation} \label{int_bound}
		\biggl|\int\limits_{\Omega}\frac{ \partial\other{f}}{\partial\bar z}(x+\I \eta) K(x+\I \eta)\mathrm{d}x\mathrm{d}\eta \biggr| \le C C'  \frac{\eta_0^{1-s}\log N}{(\Delta_0+\eta_0)^q}.
	\end{equation}
\end{lemma}
\begin{proof}[Proof of Lemma \ref{lemma:int}] 
	The proof is analogous to that of Lemma 5.6 in \cite{R2023bulk}, except in the regime $|\eta|\lesssim \eta_0$, we use the estimate $(\Delta_0+|\eta|)^{-1} \lesssim \eta^{-1}\eta_0(\Delta_0+\eta_0)^{-1}$, and in the regime $|\eta| \gtrsim \eta_0$, we use the trivial estimate $(\Delta_0+|\eta|)^{-1} \lesssim (\Delta_0+\eta_0)^{-1}$.
\end{proof}
%\begin{lemma} (c.f. Lemma 4.4 in \cite{Landon2020applCLT}) \label{lemma:int}
%	Let $f$ be the scaled test function defined in \eqref{def_f}. Let $\Omega$ be a domain of the form
%	\begin{equation}
%		\Omega := \{z\in\mathbb{C}: cN^{-\tau'}\eta_0 < |\im z| < 1, a < \re z  < b\},
%	\end{equation} 
%	such that $\supp{f} \subset (a,b)$ and $\tau',c$ are positive constants.
%	Let $K(z)$ be a holomorphic function on $\Omega$ satisfying
%	\begin{equation}
%		|K(z)| \le C(\Delta + |\im z|)^{-q}|\im z|^{-s},\quad z \in \Omega, 
%	\end{equation}
%	for some $0 \le s \le 2$ and $q \in [0,2-s]$.
%	Then there exists a constant $C' > 0$ depending only on $g$%in \eqref{scaled_f}
%	, $\chi$% in \eqref{QA_f}
%	, $q$, and $s$, such that
%	\begin{equation} \label{int_bound}
%		\biggl|\int\limits_{\Omega}\frac{ \partial\other{f}}{\partial\bar z}(x+\I \eta) K(x+\I \eta)\mathrm{d}x\mathrm{d}\eta \biggr| \le C C'  \frac{\eta_0^{1-s}}{(\Delta+\eta_0)^q}\log N.
%	\end{equation}
%\end{lemma}
%We postpone the proof of Lemma \ref{lemma:int} until the end of the section.

\begin{proof}[Proof of Proposition \ref{prop:main1}]
	Differentiating \eqref{eq:curlyTlaw}, applying Cauchy's integral formula to the error estimate \eqref{eq:curlyT_error}, and using Lemma \ref{lemma:int} to integrate the error terms over $\zeta$, we deduce that 
	\begin{equation} \label{eq:curlyT_deriv}
		\begin{split}
			%\left\{1-\Expv\right\}[\mathcal{T}(z,z)] +
			\int\limits_{\dom'} \frac{\partial\other{f}}{\partial\bar\zeta} \frac{\partial}{\partial\zeta}\mathcal{T}(z,\zeta) \mathrm{d}\bar \zeta \mathrm{d}\zeta 
			&= \int\limits_{\dom'} \frac{\partial\other{f}}{\partial\bar \zeta}\frac{\partial}{\partial\zeta}\Tr\biggl[\frac{\m'(z)}{\m(z)}\stab(z,\zeta)^{-1}(1-\stab(z,\zeta))^2\biggr]  \mathrm{d}\bar \zeta \mathrm{d}\zeta\\ 
			&+ \mathcal{O}_\prec\biggl(\frac{N^{-1/2}\eta_0^{-1/3}\Psi(z)^2 + \eta_0^{-2/3}\Psi(z) + (\Delta_0 + \dis(z)+\eta_0^{1/3})\eta_0^{-1}\Theta(z)}{|\eigB(z,z)|} \biggr).
		\end{split}
	\end{equation}
	Combining estimates \eqref{1/mj(z)}, \eqref{eq:new_term_est}, \eqref{eq:curlyT_error}, \eqref{eq:curlyT_deriv}, and using Lemma \ref{lemma:int} to integrate the error terms over $z$ concludes the proof of Proposition \ref{prop:main1}.
\end{proof}

%
%\begin{claim} \label{claim:e_partial}
%	For all $k \neq j$, the partial derivatives of $\other{e}(\lambda)$ with respect to the matrix elements admit the bound
%	\begin{equation}
%		\biggl|\frac{\partial\other{e}(\lambda)}{\partial H_{jk}}\biggr| \prec N^{-1/2}\frac{\eta_0^{-1/4}}{(\Delta+\eta_0)^{1/12}},
%	\end{equation}
%	and the second partial derivatives admit the expression
%	\begin{equation}
%		\frac{\partial^2\other{e}(\lambda)}{\partial H_{jk}^2} = \frac{2\I\lambda}{\pi}\other{e}(\lambda)\int\limits_{\Omega_0'}\frac{\partial\other{f}}{\partial{\bar \zeta}}\frac{\partial\{m_j(\zeta)m_k(\zeta)\}}{\partial\zeta}\mathrm{d}\bar\zeta\mathrm{d}\zeta + \mathcal{O}_\prec\bigl(N^{-1/2}(1+|\lambda|)^2\eta_0^{-1/4} (\Delta + \eta_0)^{-1/12}\bigr).
%	\end{equation}
%\end{claim}

In the remained of the section, we supply the proofs of Lemmas \ref{lemma:standard_estimates}, \ref{lemma:new_term_est}%, \ref{lemma:int}
, and Theorem \ref{th:Tlocallaw}. 
\begin{proof}[Proof of Lemma \ref{lemma:new_term_est}]
	We define $\mathpzc{g}_j := G_{jj}(z)-m_j(z)$, $\mathpzc{g} := (\mathpzc{g}_j)_{j=1}^N$, and $\m := \m(z)$. From the decomposition \eqref{eq:local_law_dir}, one can immediately deduce that 
	\begin{equation}
		\mathpzc{g} = \theta \vect{b} + \vect{d},
	\end{equation}
	where $\vect{b} := \vect{b}(z)$ is defined in \eqref{b_def}, $\theta :=\theta(z)$ is given in Theorem \ref{th:local_law}, and the vector $\vect{d} :=(D_{jj}(z))_{j=1}^N$ consists of the diagonal elements of $D(z)$, defined in Theorem \ref{th:local_law}. Therefore, the function $\mathcal{P}(z)$ defined in \eqref{eq:Curly_P} is given by
	\begin{equation} \label{eq:quad_est}
		\frac{1}{N}\sum\limits_{a,j}\frac{m_j'}{m_j} S_{ja}\mathpzc{g}_j\mathpzc{g}_a = \theta^2 \bigl\langle  \m'\m^{-1}\vect{b}S[\vect{b}]\bigr\rangle +  \theta\bigl\langle \m'\m^{-1}\vect{b}S[\vect{d}]\bigr\rangle  +  \theta\bigl\langle \m'\m^{-1}\vect{d}S[\vect{b}]\bigr\rangle + \bigl\langle \m'\m^{-1}\vect{d}S[\vect{d}]\bigr\rangle.
	\end{equation}

	It follows from \eqref{eq:m'_bound} and \eqref{eq:D_est} that the last three terms on the right-hand side of \eqref{eq:quad_est} are bounded by
	\begin{equation} \label{eq:last_three terms}
		\bigl\lvert\theta\bigl\langle \m'\m^{-1}\vect{b}S[\vect{d}]\bigr\rangle  +  \theta\bigl\langle \m'\m^{-1}\vect{d}S[\vect{b}]\bigr\rangle + \bigl\langle \m'\m^{-1}\vect{d}S[\vect{d}]\bigr\rangle \bigr\rvert \prec |\eigB|^{-1}\Psi^3 ,
	\end{equation}
	where $\eigB:=\eigB(z,z)$, $\Psi:=\Psi(z)$, and $\Theta:=\Theta(z)$, and we used that for all $z\in\mathcal{D}_\kappa$, $\Theta \lesssim \Psi$.
%	For all $z\in \supp{\other{f}}$, $|\re z - \sng| \lesssim \eta_0$, hence, by applying the comparison relations of Proposition \eqref{prop_scaling}, we deduce that for all $z\in \D\cap\supp{f}$, we have
%	\begin{equation} \label{eq:rhoNeta}
%		\rho N\eta \gg 1,
%	\end{equation}
%	which implies $\Theta^3 \lesssim \Psi^3$.

	To estimate the remaining term on the right-hand side of \eqref{eq:quad_est}, we show that 
	\begin{equation} \label{eq:trace_est_later}
		|\bigl\langle  \m'\m^{-1}\vect{b}S[\vect{b}]\bigr\rangle| \lesssim |\eigB|^{-1}(\Delta_0 + \dis(z) )^{1/3} ,%1 + |\eigB|^{-1}\bigl(|\sigma| + \rho + \rho^{-1}|\eta|\bigr),
	\end{equation}
	which, together with \eqref{eq:theta_est} and \eqref{eq:last_three terms}, allows us to conclude \eqref{eq:new_term_est}.
	
	Therefore, it remains to establish \eqref{eq:trace_est_later}. Recall that differentiating the Dyson equation \eqref{VDE} yields the identity $\m'(z) = \stab(z,z)^{-1}[\m(z)^2]$. Owing to the decomposition $\stab^{-1} = \eigB^{-1}\Pi + \stab^{-1}Q$ and the bound \eqref{eq:invstab0Q}, we obtain
	\begin{equation}
		\bigl\langle  \m'\m^{-1}\vect{b}S[\vect{b}]\bigr\rangle = \frac{\bigl\langle\overline{\m^{-2}\vect{b}}, \vect{b}\m S[\vect{b}]\bigr\rangle \langle \vect{b}^\ell,\m^2 \rangle}{\eigB\langle\vect{b}^\ell,\vect{b}\rangle} + \mathcal{O}(1),
	\end{equation}
	where $\vect{b}^\ell := \vect{b}^\ell(z)$ is defined in \eqref{b_def}. Since $\stab(z,z)^* = \m(\bar{z})^{-2}\stab(\bar{z},\bar{z})\m(\bar{z})^2$, the vector $\overline{\m^{-2}\vect{b}}$ is parallel to $\vect{b}^\ell$, i.e. $\overline{\m^{-2}\vect{b}} = \vect{b}^\ell \langle \vect{b},\overline{\m^{-2}\vect{b}} \rangle \langle \vect{b}, \vect{b}^\ell\rangle$. It follows from the estimates \eqref{b_vect}, \eqref{eq:m'_bound} that 
	\begin{equation} \label{eq:trace_est_final}
		|\bigl\langle  \m'\m^{-1}\vect{b}S[\vect{b}]\bigr\rangle| \lesssim 1+ |\eigB|^{-1} |\langle \vect{b}^\ell, \m\vect{b}S[\vect{b}]\rangle| \lesssim 1+|\eigB|^{-1}\bigl(\rho + |\sigma| + \rho^{-1}|\eta|\bigr).
	\end{equation}
	Owing to Lemma \ref{lemma_m}, the functions $\rho(z),\sigma(z)$, and $\rho(z)^{-1}\im z$  are  $1/3$-H\"{o}lder continuous, hence \eqref{eq:trace_est_later} follows from  \eqref{eq:trace_est_final} and \eqref{eq:sigma_comp}. This concludes the proof of Lemma \ref{lemma:new_term_est}.
\end{proof}

\begin{proof} [Proof of Lemma \ref{lemma:standard_estimates}]
	%	\begin{subequations} \label{1/mj(z)}
		%		\begin{align}
			%			\bigl(\stab(z,z)\bigl[\mathpzc{g}(z)\bigr]\bigr)_j =&~ m_j(z)\sum\limits_{a=1}^N S_{ja}\Expv\bigl[\other{e}(\lambda)\{1-\Expv\}[G_{jj}(z)G_{aa}(z)]\bigr] \label{eq:new_term} \\ 
			%			&+ m_j(z)\Expv\bigl[\other e(\lambda)\{1-\Expv\}[ T_{jj}(z,z)]\bigr]\\
			%			&- m_j(z)\Expv\biggl[\sum\limits_{k=1}^N S_{jk}G_{kj}(z) \frac{\partial\other e(\lambda)}{\partial H_{jk}}\biggr]\\
			%			&+ m_j(z)\frac{1}{2}\sum\limits_{k=1}^N \Cmlnt^{(4)}_{jk}m_j(z)m_k(z)\Expv\biggl[\frac{\partial^2\other e(\lambda)}{\partial H_{jk}^2}\biggr]\\
			%			&+\mathcal{O}\bigl((1+|\lambda|^4)(N^{-1} \Psi(z)\eta_0^{-1/4}(\Delta +\eta_0)^{-1/12})\bigr).
			%		\end{align}
		%	\end{subequations}
	Lemma \ref{lemma:standard_estimates} is proved analogously to Lemma 5.4 in \cite{R2023bulk}, but makes use of the comparison relations \eqref{eq:rho_comp} to yield improved estimate obtained using Lemma \ref{lemma:int}.
\end{proof}

\subsection{Two-point function local law} \label{sec:curlyT_proof}
The proof of Theorem \ref{th:Tlocallaw} relies on two key ingredients: the stable direction local law for two-point functions  of the resolvent (Proposition \ref{prop:stable_local_law} below), established in \cite{R2023bulk}, and the proximity of the vector of ones to the destabilizing direction of $\stab(z,\zeta)$, which is captured by Claim \ref{claim:stab_res} below. 
\begin{prop} (Lemma 6.1 in \cite{R2023bulk}) \label{prop:stable_local_law}
	For all $z, \zeta \in \D\cap\mathcal{D}_\kappa$, and for all deterministic $N\times N$ matrices $\X$ satisfying $\Pi(z,\zeta)^t\X = 0$,
	\begin{equation} \label{(1-Smm)T_law}
		\begin{split}
			\sum_{a\neq b}\X_{ab} G_{ba}(z)G_{ab}(\zeta) =&~ \bigl[\m(z)\m(\zeta)S\m(z)\m(\zeta)\bigl(1-S\m(z)\m(\zeta)\bigr)^{-1}\X\bigr]_{bb}\\ 
			&+ \Oprec\bigl(N \norm{\X}_{\max}\Psi(z)\Psi(\zeta)(\Psi(z)+\Psi(\zeta))\bigr),
		\end{split}
	\end{equation}
	where $\norm{\X}_{\max} := \max\limits_{j,k} |\X_{jk}|$.
\end{prop}

Careful examination of the proof of Lemma 6.1 in \cite{R2023bulk} reveals that the real parts of the spectral parameters $z$ and $\zeta$ do not have to lie in the bulk of spectrum. The only necessary assumption is that $z$ and $\zeta$ belong to the domain where the single resolvent local laws of Theorem \ref{th:local_law} hold true. Therefore, the proof of Proposition \ref{prop:stable_local_law} can be directly imported from \cite{R2023bulk}.

\begin{claim} \label{claim:stab_res}
	The stability operator $\stab(z,\zeta)$ satisfies the following properties.
	\begin{enumerate}
		\item[(i)] There is a threshold $c_0\sim 1$, such that for all $z,\zeta \in \D$ the estimate
		\begin{equation} \label{eq:stab_gap}
			\opnormtwo{(w-\stab(z,\zeta))^{-1}} + \opnorminf{(w-\stab(z,\zeta))^{-1}} \lesssim 1
		\end{equation}
		holds uniformly for all $w \in\mathbb{C}$ satisfying $|w|\ge \epsilon$ and $|1-w| \ge 1- 2\epsilon$, where $\epsilon>0$ is the radius in Lemma \ref{lemma_stab}.
		\item[(ii)] Furthermore, $\stab(z,\zeta)$ has a single simple eigenvalue $\eigB(z,\zeta)$ with $|\eigB(z,\zeta)| \lesssim \epsilon$, that is 
		\begin{equation} \label{eq:rankPi}
			\mathrm{rank}\,\Pi(z,\zeta) = 1, \quad \text{where} \quad \Pi(z,\zeta) := \frac{1}{2\pi\I}\oint_{|w|=\epsilon} \bigl(w-\stab(z,\zeta)\bigr)^{-1}\mathrm{d}w.
		\end{equation}
		Moreover, we have 
		\begin{equation} \label{eq:Pi_norm}
			\opnormtwo{\Pi(z,\zeta)} + \opnorminf{\Pi(z,\zeta)} \lesssim 1.
		\end{equation}
		\item[(iii)] Define $Q(z,\zeta) :=1-\Pi(z,\zeta)$, then 
		\begin{equation} \label{eq:inbstabQ}
			\opnormtwo{\stab(z,\zeta)^{-1}Q(z,\zeta)} + \opnorminf{\stab(z,\zeta)^{-1}Q(z,\zeta)} \lesssim 1.
		\end{equation}
		\item[(iv)] We have the lower bound
		\begin{equation} \label{eq:Pi1_lowerbound}
			\norm{\Pi(z,\zeta)^*\vect{1}}_\infty \gtrsim \opnorminf{\Pi(z,\zeta)^*}.
		\end{equation} 
	\end{enumerate}
\end{claim}
We defer the proof of Claim \ref{claim:stab_res} until the end of the subsection. 

Armed with Proposition \ref{prop:stable_local_law} and Claim \ref{claim:stab_res}, we are ready to prove Theorem \ref{th:Tlocallaw}.
\begin{proof}[Proof of Theorem \ref{th:Tlocallaw}]
	Denote $G := G(z)$, $\other{G} := G(\zeta)$, $\m := \m(z)$, $\other{\m} := \m(\zeta)$, $\Pi := \Pi(z,\zeta)$, and let $W := \m^{-1}\m'S$. It follows immediately from Assumption \eqref{cond_A} and estimate \eqref{eq:m'_bound}, that $\norm{W}_{\mathrm{max}} \lesssim N^{-1}|\eigB(z,z)|^{-1}$.
	Estimate \eqref{eq:Pi1_lowerbound} asserts that $\Pi^*\vect{1} \neq 0$, hence there exists an $N\times N$ matrix $Y := Y(z,\zeta)$ and a vector $\s := \s(z,\zeta)$ such that 
	\begin{equation} \label{eq:Wdecomp}
		W = \Y + \vect{1}\s^*, \quad \Pi^t Y = 0.
	\end{equation}
	Under the decomposition \eqref{eq:Wdecomp}, the left-hand side of \eqref{eq:curlyTlaw} can be written as
	\begin{equation} \label{eq:curlyT_sums}
		\mathcal{T}(z,\zeta) = \sum\limits_{b}\sum\limits_{a\neq b} \Y_{ab}G_{ba}\other{G}_{ab} + \sum\limits_{b} \overline{s}_{b}\frac{G_{bb}-\other{G}_{bb}}{z-\zeta} - \sum\limits_{b}\overline{s}_{b}G_{bb}\other{G}_{bb},
	\end{equation}
	where we applied the resolvent identity in the form $G(z)G(\zeta)(z-\zeta) = G(z)-G(\zeta)$. We now estimate each sum in \eqref{eq:curlyT_sums} separately. 
	
	By the stable direction local law of Proposition \ref{prop:stable_local_law}, the first sum on the right-hand side of \eqref{eq:curlyT_sums} admits the estimate
	\begin{equation} \label{eq:curlyT_stable}
		\sum\limits_{b}\sum\limits_{a\neq b} \Y_{ab}G_{ba}\other{G}_{ab} = \Tr\bigl[\m\other{\m}S\m\other{\m}\bigl(1-S\m\other{\m}\bigr)^{-1}\Y\bigr] + \Oprec\bigl(N|\eigB|^{-1}\Psi\other{\Psi}(\Psi+\other{\Psi})\bigr),
	\end{equation}
	where $\eigB:=\eigB(z,z)$, $\Psi := \Psi(z)$, $\other{\Psi} := \Psi(\zeta)$.
	
	The third sum on the right-hand side of  \eqref{eq:curlyT_sums} is estimated using \eqref{eq:m'_bound}, \eqref{eq:theta_est} and \eqref{eq:D_est},
	\begin{equation} \label{eq:curlyT_GG}
		\sum\limits_{b}\overline{s}_{b}G_{bb}\other{G}_{bb} = \langle\s, \m\other{\m} \rangle + \mathcal{O}_\prec\bigl(|\eigB(z,z)|^{-1}(\Theta+\other{\Theta} + \Psi\other{\Psi})\bigr),
	\end{equation}
	where $\Theta := \Theta(z)$, $\other{\Theta} := \Theta(\zeta)$ are defined in \eqref{eq:control_parameters}. Furthermore, we have $\Psi\other{\Psi} \lesssim \Psi^2+\other{\Psi}^2 \lesssim \Theta +\other{\Theta}$ for all $z,\zeta\in\D\cap\mathcal{D}_\kappa$ as a result of the comparison relations \eqref{eq:rho_comp}. 
	
	To estimate the second sum in \eqref{eq:curlyT_sums}, we employ the decomposition \eqref{eq:local_law_dir} and the first bound in \eqref{eq:D_est} to deduce that 
	\begin{equation} \label{eq:curlyT_G-G}
		 \mathcal{E}_2(z,\zeta) := 
		 \sum\limits_{b} \overline{s}_{b}\frac{G_{bb}-\other{G}_{bb}}{z-\zeta} - \bigl\langle\s, \frac{\m-\other{\m}}{z-\zeta} \bigr\rangle  
		 =
		 \frac{\theta(z) }{z-\zeta} \bigl\langle\s, \vect{b}(z)\bigl\rangle +\frac{\theta(\zeta) }{z-\zeta} \bigl\langle\s, \vect{b}(\zeta)\bigr\rangle + \mathcal{O}_\prec\bigl(\frac{\Psi^2 + \other{\Psi}^2}{|\beta||z-\zeta|}\bigr).
	\end{equation}
	We now estimate the scalar products $\langle\s,\vect{b}(z)\rangle$ and $\langle\s,\vect{b}(\zeta)\rangle$. 
	Applying $\vect{b}(z)^t\Pi^t$ to both sides of \eqref{eq:Wdecomp}, multiplying by $\vect{b}(z)$ from the right, and dividing by $\langle\Pi[\vect{b}(z)]\rangle$ yields
	\begin{equation} \label{eq:sb}
		\langle\s,\vect{b}(z)\rangle = \langle\Pi[\vect{b}(z)]\rangle^{-1}\langle\Pi[\vect{b}(z)]\m^{-1}\m'S[\vect{b}(z)]\rangle = \mathcal{O}\bigl( |\eigB|^{-1}( \Delta_0 + \dis(z)+\dis(\zeta))^{1/3}\bigr),
	\end{equation}
	where we used the estimate $\Pi = \Pi(z,z) + |z-\zeta|^{1/3} + \rho(z)$ resulting from \eqref{Pi_pert}, together with the bounds $\norm{\m'}_\infty \lesssim |\eigB(z,z)|^{-1}$ and \eqref{eq:trace_est_later}. 
	By \eqref{Pi_pert}, $\vect{b}(z)$ is  $1/3$-H\"{o}lder continuous, hence the same bound holds for $\langle\s,\vect{b}(\zeta)\rangle$. 
	
	Therefore, the function $\eigB(z,z)\mathcal{E}_2(z,\zeta)$ is analytic and admits the bound  
	\begin{equation} \label{eq:curlyT_G-G-bound}
		|\eigB(z,z)\mathcal{E}_2(z,\zeta)|
		\prec 		N( \Delta + \dis(z)+\dis(\zeta))^{1/3}\Theta\other{\Theta}+N\Psi^2\other{\Theta} + N\Theta\other{\Psi}^2.
	\end{equation}
	In the regime $|z-\zeta| \ge \max\{\dist(z, \supp{\rho}), \dist(\zeta, \supp{\rho})\}/2$, the estimate \eqref{eq:curlyT_G-G-bound} follows immediately from \eqref{eq:curlyT_G-G}, \eqref{eq:sb} and \eqref{eq:theta_est}. In the complementary regime, \eqref{eq:curlyT_G-G-bound} follows from Cauchy's integral formula.
	
	Combining estimate \eqref{eq:curlyT_stable}, \eqref{eq:curlyT_GG}, \eqref{eq:curlyT_G-G-bound} with identities $\m\other{\m}(1-S\m\other{\m})^{-1}[\vect{1}] = (z-\zeta)^{-1}(\m-\other{\m})$ and $\m\other{\m}S\m\other{\m}(1-S\m\other{\m})^{-1} = \m\other{\m}(1-S\m\other{\m})^{-1} - \m\other{\m}$ concludes the proof of Theorem \ref{th:Tlocallaw}.
\end{proof}

%
%
%
%\subsection{Spectral Gap and Destabilizing Eigenprojector} \label{subs_spectral}
	We close the subsection by presenting the proof of Claim \ref{claim:stab_res}.
	
	\begin{proof}[Proof of Claim \ref{claim:stab_res}]
		Let $\stab:=\stab(z,\zeta)$, $\stab_0 := \stab(\sng,\sng)$.	
		
		First, we prove \eqref{eq:stab_gap}. From the uniform bound \eqref{dm_bound}, we have $\opnormtwo{\stab-\stab_0} \lesssim |z-\sng|^{1/3} + |\zeta-\sng|^{1/3}$, hence there exists a constant $c>0$ such that for all $z,\zeta\in\D$, and for all $w$ satisfying $|w|>\epsilon$ and $|1-w| \ge 1-2\epsilon$, we have
		\begin{equation} \label{eq:diff_norm_small}
			\opnormtwo{(w-\stab_0)^{-1}}\opnormtwo{\stab-\stab_0} \le 1/2.
		\end{equation} 
		We conclude the proof of \eqref{eq:stab_gap} via the following chain of inequalities
		\begin{equation}
			\norm{(w-\stab)^{-1}} \lesssim \opnormtwo{(w-\stab_0)^{-1}}\opnormtwo{\bigl(1-(w-\stab_0)^{-1}(\stab-\stab_0)\bigr)^{-1}} \lesssim 1,
		\end{equation}
		where in the last step we used \eqref{stab0_gap} and \eqref{eq:diff_norm_small}. The proof of the $\opnorminf{\cdot}$ bound is analogous. 
		
		We turn to prove \eqref{eq:rankPi}. Let $\Pi := \Pi(z,\zeta)$ be defined as in \eqref{eq:rankPi}, and let $\Omega:= \{w\in\mathbb{C}:|w|\ge\epsilon, |1-w|\ge 1-2\epsilon\}$. Observe that the estimate \eqref{eq:stab_gap} and the analyticity of $\Tr[(w-\stab)^{-1}]$ in $w$ imply  
		\begin{equation} \label{eq:trPi}
			\frac{1}{2\pi\I}\oint_{\partial\Omega}\Tr\bigl[(w-\stab)^{-1}\bigr]\mathrm{d}w = 0, \quad z,\zeta \in \D,
		\end{equation}
		and hence $\stab$ has no eigenvalues in $\Omega$ for any $z,\zeta \in \D$. Furthermore, since the function $\m$ is continuous on the line segments $[\sng,z]$ and $[\sng,\zeta]$, so are the eigenvalues of $\stab(z,\zeta)$. Therefore, no eigenvalue moves between the connected components of complement of $\Omega$, and by Lemma \ref{lemma_stab}, $\mathrm{rank}\,\Pi(\sng,\sng) = 1$, hence \eqref{eq:rankPi} holds.
		
		The estimate $\eqref{eq:Pi_norm}$ follows form \eqref{m_bound}, \eqref{v_def}, \eqref{f_def} and \eqref{Pi_pert} with $z_0=\zeta_0:=\sng$.	
		Similarly, to prove estimate $\eqref{eq:inbstabQ}$, we combine \eqref{eq:invstab0Q} and \eqref{Q/B_pert} with $z_0=\zeta_0 := \sng$.
		
		Finally, we prove the lower bound \eqref{eq:Pi1_lowerbound}. Similarly to \eqref{eq:Pi_norm}, we have $\opnorminf{\Pi^*} \lesssim 1$. From \eqref{Pi_pert}, we deduce that $\opnorminf{\Pi-\Pi_0} \lesssim |z-\sng|^{1/3}+|\zeta-\sng|^{1/3}$. It follows from \eqref{m_bound}, \eqref{v_def}, and \eqref{f_def}, that  $\norm{\Pi_0^*\vect{1}}_\infty \gtrsim 1$, hence \eqref{eq:Pi1_lowerbound} holds for all $z,\zeta \in \D$ with sufficiently small $c\sim 1$. This concludes the proof of Claim \ref{claim:stab_res}.		
	\end{proof}

\section{Computation of the Variance} \label{sec:main2proof}
As was established in \cite{Alt2020energy}, the function $\dM(\sng, \cdot)$ can be estimated by an explicit function $\sol$, given by the appropriate solution to cubic equation associated with \eqref{dm_cubic}.
\begin{lemma} (Proposition 7.10 in \cite{Alt2020energy}) \label{lemma:sol}
	There exists a threshold $c_* \sim 1$, such that for all $w \in [-c_*,c_*]$, the function $\dM(\sng,\sng+w)$ admits the estimate 
	\begin{equation} \label{eq:dM_approx}
		\dM(\sng,\sng+w) := \dM(\sng+\I0, w) = \sol(w) + \mathcal{O}\bigl(\min\{|w||\Delta_0|^{-1/3}, |w|^{2/3}\}\bigr).
	\end{equation}
	\begin{enumerate}
		\item[(i)] If $\sng$ is an exact cusp (recall Definition \ref{def_sng}), then the function $\sol(w)$ is given by
		\begin{equation} \label{eq:sol_cusp}
			\sol(w) := \frac{\pi^{1/3}}{\psi^{1/3}} |w|^{1/3}\begin{cases}
				e^{\I\pi/3}, \quad &w\ge 0,\\
				e^{2\I\pi/3}, \quad &w < 0,
			\end{cases}
		\end{equation}
		where $\psi :=\psi(\sng)$ is defined in \eqref{psi_def}.
		\item[(ii)] If $\sng$ a simple edge point, i.e., $\Delta(\sng) \ge \Delta_*$ for some positive threshold $\Delta_* \sim 1$, then the function $\sol(w)$ is given by
		\begin{equation} \label{eq:sol_edge}
			\sol(w) := \frac{\pi^{1/2}}{|\sigma|^{1/2}}|w|^{1/2}\begin{cases}
				\I  , \quad &\sign(w) = \sign(\sigma),\\
				-\sign(\sigma), \quad &\sign(w) = - \sign(\sigma),
			\end{cases}
		\end{equation}
		where $\sigma := \sigma(\sng)$ is defined in \eqref{sigma_def}.
		\item [(iii)] If $\sng$ is an edge point adjacent to a gap of size $\Delta(\sng) < \Delta_*$, then the function $\sol(w)$ is given by
		\begin{equation} \label{eq:sol_gap}
			\sol(w) := \frac{|\sigma|}{3\psi} \bigl(\gapsol(\sign(\sigma)+2\widehat{\Delta}^{-1}w) - \sign(\sigma)\bigr),
		\end{equation}
		where $\widehat{\Delta}:=\widehat{\Delta}(\sng)$ is defined in \eqref{def:delta_hat} and
		\begin{equation}
			\gapsol(\lambda) := \begin{cases}
				e^{\I\pi/3}(\lambda + \sqrt{\lambda^2-1})^{1/3} + e^{-\I\pi/3}(\lambda - \sqrt{\lambda^2-1})^{1/3}, \quad &\lambda \ge 1,\\
				e^{\I\pi/3}(\lambda + \I\sqrt{1-\lambda^2})^{1/3} + e^{-\I\pi/3}(\lambda - \I\sqrt{1-\lambda^2})^{1/3}, \quad &|\lambda| < 1,\\
				-e^{\I\pi/3}(-\lambda - \sqrt{\lambda^2-1})^{1/3} - e^{-\I\pi/3}(-\lambda + \sqrt{\lambda^2-1})^{1/3}, \quad &\lambda \le 1.\\
			\end{cases}
		\end{equation}
		Furthermore, the function $\gapsol(\lambda)$ satisfies the equation
		\begin{equation} \label{eq:gapsol_eq}
			\gapsol(\lambda)^3 - 3\gapsol(\lambda) + 2 \lambda = 0.
		\end{equation}
	\end{enumerate}
\end{lemma}

We show that the the variance $V(f)$ can be estimated by an integral with a kernel expressed in terms of the explicit function $\sol(w)$.
\begin{lemma} \label{lemma:variance} Under the conditions of Theorem \ref{th_main}, there exists an interval $\mathcal{R}$ of length $|\mathcal{R}| \sim N^{-\varepsilon_0/3}$, with $\dist\{\sng,\mathbb{R}\backslash\mathcal{R}\}\sim N^{-\varepsilon_0/3}$, such that the variance $V(f)$ defined in \eqref{variance_V} satisfies
	\begin{equation} \label{eq:V_sol}
		V(f) = \frac{1}{4\pi^2}\iint\limits_{\mathcal{R}^2} \bigl(f(\sng+w)-f(\sng+\other{w})\bigr)^2\other{\mathcal{K}}(w,\other{w}) \mathrm{d}w\mathrm{d}\other{w} + \mathcal{O}(N^{-\varepsilon_0/9}+\eta_0^{1/9}),
	\end{equation}
	where the kernel $\other{\mathcal{K}}(w,\other{w})$ is given by 
	\begin{equation} \label{eq:K(h)_kernel}
		\other{\mathcal{K}}(w,\other{w}) =2\re\biggl[ \frac{\bigl((\sigma + \psi (\sol + \other{\sol}) )^2 + 2\psi^2\sol\other{\sol}\bigr)(\sol-\other{\sol})^2}{(2\sigma\sol + 3\psi\sol^2)(2\sigma\other{\sol} + 3\psi\other{\sol}^2)(w-\other{w})^2} 
		- \frac{\bigl((\sigma + \psi (\bar{\sol} + \other{\sol}) )^2 + 2\psi^2\bar{\sol}\other{\sol}\bigr)(\bar{\sol}-\other{\sol})^2}{(2\sigma\bar{\sol} + 3\psi\bar{\sol}^2)(2\sigma\other{\sol} + 3\psi\other{\sol}^2)(w-\other{w})^2} \biggr],
	\end{equation}
	and $\sol := \sol(w)$, $\other{\sol} := \sol(\other{w})$ are defined in Lemma \ref{lemma:sol}.
\end{lemma}
The proof of Lemma \ref{lemma:variance} is deferred to Section \ref{sec:variance_lemma_proof}.
Equipped with Lemmas \ref{lemma:sol} and \ref{lemma:variance}, we proceed to prove Proposition \ref{prop:main2}.
\begin{proof}[Proof of Proposition \ref{prop:main2}]
	We focus on proving \eqref{eq:main2} and the continuity of the variance functional $\mathrm{Var}_{\alpha}$ in \eqref{eq:Var_converge}. The proof of \eqref{eq:main2_bias} and the continuity of the bias functional $\mathrm{Bias}_{\alpha}$ are deferred to Section \ref{sec:bias}.
	
	First, assume that $\sng$ is a simple edge point, that is $\Delta(\sng) \ge \Delta_*$, and hence $\sigma(\sng) \sim 1$. Plugging \eqref{eq:sol_edge} into \eqref{eq:K(h)_kernel} reveals that for all $x:=\sng +w,y :=\sng+\other{w} \in \mathcal{R}\cap\supp{\rho}$,
	\begin{equation}
		\other{\mathcal{K}}(w,\other{w}) %= \frac{(|w|+|\other{w}|) - \pi\sigma^{-3}\psi^2(4 |w|^2 + 9 |w||\other{w}|  + 4 |\other{w}|^2 )/2 - 9\pi^2\sigma^{-6}\psi^4(|w|^2 |\other{w}|  + |w| |\other{w}|^2)/2}{|w|^{1/2}|\other{w}|^{1/2} (1 - 9\pi\sigma^{-3}\psi^2 |w|/4) (1 - 9\pi\sigma^{-3}\psi^2 |\other{w}|/4)},
		= \frac{|w|+|\other{w}|}{|w|^{1/2}|\other{w}|^{1/2} (w-\other{w})^2 }\bigl(1 + \mathcal{O}(N^{-\varepsilon_0/3})\bigr)=\frac{\mathcal{K}_{\mathrm{edge}}(|w|,|\other{w}|)}{(w-\other{w})^2}\bigl(1 + \mathcal{O}(N^{-\varepsilon_0/3})\bigr),
	\end{equation}
	where we used that $\sigma(\sng)\sim 1$, $\psi(\sng)\lesssim 1$, and $|w|,|\other{w}|\lesssim N^{-\varepsilon_0/3}$. Similarly, if either $x:=\sng +w$ or $y :=\sng+\other{w}$  lie in $\mathcal{R}\backslash\supp{\rho}$, then $\other{K}(w,\other{w}) = 0$. Furthermore, 
	\begin{equation}
		\int\limits_{0}^{+\infty}\int\limits_{N^{-\varepsilon_0/3}}^{+\infty}\frac{\bigl(f(\sng+\sign(\sigma)w)-f(\sng+\sign(\sigma)\other{w})\bigr)^2}{(w-\other{w})^2}\mathcal{K}_{\mathrm{edge}}(w,\other{w})\mathrm{d}w\mathrm{d}\other{w} = \mathcal{O}\bigl( \eta_0^{1/2}N^{\varepsilon_0/3}\bigr).
	\end{equation}
	Therefore, if $\sng$ is a simple edge, then $V(f)$ satisfies \eqref{eq:main2} with $\mathrm{Var}_{\infty}$ given by \eqref{main_edge}. The explicit expression \eqref{eq:norm_edge} follows from \eqref{main_edge} and \eqref{eq:K_edge} by changing the integration variables to $x^{1/2}$, and $y^{1/2}$.
	
	Next, we consider the case of $\sng$ being an exact cusp point, i.e. $\sigma(\sng)=0$. Note that in this case the estimate \eqref{eq:psi_sim_1} implies that $\psi(\sng) \sim 1$. It follows from \eqref{eq:K(h)_kernel} and \eqref{eq:sol_cusp}, that 
	\begin{equation}
		\other{\mathcal{K}}(w,\other{w}) %=\frac{2}{9(w-\other{w})^2}\re\biggl[ \frac{\bigl(\sol^2 + 4\sol\other{\sol}+ \other{\sol}^2 \bigr)(\sol-\other{\sol})^2}{\sol^2\other{\sol}^2} 
		%- \frac{\bigl(\bar{\sol}^2+4\bar{\sol}\other{\sol} + \other{\sol}^2 \bigr)(\bar{\sol}-\other{\sol})^2}{\bar{\sol} ^2\other{\sol}^2} \biggr],
		=\frac{1}{3}\cdot\frac{\sign(w\other{w})(|w|^{4/3}+|\other{w}|^{4/3})+2|w|^{1/3}|\other{w}|^{1/3}(|w|^{2/3}+|\other{w}|^{2/3})}{(w-\other{w})^2|w|^{2/3}|\other{w}|^{2/3}} = \frac{\mathcal{K}_{\mathrm{cusp}}(\sqrt[3]{w},\sqrt[3]{\other{w}})}{9 |w|^{2/3}|\other{w}|^{2/3}},
	\end{equation}
	where we adhere to the convention $\sqrt[3]{x} := |x|^{1/3}\sign(x)$.
	Observe that 
	\begin{equation} \label{eq:tail_est}
		\int\limits_{-\infty}^{+\infty}\int\limits_{N^{-\varepsilon_0/3}}^{+\infty}\frac{\bigl(g(w)-g(\other{w})\bigr)^2}{(w-\other{w})^2|w|^{2/3}|\other{w}|^{2/3}}\mathcal{K}_{\mathrm{cusp}}(w,\other{w})\mathrm{d}w\mathrm{d}\other{w} = \mathcal{O}\bigl(\eta_0^{1/3}N^{\varepsilon_0/9}\bigr),
	\end{equation}
	which follows from the fact that $f$ is supported on $[\sng -C\eta_0, \sng +C\eta_0]$, and the function $\mathcal{K}_{\mathrm{cusp}}$ satisfies
	\begin{equation} \label{eq:cusp_K_antidev}
		\frac{\mathcal{K}_{\mathrm{cusp}}(x,y)}{(x^3-y^3)^2} = \frac{\partial}{\partial x} \frac{3\sign(xy)(x^2+xy)}{(x^3-y^3)}.
	\end{equation} 
	We conclude that if $\sng$ is an exact cusp point, \eqref{eq:main2} holds with $\mathrm{Var}_{0}$ given by \eqref{main_cups}. 
	
	To prove the equivalence \eqref{eq:norm_cusp}, it suffices to consider even and odd test functions $g$. Indeed, for any $g\in C^2_c$, let $g^{\mathrm{e}}(x):= \tfrac{1}{2}(g(x)+g(-x))$ and $g^{\mathrm{o}}(x):=\tfrac{1}{2}(g(x)-g(-x))$ be the even and odd parts of the function $g$, respectively. Then, since $\mathcal{K}_{\mathrm{cusp}}(-x,-y) = \mathcal{K}_{\mathrm{cusp}}(x,y)$, we have 
	\begin{equation}
		\mathrm{Var}_{0}(g) = \mathrm{Var}_{0}(g^{\mathrm{e}})+\mathrm{Var}_{0}(g^{\mathrm{o}}).
	\end{equation}
	For odd test functions, performing the change of variables $w\mapsto t^3$, $\other{w}\mapsto \other{t}^3$, we find
	\begin{equation}	\label{eq:cusp_odd_V}
		\mathrm{Var}_{0}(g^{\mathrm{o}}) 
		%\frac{2}{3}\int\limits_{0}^{+\infty}\int\limits_{0}^{+\infty}\bigl(g(w)-g(\other{w})\bigr)^2\frac{w^{10/3}+5w^2\other{w}^{4/3}+5w^{4/3}\other{w}^2 +\other{w}^{10/3}}{w^{2/3}\other{w}^{2/3}(w^2-\other{w}^2)^2}\mathrm{d}w\mathrm{d}\other{w},
		= \frac{3}{\pi^2}\int\limits_{0}^{+\infty}\int\limits_{0}^{+\infty}\frac{\bigl(g^o(x^3)-g^o(y^3)\bigr)^2}{(x^2-y^2)^2}\frac{x^8-x^6y^2+6t^4y^4-x^2y^6+y^8}{(x^4+x^2y^2+y^4)^2}(x^2+y^2)\mathrm{d}x\mathrm{d}y,
	\end{equation}
	where we used \eqref{eq:cusp_K_antidev} to assert that for all $g\in C^2_c(\mathbb{R})$,
	\begin{equation}
		\int\limits_{0}^{+\infty}\int\limits_{0}^{+\infty}\bigl(g^o(x)\bigr)^2\frac{x^{4/3}-2xy^{1/3} -2x^{1/3}y+y^{4/3}}{(x+y)^2x^{2/3}y^{2/3}}\mathrm{d}x\mathrm{d}y=0.
	\end{equation}
	Identity \eqref{eq:cusp_odd_V} implies that 
	\begin{equation}
		\mathrm{Var}_{0}(g^{\mathrm{o}}) \sim \norm{g^o(|x|^3)}_{\dot{H}^{1/2}}^2, \quad \mathrm{Var}_{0}(g^{\mathrm{o}}) = \frac{3}{4\pi^{2}}\norm{g^o(|x|^3)}_{\dot{H}^{1/2}}^2 -\frac{1}{4\pi^{2}}\norm{g^o(|x|)}_{\dot{H}^{1/2}}^2.
	\end{equation}
	On the other hand, for even test functions we have
	\begin{equation}	\label{eq:cusp_even_V}
		\mathrm{Var}_{0}(g^{\mathrm{e}})
		%=\frac{4}{3}\int\limits_{0}^{+\infty}\int\limits_{0}^{+\infty}\bigl(g^e(w)-g^e(\other{w})\bigr)^2\frac{w^{8/3}+2w^2\other{w}^{2/3}+2w^{2/3}\other{w}^2 +\other{w}^{8/3}}{w^{1/3}\other{w}^{1/3}(w^2-\other{w}^2)^2}\mathrm{d}w\mathrm{d}\other{w},
		= \frac{3}{2\pi^2}\int\limits_{0}^{+\infty}\int\limits_{0}^{+\infty}\frac{\bigl(g^e(x^{3/2})-g^e(y^{3/2})\bigr)^2}{(x-y)^2}\frac{x^4+2x^3y+2xy^3 +y^4}{(x^2+xy+y^2)^2}\mathrm{d}x\mathrm{d}y,
	\end{equation}
	which immediately implies that 
	\begin{equation}
		\mathrm{Var}_{0}(g^{\mathrm{e}}) \asymp \norm{g^e(|x|^{3/2})}_{\dot{H}^{1/2}}^2, \quad \mathrm{Var}_{0}(g^{\mathrm{e}}) = \frac{3}{2\pi^2} \norm{g^e(|x|^{3/2})}_{\dot{H}^{1/2}(\mathbb{R}_+)}^2 - \frac{1}{2\pi^2} \norm{g^e(|x|^{1/2})}_{\dot{H}^{1/2}(\mathbb{R}_+)}^2.
	\end{equation}
	
	Finally, we consider the case of $\sng$ being an edge-point adjacent to gap of size $\Delta(\sng) < \Delta_*$. We choose the threshold $\Delta_* \sim 1$ in small enough, such that $\psi(\sng) \sim 1$ by  \eqref{eq:psi_sim_1}. In this case, estimates \eqref{eq:sol_gap} and \eqref{eq:K(h)_kernel} imply that
	\begin{equation} \label{eq:gapsol_ker}
		\other{\mathcal{K}}(w,\other{w}) =\frac{2}{9}\re\biggl[ 
		\frac{\bigl(3 + \gapsol^2 + 4 \gapsol\other{\gapsol} +\other{\gapsol}^2\bigr)(\gapsol-\other{\gapsol})^2}
		{(\gapsol^2 - 1)(\other{\gapsol}^2 - 1)(w-\other{w})^2} 
		- \frac{\bigl(3 + \bar{\gapsol}^2 + 4 \bar{\gapsol}\other{\gapsol} +\other{\gapsol}^2\bigr)(\bar{\gapsol}-\other{\gapsol})^2}
		{(\bar{\gapsol}^2 - 1)(\other{\gapsol}^2 - 1)(w-\other{w})^2}
		\biggr],
	\end{equation}
	where $\gapsol := \gapsol(\sign(\sigma)+2\widehat{\Delta}^{-1}w)$, $\other{\gapsol} := \gapsol(\sign(\sigma)+2\widehat{\Delta}^{-1}\other{w})$, and $\widehat{\Delta} := \widehat{\Delta}(\sng)$ is defined in \eqref{eq:bias_edge}.
	Since for all $\lambda \notin [-1,1]$, $\im \gapsol(\lambda) > 0$, taking the imaginary part of \eqref{eq:gapsol_eq} yields
	\begin{equation} \label{eq:gapsom_im_re_eq}
		-\bigl(\im\gapsol(\lambda)\bigr)^2 + 3\bigl(\re\gapsol(\lambda)\bigr)^2 - 3 = 0, \quad |\lambda| > 1.
	\end{equation}
	On the other hand, for all $\lambda \in [-1,1]$, $\im\gapsol(\lambda) =0$, hence the right-hand side of \eqref{eq:gapsol_ker} is zero.
	
	We now perform a change of variable $\sign(\sigma)\re\gapsol \mapsto t$, $\sign(\sigma)\re\other{\gapsol}\mapsto \other{t}$. It follows from \eqref{eq:gapsom_im_re_eq}, that $\im\gapsol = (3t^2 -3)^{1/2}$, and $\im\other{\gapsol} = (3\other{t}^2 -3)^{1/2}$. Furthermore, by examining the real part of \eqref{eq:gapsol_eq}, we deduce that 
	$(w-\other{w})^2 = \frac{1}{4}\widehat{\Delta}^{2}(4t^3-3t-4\other{t}^3+3\other{t})^2$, and $\mathrm{d}w\mathrm{d}\other{w} = \tfrac{9}{4}\widehat{\Delta}^{2}(4t^2-1)(4\other{t}^2-1)\mathrm{d}t\mathrm{d}\other{t}$. Therefore, \eqref{eq:V_sol} yields \eqref{eq:main2} with $\widehat{\alpha} = \tfrac{\widehat{\Delta}}{2\eta_0}$, and $g_{\widehat{\alpha}}(t) := g\bigl(\sign(\sigma)\widehat{\alpha}(4t^3-3t-1)\bigr)$. Here we used the estimate
	\begin{equation}
		\int\limits_{|t|>1}\bigl(g_{\widehat{\alpha}}(t)\bigr)^2\int\limits_{|\other{t}|>\re\gapsol(N^{-\varepsilon_0/3})}\frac{\mathcal{K}_{\mathrm{gap}}(t,\other{t})}{(t-\other{t})^2}\mathrm{d}\other{t}\mathrm{d}t = \mathcal{O}\biggl( \frac{\re\gapsol(\widehat{\Delta}^{-1}\eta_0)}{\re\gapsol(\widehat{\Delta}^{-1}N^{-\varepsilon_0/3})}\biggr) = \mathcal{O}\bigl( \eta_0^{1/3}N^{\varepsilon_0/9}\bigr),
	\end{equation}  
	which is proved similarly to \eqref{eq:tail_est}.

	To prove the equivalence \eqref{eq:norm_gap}, and the continuity of the variance functional $\mathrm{Var}_{\alp}$  \eqref{eq:Var_converge}, we consider even and odd parts of the function $g_{\widehat{\alpha}}$ separately. Define $\other{V}(g_{\widehat{\alpha}})$ to be the integral on the right-hand side of \eqref{main_gap}. Since $\mathcal{K}_{\mathrm{gap}}(-x,-y)=\mathcal{K}_{\mathrm{gap}}(x,y)$, $\other{V}(g_{\widehat{\alpha}}) = \other{V}(g^{\mathrm{o}}_{\widehat{\alpha}})+\other{V}(g^\mathrm{e}_{\widehat{\alpha}})$.
	
	For the odd part of $g_{\alpha}$, we change the variable $x \mapsto -x$ for $x<-1$, and $y\mapsto -y$ for $y<-1$, to obtain that 
	\begin{equation} \label{eq:gap_odd}
		\other{V}(g^{\mathrm{o}}_{\widehat{\alpha}}) = \frac{3}{2\pi^2}\int\limits_{0}^{+\infty}\int\limits_{0}^{+\infty}\bigl(g^{\mathrm{o}}_{\widehat{\alpha}}(x)-g^{\mathrm{o}}_{\widehat{\alpha}}(y)\bigr)^2 \frac{(x+y)^2 \mathcal{K}_{\mathrm{gap}}(x,y)-(x-y)^2 \mathcal{K}_{\mathrm{gap}}(x,-y)}{(x^2-y^2)^2}\mathrm{d}x\mathrm{d}y,
	\end{equation}
	where we used the fact that 
	\begin{equation}
		\int\limits_{0}^{+\infty}\int\limits_{0}^{+\infty}\frac{\bigl(g^{\mathrm{o}}_{\widehat{\alpha}}(x)\bigr)^2}{(x-y)^2} \mathcal{K}_{\mathrm{gap}}(x,-y)\mathrm{d}x\mathrm{d}y = 0.
	\end{equation}
	Performing the change of variable $x \mapsto \sqrt{\gamma^2 t^2+1}$, $y\mapsto \sqrt{\gamma^2 q^2+1}$ reduces the integral on the right-hand side of \eqref{eq:gap_odd} to 
	\begin{equation}
		\frac{3}{\pi^2}\int\limits_{0}^{\infty}\int\limits_{0}^{\infty} \bigl(g_{\widehat{\alpha}}^\mathrm{o}(\sqrt{\gamma^2t^2+1})-g_{\widehat{\alpha}}^\mathrm{o}(\sqrt{\gamma^2q^2+1})\bigr)^2 \frac{ t^2+q^2}{(t^2-q^2)^2} \mathcal{W}_\mathrm{o}(\gamma t,\gamma q)\mathrm{d}t\mathrm{d}q,
	\end{equation}
	where 
	\begin{equation} \label{eq:odd_gap_P}
		\mathcal{W}_\mathrm{o}(t,q) = 1 -  \frac{3(1 + 4 t^2) (1 + 4 q^2) \bigl( (t (3 + 4 t^2))^2 + (q (3 + 4 q^2))^2\bigr)}{(t^2 + q^2) (3 + 4 t^2 - 4 t q + 4 q^2)^2 (3 + 4 t^2 + 4t q + 4 q^2)^2}.
	\end{equation}
	It is straightforward to check that $\mathcal{W}_\mathrm{o}(q,\other{q}) \sim 1$, which yields \eqref{eq:norm_gap} for the odd part of $g_{\alpha}$. Next, we set the value $\gamma > 0$ to satisfy the equation $8\gamma^3+9\gamma^2 = 2\widehat{\alpha}^{-1}$ (note that for $0 < \Delta <\Delta_*$, $\widehat{\alpha} \in (0, +\infty)$). Under this choice, 
	\begin{equation}
		g_{\widehat{\alpha}}^\mathrm{o}(\sqrt{\gamma^2t^2+1}) = \frac{g\bigl(u_{\gamma}(t)\bigr) - g\bigl(-2\sign(\sigma)\widehat{\alpha} - u_{\gamma}(t)\bigr)}{2}, \quad u_{\gamma}(t) := 2\frac{(1 + 4 x^2 \gamma^2) \sqrt{1 + x^2 \gamma^2} - 1}{\sign(\sigma)(8\gamma^3+9\gamma^2)}.
	\end{equation}
	The function $u(t)$ satisfies $0\le u_{\gamma}(t) \lesssim \max\{t^2,t^3\}$, $|\partial_t u_{\gamma}(t)| \lesssim \max\{t,t^2\}$ for all $t \in [0,+\infty)$, uniformly in $\gamma\in(0,+\infty)$. Therefore, since the support of $g$ is compact, we can use the dominated convergence to take the limit $\gamma \to \gamma_0$, where $\gamma_0\in[0,+\infty]$ satisfies $8\gamma_0^3+9\gamma_0^2 = 2\alpha^{-1}$, and $\alpha = \lim_{N\to \infty}\widehat{\alpha}$.
	
	If $\gamma_0\in (0,+\infty)$, we deduce that $\lim_{N\to\infty}\other{V}(g^{\mathrm{o}}_{\widehat{\alpha}}) = \other{V}(g^{\mathrm{o}}_{\alpha})$. 
	
	If $\gamma_0 = 0$, then 
	\begin{equation}
		\lim\limits_{\gamma\to 0} \mathcal{W}_\mathrm{o}(\gamma t,\gamma q) = \frac{1}{3}, \quad \lim\limits_{\gamma\to 0}u_{\gamma}(t) = \sign(\sigma)t^2,
	\end{equation}
	hence we recover $\lim_{N\to\infty}\other{V}(g^{\mathrm{o}}_{\widehat{\alpha}}) = \tfrac{1}{2}\mathrm{Var}_{\infty}(g)$.
	
	If $\gamma_0 = +\infty$, then 
	\begin{equation}
		\lim\limits_{\gamma\to +\infty} \mathcal{W}_\mathrm{o}(\gamma t,\gamma q) = \frac{3 t^2 q^2 (t^4 - t^2 q^2 + q^4)}{(t^4 + t^2 q^2 + q^4)^2}, \quad \lim\limits_{\gamma\to 0}u_{\gamma}(t) = \sign(\sigma)t^3,
	\end{equation}
	hence we recover $\lim_{N\to\infty}\other{V}(g^{\mathrm{o}}_{\widehat{\alpha}}) = \mathrm{Var}_{0}(g^{\mathrm{o}})$. This proves \eqref{eq:Var_converge} for the odd part of $g_{\widehat{\alpha}}$.

%	Furthermore, integrating the second term in \eqref{eq:odd_gap_P} separately and performing the change of variable $3q+4q^3 \mapsto u$, $3\other{q}+4\other{q}^3 \mapsto \other{u}$ yields
%	\begin{equation}
%		\other{V}(g_{\widehat{\alpha}}^o) = \frac{3}{\pi^2}\norm{g_{\widehat{\alpha}}^o(\sqrt{q^2+1})}_{\dot{H}^{1/2}}^2 - \frac{1}{\pi^2}\norm{g^{asym}(\sqrt{q^2+1})}_{\dot{H}^{1/2}}^2,
%	\end{equation}
%	where $g^{asym}(x) := g\bigl(\tfrac{1}{2}(x^2+\eta_0^{-2}\widehat{\Delta}^2)^{1/2} - \tfrac{1}{2}\sign(\sigma)\eta_0^{-1}\widehat{\Delta}\bigr) - g\bigl(-\tfrac{1}{2}(x^2+\eta_0^{-2}\widehat{\Delta}^2)^{1/2} - \tfrac{1}{2}\sign(\sigma)\eta_0^{-1}\widehat{\Delta}\bigr)$.
	
	For the even part of $g_{\widehat{\alpha}}$, we perform the change of variable $x\sqrt{x^2-1}\mapsto \gamma t$, $y\sqrt{y^2-1}\mapsto \gamma s$ to obtain
	\begin{equation}
		\other{V}(g^{\mathrm{\mathrm{e}}}_{\widehat{\alpha}}) = \frac{3}{4\pi^2}\int\limits_{0}^{\infty}\int\limits_{0}^{\infty} \biggl(g_{\alpha}^\mathrm{e}\bigl(\sqrt{\tfrac{1}{2}+\tfrac{1}{2}\sqrt{4(\gamma t)^2+1}}\bigr)-g_{\alpha}^e\bigl(\sqrt{\tfrac{1}{2}+\tfrac{1}{2}\sqrt{4(\gamma q)^2+1}}\bigr)\biggr)^2 \frac{ \mathcal{W}_\mathrm{e}(\gamma t,\gamma q)}{(t-q)^2} \mathrm{d}t\mathrm{d}q,
	\end{equation}
	where 
	\begin{equation}
		\begin{split}
			\mathcal{W}_\mathrm{e}(t,q) =&~ \frac{(\sqrt{1 + 4 t^2} + \sqrt{1 + 4 q^2})^2\bigl(2(1+8t^2)^2 +2(1+8q^2)^2+64 t^2 q^2 (7 + 16 t^2 + 16 q^2)-3 \bigr)}{\sqrt{1 + 4 t^2} \sqrt{1 + 4 q^2} (t+q)^2\bigl(1 + 8 t^2 + 8 q^2 + 
				2 (\sqrt{1 + 4 t^2} + \sqrt{1 + 4 q^2})^2\bigr)^2}\\
			&+\frac{(\sqrt{1 + 4 t^2} + \sqrt{1 + 4 q^2})^2 (-1 + 24 t^2 + 128 t^4 + 24 q^2 + 128 q^4)}{(t+q)^2\bigl(1 + 8 t^2 + 8 q^2 + 
				2 (\sqrt{1 + 4 t^2} + \sqrt{1 + 4 q^2})^2\bigr)^2}.
		\end{split}
	\end{equation}
	It can be checked, that $\mathcal{W}_\mathrm{e}(t,q) \sim 1$, establishing the equivalence \eqref{eq:norm_gap}.
	We now set $\gamma$ to be the unique positive solution of the equation $36\gamma^2(8 + 9\gamma^{1/2})= \widehat{\alpha}^{-1}$.
	Similarly to the treatment of the odd part of $g_{\widehat{\alpha}}$, we defined
	\begin{equation}
		v_\gamma(t) := \frac{(8+9\gamma^{1/2})\bigl((2\sqrt{1+4\gamma^2t^2}-1)\sqrt{\tfrac{1}{2}+\tfrac{1}{2}\sqrt{1+4\gamma^2t^2}}-1\bigr)}{36\gamma^2},
	\end{equation}
	and observe that 
	\begin{equation}
		\begin{split}
			\lim\limits_{\gamma\to 0}\mathcal{W}_\mathrm{e}(\gamma t,\gamma q) &= \frac{8}{3}\frac{t^2+q^2}{(q+t)^2}, \quad \lim\limits_{\gamma\to 0} v_\gamma(t) = t^2,\\
			 \lim\limits_{\gamma\to +\infty}\mathcal{W}_\mathrm{e}(\gamma t,\gamma q) &= 2\frac{t^4+2tq^3+2qt^3+q^4}{(t^2+tq+q^2)^2}, \quad\lim\limits_{\gamma\to 0} v_\gamma(t) = t^{3/2}.
		\end{split}
	\end{equation}
	By using dominated convergence, we deduce that $\lim\limits_{N\to\infty} \other{V}(g^{\mathrm{e}}_{\widehat{\alpha}}) = \frac{1}{2}\mathrm{Var}_{\infty}(g)$ if $\alpha = \infty$, and for $\alpha<\infty$,  $\lim\limits_{N\to\infty} \other{V}(g^{\mathrm{e}}_{\widehat{\alpha}}) = \mathrm{Var}_{\alpha}(g^{\mathrm{e}})$. 
	 This concludes the proof of Proposition \ref{prop:main2}.
\end{proof}
%
%\com
%This section is work in progress, hence the chaotic presentation. Also, some of the computation can be significantly shortened in the later drafts, but I keep them for the sake of completion.
%
% 
\section{Perturbation Estimates} \label{sec:pert} %This section is clean
In this section we collect the perturbative estimates which are used throughout the proof. For convenience of notation in the sequel, the quantities $\Delta$, $\widehat{\Delta}$, $\sigma$, $\psi$, the vectors  $\vv$, $\vect{f}$, $\vect{p}$, $\vect{b}$, $\vect{b}^{\ell}$, $\vect{r}$, and the operator $F$ are assumed to be evaluated at the point $\sng$ unless explicitly stated otherwise. 

In the following claim we collect estimates on $\dM(z,\zeta)$ used throughout the proof.
\begin{lemma} 
	Uniformly for all $z,\zeta \in \D\cap\mathbb{H}$, the function $\dM(z,\zeta)$ satisfies the bound 
	\begin{equation} \label{eq:dM_lower_bound}
		|\dM(z,\zeta)| \gtrsim \M(z,|\zeta-z|),
	\end{equation}
	where $M(z,w)$ is the control parameter defined in \eqref{eq:better_dM_bound}.
	Consequently, for all $z,\zeta \in \D$, we have the upper bound
	\begin{equation} \label{eq:dm_dM_bound}
		\norm{\m(\zeta)-\m(z)}_{\infty} \lesssim |\dM(z,\zeta)|.
	\end{equation}
	In particular, for all $z\in\D$,
	\begin{equation} \label{eq:dM_sng_bound}
		|\dM(\sng,z)| \gtrsim |z-\sng|, \quad \text{and}\quad |\dM(\sng,z)|^2 \gtrsim |z-\sng|.
	\end{equation}
\end{lemma}
\begin{proof}
	First, we prove \eqref{eq:dM_lower_bound}.
	Let $\mu_k(z)$ be the coefficients of the cubic equation \eqref{dm_cubic}. It follows from \eqref{eq:cubic_mus}, that there exists an $N$-independent constant $c_{\mu} > 0$, such that $|\mu_0(z,\zeta)| \ge c_{\mu}$ for all $z,\zeta \in \D$. 
	
	Define $C_{\sng,c_0} := \inf\{ |\dM(z,\zeta)|/M(z,|\zeta-z|):\, z,\zeta \in \D, \,z\neq\zeta \}$. We can assume that $C_{\sng,c_0} < \frac{1}{3}\min\{1,c_{\mu}\}$, otherwise the result is trivial. Therefore, there exists a pair of points $z\neq \zeta \in \D$, such that $|\dM(z,\zeta)|< \frac{1}{2}C_{\sng,c_0} M(z,|\zeta-z|)$. It follows readily from the comparison relations of Proposition \ref{prop_scaling}, that for such $z$ and $\zeta$,
	\begin{equation}
		|\mu_3(z)\dM^3+\mu_2(z)\dM^2+\mu_1(z)\dM| < \tfrac{1}{2}c_{\mu}|z-\zeta|.
	\end{equation}
	On the other hand, $|\mu_3(z)\dM^3+\mu_2(z)\dM^2+\mu_1(z)\dM| = |\mu_0||z-\zeta| \ge c_{\mu}|z-\zeta|$, which is a contradiction. Therefore, $C_{\sng,c_0} \ge \min\{1/3,c_{\mu}/3\} \sim 1$, which implies \eqref{eq:dM_lower_bound}.
	
	Estimate \eqref{eq:dm_dM_bound} follows trivially from the expansion \eqref{eq:dM_dir}, the lower bounds \eqref{eq:dM_lower_bound}, and the fact that $M(z,|\zeta-z|) \gtrsim |\zeta-z|$ for all $z\in \D$ by the definition \eqref{eq:better_dM_bound}.
	Estimates \eqref{eq:dM_sng_bound} follow from the fact that $\mu_3(\sng) \sim 1$, $\mu_2(\sng) = \sigma \sim \Delta_0^{1/3}$, and $\mu_1(\sng) = 0$.
\end{proof}

\begin{prop} \label{prop_eigval}
	Let $z, \zeta \in \D\backslash\mathbb{R}$ be two spectral parameters, then the eigenvalue $\eigB(z,\zeta)$ of $\stab(z,\zeta)$ with the smallest modulus admits the following estimates
	%	\begin{equation} \label{eig_estimate}
		%		\eigB(z,\zeta) = - \sigma\bigl(\dM(z)+\dM(\zeta)\bigr) - \psi \bigl(\dM(z)^2+\dM(\zeta)^2\bigr) - (\psi+\frac{\sigma^2}{\langle\vect{f}^2\rangle})\dM(\zeta)\dM(z) + \mathcal{O}\bigl(|w|+|\other{w}| + |\eta|^{1/3} + |\other{\eta}|^{1/3} \bigr),
		%	\end{equation}
	\begin{equation} \label{eig_estimate}
		\begin{split}
			\eigB(z,\zeta) =& - \frac{\sigma}{\langle\vect{f}^2\rangle}\bigl(\dM(z)+\dM(\zeta)\bigr) - \frac{\psi}{\langle\vect{f}^2\rangle} \bigl(\dM(z)^2+\dM(z)\dM(\zeta)+\dM(\zeta)^2\bigr) -\frac{\sigma^2}{\langle\vect{f}^2\rangle^2}\dM(z)\dM(\zeta) \\
			&+ \mathcal{O}\bigl(|z-\sng| + |\zeta-\sng| %+ (\Delta+|w| + |\other{w}|+|\eta|+|\other{\eta}| )^{1/3}(|\eta|+|\other{\eta}|)^{1/3} 
			\bigr),
		\end{split}
	\end{equation}
	where $\dM(z) := \dM(\sng,z)$ is defined in \eqref{eq:dM_def}.
\end{prop}
\begin{proof}
	We apply Lemma \ref{lemma_ncpt} with $z_0 =\zeta_0 = \sng$. It follows from \eqref{beta_est} that $\eigB_0 := \eigB(\sng,\sng) = 0$. Define $\diff\m := \m(z)\m(\zeta) - \m_0^{2}$, where $\m_0 := \m(\sng)$, then \eqref{eq:dM_dir} implies
	\begin{equation} \label{d2m}
		\diff\m = (\dM +\other{\dM})\m_0\vect{b} + \dM\other{\dM}\vect{b}^2 + (\dM^2+\other{\dM}^2)\m_0\vect{r} + \mathcal{O}\bigl(|z-\sng| + |\zeta-\sng| %+ (|w| + |\other{w}|+|\eta|+|\other{\eta}|)^{1/3}(|\eta|+|\other{\eta}|)^{1/3}
		\bigr),
	\end{equation}
	where  $\dM :=  \dM(\sng,z)$, $\other{\dM} :=\dM(\sng,\zeta)$, $\vect{r} := \vect{r}(\sng)$ are defined in \eqref{eq:dM_def}, and $\vect{b}:=\vect{b}(\sng)$ is given in \eqref{b_def}.
	
	Estimate \eqref{dm_bound} implies that $\norm{\diff\m}_\infty \lesssim (|z-\sng|+|\zeta-\sng|)^{1/3}$, hence plugging \eqref{d2m} into \eqref{eig_pert} with $z_0 = \zeta_0 := \sng$ yields
	\begin{equation} \label{eig_tr_expr}
		\begin{split}
			\eigB(z,\zeta) =& -(\dM +\other{\dM})\Tr\bigl[\frac{\vect{b}}{\m_0}\Pi_0\bigr]
			-\dM\other{\dM}\Tr\bigl[\frac{\vect{b}^2}{\m_0^2}\Pi_0\bigr]
			- (\dM + \other{\dM})^2\Tr\bigl[ \frac{\vect{b}}{\m_0}\frac{Q_0}{\stab_0}(1-\stab_0)\frac{\vect{b}}{\m_0}\Pi_0\bigr] \\
			&-(\dM^2+\other{\dM}^2)\Tr\bigl[\frac{\vect{r}}{\m_0}\Pi_0\bigr]
			+ \mathcal{O}\bigl(|z-\sng| + |\zeta-\sng|\bigr),
		\end{split}
	\end{equation}
	where $\Pi_0$, $Q_0$ are the spectral projections corresponding to $\stab_0 := \stab(\sng,\sng)$, and all $\vect{b},\vect{r},\m_0$ are interpreted as diagonal matrices. We now compute the traces in \eqref{eig_tr_expr} individually.	
	
	First, if follows from \eqref{b_vect} that $\Pi_0 = \langle \vect{f}^2\rangle^{-1}(|\m_0|\vect{f})(|\m_0|^{-1}\vect{f})^*$, where $\vect{f} := \vect{f}(\sng)$ is defined in \eqref{f_def}, hence
	\begin{equation} \label{sigma_tr}
		\Tr\bigl[\frac{\vect{b}}{\m_0}\Pi_0\bigr] = \frac{\langle \vect{f}^2\vect{p},|\m_0|^{-1}\vect{b} \rangle}{\langle\vect{f}^2\rangle} = \frac{\sigma}{\langle\vect{f}^2\rangle},
	\end{equation}
	where $\sigma := \sigma(\sng)$, $\vect{p} := \vect{p}(\sng)$ are defined in \eqref{sigma_def}.
	
	Similarly, we compute
	\begin{equation} \label{traces1}
		\Tr\bigl[\frac{\vect{b}^2}{\m_0^2}\Pi_0\bigr] = \frac{\langle \vect{f}^4\rangle}{\langle\vect{f}^2\rangle}, \quad \Tr\bigl[\frac{\vect{r}}{\m_0}\Pi_0\bigr] = 
		\frac{\langle \vect{f}^2\vect{p}, \frac{(1-\vv\vv^*)}{1-F}[\vect{f}^2\vect{p}] \rangle}{\langle\vect{f}^2\rangle} = 
		\frac{\psi + \langle \vect{f}^4\rangle - \frac{\sigma^2}{\langle \vect{f}^2\rangle}}{2\langle\vect{f}^2\rangle}, 
	\end{equation}
	\begin{equation} \label{traces2}
		\Tr\bigl[\frac{\vect{b}}{\m_0}\frac{Q_0}{\stab_0}(\stab_0-1)\frac{\vect{b}}{\m_0}\Pi_0\bigr] =
		-\frac{\langle \vect{f}^2\vect{p}, \frac{(1-\vv\vv^*)F}{1-F}[\vect{f}^2\vect{p}] \rangle}{\langle\vect{f}^2\rangle} = 
		\frac{-\psi + \langle \vect{f}^4\rangle - \frac{\sigma^2}{\langle \vect{f}^2\rangle}}{2\langle\vect{f}^2\rangle}.
	\end{equation}
	Plugging \eqref{traces1} and \eqref{traces2} into \eqref{eig_tr_expr} concludes the proof of Proposition \ref{prop_eigval}.
\end{proof}

\begin{lemma}  \label{lemma_m'}
	Let $z$ be a spectral parameter in $\D\backslash\mathbb{R}$, and let $\beta(z)$ be as defined in \eqref{beta_est}, then 
	\begin{equation}
		\begin{split}
			\langle \vect{f}^2\rangle\eigB(z)\frac{\m'(z)}{\m(z)} =& \frac{\vect{b}}{\m_0}\bigl( \pi + 2\dM(\sng,z) \bigl\langle|\m_0|, \frac{1-\vv\vv^*}{1-F}[\vect{f}^2\vect{p}]  \bigr\rangle \bigr)\\ &+ \dM(\sng,z)\bigl(2\pi\vect{p} \frac{1-\vv\vv^*}{1-F}[\vect{f}^2\vect{p}] - \frac{\pi\vect{b}^2}{|\m_0|^2} \bigr) +\mathcal{O}\bigl(\Delta_0^{1/3}|\dM(\sng,z)| + |\dM(\sng,z)|^2 %+ |z-\sng|
			\bigr),
		\end{split}
	\end{equation}
	where $\dM(\sng,z)$ is defined in \eqref{eq:dM_def}, and $\vect{b},\vect{p},\vv, F$, and $\sigma$ are evaluated at $\sng$.
\end{lemma}
\begin{proof}
	Differentiating the Dyson equation \eqref{VDE} yields the identity $\m'(z) = \stab(z,z)^{-1}[\m(z)^2]$.
	It follows form \eqref{eq:dM_dir} and \eqref{eq:dM_sng_bound}, that 
	\begin{equation}
		\m(z)^2 = \m_0^2 + 2\dM\m_0\vect{b} + \mathcal{O}\bigl(|\dM|^2 %+ |z-\sng|
		\bigr),
	\end{equation}
	where $\dM := \dM(\sng,z)$. 
	Therefore, by decomposing $\stab^{-1} = \eigB^{-1}\Pi + \stab^{-1}Q$, and using estimates \eqref{Pi_pert}, \eqref{Q/B_pert}, \eqref{eq:inbstabQ} we obtain
	\begin{equation} \label{m'_est}
		\begin{split}
			\eigB(z)\m'(z) =&\, \Pi_0[\m_0^2] + 2\dM\Pi_0[\m_0\vect{b}] - 2\dM\Pi_0\frac{\vect{b}}{\m_0}\frac{\stab_0-1}{\stab_0}Q_0[\m_0^2]\\
			&+ 2\dM\frac{Q_0}{\stab_0}\frac{\vect{b}}{\m_0}\Pi_0[\m_0^2]  + \mathcal{O}(|\beta(z)|+|\dM|^2%+ |z-\sng|
			),
		\end{split}
	\end{equation}
	where $\stab_0 := \stab(\sng, \sng)$, $\Pi_0 := \Pi(\sng,\sng)$, and $Q_0 := 1 -\Pi_0$.
	Observe that \eqref{eig_estimate} and \eqref{eq:sigma_comp} imply $|\beta(z)| \lesssim \Delta_0^{1/3}|\dM| + |\dM|^2% + |z-\sng|
	$. We compute 
	\begin{equation} \label{Pim^2}
		\Pi_0[\m_0^2] = \frac{\pi}{\langle \vect{f}^2\rangle}\vect{b}, \quad \Pi_0[\vect{b}\m_0] = \frac{\langle \vect{f}^2\vect{p},|\m_0|\rangle}{\langle \vect{f}^2\rangle}\vect{b}.
	\end{equation}
	Furthermore, we have $\stab_0^{-1}Q_0 = |\m_0|(1-F)^{-1}(1-\vv\vv^*)|\m_0|^{-1}$ and $\stab_0-1 = -|\m_0|F|\m_0|^{-1}$, therefore
	\begin{equation} \label{Qm^2}
		\frac{Q_0}{\stab_0}\frac{\vect{b}^2}{\m_0} = |\m_0|\frac{1-\vv\vv^*}{1-F}[\vect{f}^2\vect{p}], \quad \Pi_0\frac{\vect{b}}{\m_0}\frac{\stab_0-1}{\stab_0}Q_0[\m_0^2] = - \langle \vect{f}^2\rangle^{-1}\langle |\m_0|, \frac{1-\vv\vv^*}{1-F}F[\vect{f}^2\vect{p}]\rangle \vect{b}
	\end{equation}
	Plugging \eqref{Pim^2} and \eqref{Qm^2} into \eqref{m'_est} yields
	\begin{equation} \label{eq:m'_est}
		\begin{split}
			\langle \vect{f}^2\rangle\eigB(z)\m'(z) =& \,\vect{b}\bigl(\pi + 2\dM\langle |\m_0|, \frac{1-\vv\vv^*}{1-F}[\vect{f}^2\vect{p}]\rangle + \frac{2\pi\sigma \dM}{\langle \vect{f}^2\rangle} \bigr)\\ 
			&+ 2\pi\dM|\m_0|\frac{1-\vv\vv^*}{1-F}[\vect{f}^2\vect{p}] + \mathcal{O}\bigl( \Delta_0^{1/3}|\dM| + |\dM|^2% + |z-\sng|
			\bigr),
		\end{split}
	\end{equation}
	which, after dividing by $\m(z) = \m_0 + \dM\vect{b} + \mathcal{O}(|\dM|^2)$, concludes the proof of Lemma \ref{lemma_m'}.
	
\end{proof}

\begin{prop} \label{prop:eigB} Let $z,\zeta \in \D\cap\mathbb{H}$, then the eigenvalues $\eigB(z,\zeta)$ and $\eigB(\bar z,\zeta)$ satisfy
	\begin{equation} \label{eq:eigB}
		\eigB(z,\zeta) = \frac{\pi}{\langle \vect{f}(z)^2\rangle}\frac{\zeta-z}{\dM(z,\zeta)}\bigl(1+\mathcal{O}(|\dM(z,\zeta)|+ \rho(z) +\rho(z)^{-1}|\im z|)\bigr),
	\end{equation}
	\begin{equation} \label{eq:eigBbar}
		\eigB(\bar z,\zeta) = \frac{\pi}{\langle \vect{f}(z)^2\rangle}\frac{\zeta-\bar z}{\dM( z,\zeta)+2\I\rho(z)}\bigl(1+\mathcal{O}(|\dM(z,\zeta)|+\rho(z) +\rho(z)^{-1}|\im z|)\bigr).
	\end{equation}
\end{prop}
\begin{proof}
	First, we focus on $\eigB(z,\zeta)$. Taking the difference of \eqref{VDE} at $z$ and $\zeta$, we obtain
	\begin{equation} \label{eq_diffm}
		\stab(z,\zeta)\bigl[\m(\zeta)-\m(z)\bigr]  = (\zeta-z)\m(z)\m(\zeta).
	\end{equation}
	Let $\eta := \im z$.
	Applying $\stab(z,\zeta)^{-1}$ and multiplying both sides of \eqref{eq_diffm} by $\vect{b}^\ell(z)$ yields
	\begin{equation}
		\dM(z,\zeta)\langle\vect{b}^\ell(z),\vect{b}(z)\rangle = (\zeta-z)\eigB(z,\zeta)^{-1}\bigl\langle\vect{b}^\ell(z),\Pi(z,\zeta)[\m(z)\m(\zeta)]\bigr\rangle \bigl( 1 + \mathcal{O}(|\eigB(z,\zeta)|)\bigr),
	\end{equation}
	where we used the decomposition $\stab^{-1} = \eigB^{-1}\Pi + \stab^{-1}Q$, and the estimate \eqref{eq:inbstabQ}. 
	Furthermore, it follows from \eqref{Pi_pert} and \eqref{eq:dm_dM_bound} that $\Pi(z,\zeta) = \Pi(z,z) + \mathcal{O}(|\dM(z,\zeta)|)$, and $\langle \vect{b}^\ell(z), \m(z)\m(\zeta)\rangle = \pi + \mathcal{O}(|\dM(z,\zeta)|+\rho(z) +\rho(z)^{-1}|\eta| )$ therefore
	\begin{equation}
		\eigB(z,\zeta)^{-1} = \frac{\langle \vect{f}(z)^2\rangle}{\pi}\frac{\dM(z,\zeta)}{\zeta-z}\bigl(1+\mathcal{O}(|\eigB(z,\zeta)|+|\dM(z,\zeta)|+\rho(z) +\rho(z)^{-1}|\eta|)\bigr),
	\end{equation}
	where we additionally used $\langle\vect{b}^\ell(z),\vect{b}(z)\rangle = \langle \vect{f}(z)^2\rangle + \mathcal{O}(\rho(z) + \rho(z)^{-1}|\eta|)$. Note that by estimates \eqref{eq:eig_comp} and \eqref{eig_pert}, $|\eigB(z,\zeta)| \lesssim \rho(z) +\rho(z)^{-1}|\eta| + |\dM(z,\zeta)|$. Therefore, \eqref{eq:eigB} is established. 
	
	We turn to analyze $\eigB(\bar{z},\zeta)$. We substitute $z:=\bar z$ in  equation \eqref{eq_diffm}, apply $\stab(\bar z,\zeta)^{-1}$ to both sides, and multiply by $\vect{b}^\ell(z)$, to obtain
	\begin{equation} \label{eq:eigbbar_expr}
		\langle\vect{b}^\ell( z), \m(\zeta)-\m(\bar{z}) \rangle  = (\zeta-\bar z)\eigB(\bar z,\zeta)^{-1}\bigl\langle\vect{b}^\ell(z),\Pi(\bar z,\zeta)[\m(\bar z)\m(\zeta)]\bigr\rangle \bigl( 1 + \mathcal{O}(|\eigB(\bar z,\zeta)|)\bigr).
	\end{equation}
	Again, it follows from \eqref{eig_pert} and \eqref{eigF_1}, that $\eigB(\bar{z},\zeta) \lesssim \rho(z) + \rho(z)^{-1}|\eta| + |\dM(z,\zeta)|$.
	
	Observe that $\m(z)-\m(\bar{z}) = 2\I \rho(z) |\m(z)|\vect{f}(z)$, and $\langle\vect{b}^\ell(z),|\m(z)|\vect{f}(z) \rangle = \langle\vect{b}^\ell(z),\vect{b}(z)\rangle + \mathcal{O}(\rho(z)+\rho(z)^{-1}|\eta|)$, hence 
	\begin{equation}
		\langle\vect{b}^\ell( z), \m(\zeta)-\m(\bar{z}) \rangle = \langle\vect{b}^\ell(z),\vect{b}(z)\rangle\bigl(\dM(z,\zeta) + 2\I\rho(z) \bigr) + \mathcal{O}(\rho(z)^2+|\eta|).
	\end{equation}
	Moreover, by \eqref{v_def}, \eqref{f_def}, and \eqref{b_vect}, one readily obtains that $|\im \langle\vect{b}^\ell( z), \m(\zeta)-\m(\bar{z}) \rangle| \gtrsim \rho(z)$. Therefore, using $\langle\vect{b}^\ell(z),\vect{b}(z)\rangle = \langle \vect{f}(z)^2\rangle(1 + \mathcal{O}(\rho(z)+\rho(z)^{-1}|\eta|))$, we conclude
	\begin{equation} \label{eq:dM+rho}
		\langle\vect{b}^\ell( z), \m(\zeta)-\m(\bar{z})\rangle = \langle\vect{f}(z)^2\rangle\bigl(\dM(z,\zeta) + 2\I\rho(z) \bigr) \bigl(1 + \mathcal{O}(\rho(z)+\rho(z)^{-1}|\eta|)\bigr).
	\end{equation}
	Combining estimates \eqref{eq:eigbbar_expr}, \eqref{eq:dM+rho}, and $\Pi(\bar z, \zeta) = \Pi(z,z) + \mathcal{O}(\rho(z) + |\dM(z,\zeta)|)$, we arrive at \eqref{eq:eigBbar}, thus concluding the proof of Proposition \ref{prop:eigB}.
	%	\begin{equation}
		%		\eigB(\bar z,\zeta)^{-1} = \frac{\langle \vect{f}(z)^2\rangle}{\pi}\frac{\dM+2\I\rho}{\zeta-\bar z}\bigl(1+\mathcal{O}(|\eigB(\bar z,\zeta)|+|\dM|+|z-\zeta|+\rho +\rho^{-1}|\eta|)\bigr),
		%	\end{equation}	
	%	\begin{equation}
		%		\langle \vect{b}^\ell(z), \m(\zeta) -\m(z) \rangle = \langle \vect{b}^\ell(z), \vect{b}_0\rangle \bigl(\dM(\zeta)-\dM( z)\bigr) + \langle Q_0^*[\vect{b}^\ell(z)], \m(\zeta) -\m(z) \rangle
		%	\end{equation}
	%	By \eqref{Pi_pert}, we have $Q_0^*[\m(\zeta) -\m(z)] = \mathcal{O}(|\dM(z)|+|z-\sng|)$, therefore
	%	\begin{equation}
		%		\langle \vect{b}^\ell(z), \m(\zeta) -\m(z) \rangle = \langle \vect{b}^\ell(z), \vect{b}_0\rangle \bigl(\dM(\zeta)-\dM( z)\bigr) + \mathcal{O}\bigl((|\dM(z)|+|z-\sng|)(|\dM(z,\zeta)|+|z-\zeta|)\bigr).
		%	\end{equation}
	%	Furthermore, by \eqref{eq:dM_dir} and \eqref{eq:dM_lower_bound}, $|\langle\vect{b}^\ell(z), \m(\zeta) - \m(z) \rangle| \gtrsim |\dM(z,\zeta)|$ for all $z,\zeta \in \D\cap\mathbb{H}$, hence
	%	\begin{equation}
		%		\langle \vect{b}^\ell(z), \m(\zeta) -\m(z) \rangle = \langle \vect{f}^2(\sng)\rangle \bigl(\dM(\zeta)-\dM( z)\bigr)\bigl(1 + \mathcal{O}\bigl(|\dM(z)|+|z-\sng|\bigr)\bigr).
		%	\end{equation}
\end{proof}

\section{Proof of Lemma \ref{lemma:variance}} \label{sec:variance_lemma_proof}
The proof of Lemma \ref{lemma:variance} consists of three technical steps. First, we use the perturbative formulas obtained in Section \ref{sec:pert} to estimate the kernel $\mathcal{K}(z,\zeta)$ defined in \eqref{kernel_K} in terms of the function $\dM(\sng, z)$ given in \eqref{eq:dM_def}. In the second step, we use the bounds on the kernel obtained in the first step to express the variance $V(f)$ as a double integral over a subset of the real line via integration by parts and dominated convergence. The third and final step consists of using the approximation \eqref{eq:dM_approx} to obtain the formula \eqref{eq:V_sol} for $V(f)$. We now present the steps in detail.

\subsection{Integral Kernel $\mathcal{K}(z,\zeta)$}
\begin{lemma} \label{lemma:K_crude} 
	For all $z,\zeta \in \D$, the kernel $\mathcal{K}(z,\zeta)$ satisfies the estimate
	\begin{equation}
		\begin{split} \label{eq:K_est}
			\mathcal{K}(z,\zeta) =&~ 2\pi^2\frac{(\sigma + \psi\dM + \psi\other{\dM} )^2 + 2\psi^2\dM\other{\dM}}{\langle\vect{f}^2\rangle^4\beta(z)\beta(\zeta)\beta(z,\zeta)^2} +\mathcal{O}\bigl(\frac{(\Delta_0^{1/3}+|\dM|+|\other{\dM}|)^2(|\dM|+|\other{\dM}|)}{\beta(z)\beta(\zeta)\beta(z,\zeta)^2} \bigr),
		\end{split}
	\end{equation}
	where $\dM := \dM(\sng,z)$, $\other{\dM} := \dM(\sng,\zeta)$ are as defined in \eqref{eq:dM_def}.
	Moreover, the bound
	\begin{equation} \label{eq:K_crude_bound}
		\bigl\lvert\mathcal{K}(z,\zeta)\bigr\rvert \lesssim |\eigB(z,z)\eigB(\zeta,\zeta) \eigB(z,\zeta)^2|^{-1}\bigl( (\Delta_0 + \dis(z) +\dis(\zeta) )^{2/3} %+ |\eigB(z,\zeta)| 
		\bigr),
	\end{equation}
	holds uniformly in $z,\zeta \in \D$.	
\end{lemma}
\begin{proof}[Proof of Lemma \ref{lemma:K_crude}]
Recall from \eqref{kernel_K} that the kernel $\mathcal{K}(z,\zeta)$ is defined as
\begin{equation} \label{K_def}
	\mathcal{K}(z,\zeta) = \Tr\bigl[2\frac{\m'}{\m}\stab^{-1}S\m\other{\m}'\stab^{-1}  - S \m'\other{\m}'
	\bigr] + \frac{1}{2}\frac{\partial^2}{\partial z\partial\zeta}\left\langle\overline{\m\other{\m}}, \Cmlnt^{(4)}\m\other{\m}\right\rangle,
\end{equation}
where $\stab:= \stab(z,\zeta)$ is the stability operator given in \eqref{stab_def}, and $\m := \m(z)$, $\other{\m} := \m(\zeta)$. We adhere to this notation throughout the rest of the section.

Plugging the decomposition $\stab^{-1} = \eigB^{-1}\Pi + \stab^{-1}(1-\Pi)$ into \eqref{K_def}, we obtain
\begin{equation} \label{K_traces}
	\begin{split}
		\mathcal{K}(z,\zeta) =&~ 
		2\eigB^{-2} \Tr\bigl[ \frac{\m'}{\m}\Pi\bigr]\Tr\bigl[\frac{\other{\m}'}{\other{\m}}\Pi\bigr] 
		+ 2\eigB^{-1} \bigl( \Tr\bigl[\frac{\m'}{\m}\Pi\frac{\other{\m}'}{\other{\m}}\frac{Q}{\stab} + \frac{\m'}{\m}\frac{Q}{\stab}\frac{\other{\m}'}{\other{\m}}\Pi - \frac{\m'\other{\m}'}{\m\other{\m}}\Pi\bigr] \bigr)\\
		&+ \Tr\bigl[2\frac{\m'}{\m}\frac{Q}{\stab}\frac{\other{\m}'}{\other{\m}}\frac{Q}{\stab} - 2\frac{\m'\other{\m}'}{\m\other{\m}}\frac{Q}{\stab} - S \m'\other{\m}'\bigr] + \frac{1}{2}\frac{\partial^2}{\partial z\partial\zeta}\left\langle\overline{\m\other{\m}}, \Cmlnt^{(4)}\m\other{\m} \right\rangle,
	\end{split}
\end{equation}
where $\eigB := \eigB(z,\zeta)$ is the destabilizing eigenvalue of $\stab(z,\zeta)$, $\Pi:= \Pi(z,\zeta)$ is the corresponding eigenprojector, and $Q := Q(z,\zeta) = 1-\Pi(z,\zeta)$.

We begin by computing the $\eigB^{-2}$ term. Combining \eqref{Pi_pert}, \eqref{sigma_tr}, \eqref{traces1}, \eqref{traces2} and Lemma \ref{lemma_m'}, we deduce that 
%\begin{equation}
%	\begin{split}
	%		\langle \vect{f}^2\rangle\eigB(z)\Tr\bigl[\Pi\frac{\m'}{\m}\bigr] =& \frac{\pi\sigma}{\langle\vect{f}^2\rangle}
	%%		+ 2\sigma\dM(z) \bigl\langle|\m_0|, \frac{1-\vv\vv^*}{1-F}[\vect{f}^2\vect{p}]  \bigr\rangle
	%		+ 2\pi\dM\Tr\bigl[\Pi_0\vect{p} \frac{1-\vv\vv^*}{1-F}[\vect{f}^2\vect{p}]\bigr]
	%		-\pi \dM\Tr\bigl[\Pi_0\frac{\vect{b}^2}{|\m_0|^2} \bigr]\\
	%		& + \pi(\dM+\other{\dM})\Tr\bigl[ \frac{(2-\stab_0)Q_0}{\stab_0}\frac{\vect{b}}{\m_0}\Pi_0\frac{\vect{b}}{\m_0}\bigr] +\mathcal{O}\bigl(|\sigma|(|\dM|+|\other{\dM}|) +|\dM|^2+|\other{\dM}|^2 + |\eta|^{1/3} +|\other{\eta}|^{1/3}\bigr).
	%	\end{split}
%\end{equation}
%Using , we conclude that
\begin{equation} \label{K_term1}
	\begin{split}
		\langle\vect{f}^2\rangle\eigB(z)\Tr\bigl[\Pi\frac{\m'}{\m}\bigr] =& \frac{\pi\sigma}{\langle\vect{f}^2\rangle}
		+ 2\pi\dM\frac{\psi }{\langle \vect{f}^2\rangle} + \pi\other{\dM}\frac{\psi}{\langle\vect{f}^2\rangle} +\mathcal{O}\bigl(\Delta_0^{1/3}|\dM| + |\dM|^2 + |\other{\dM}|^2% +|z-\sng| + |\zeta-\sng|
		\bigr),
	\end{split}
\end{equation}
with $\dM := \dM(\sng, z)$, $\other{\dM} := \dM(\sng, \zeta)$.
Therefore, 
\begin{equation}
	\begin{split}
		\eigB(z)\eigB(\zeta)\Tr\bigl[ \frac{\m'}{\m}\Pi\frac{\other{\m}'}{\other{\m}}\Pi\bigr] =&~ \frac{\pi^2}{\langle\vect{f}^2\rangle^4}(\sigma + 2\psi\dM + \psi\other{\dM} )(\sigma + \psi\dM + 2\psi\other{\dM} )\\ &+ \mathcal{O}\bigl(\Delta_0^{2/3}(|\dM|+|\other{\dM}|) + \Delta_0^{1/3}(|\dM|^2+|\other{\dM}|^2)% + |z-\sng| + |\zeta-\sng| 
		\bigr).
	\end{split}
\end{equation}

We now compute the $\beta^{-1}$ term in \eqref{K_traces}. It follows by Lemma \ref{lemma_m'} and estimates \eqref{Pi_pert}, \eqref{Q/B_pert}, that 
\begin{equation}
	\begin{split}
		\eigB(z)\eigB(\zeta)\Tr\bigl[\frac{\m'}{\m}\Pi\frac{\other{\m}'}{\other{\m}}\frac{Q}{\stab} + \frac{\m'}{\m}\frac{Q}{\stab}\frac{\other{\m}'}{\other{\m}}\Pi - \frac{\m'\other{\m}'}{\m\other{\m}}\Pi\bigr] = \frac{\pi^2\psi}{\langle\vect{f}^2\rangle^3} + \mathcal{O}\bigl(\Delta_0^{2/3}+|\dM|+|\other{\dM}|\bigr).
	\end{split}
\end{equation}

By Assumptions \eqref{cond_A}, \eqref{cond_B}, and estimate \eqref{m_bound}, the terms in the second line of \eqref{K_traces} are bounded by $\mathcal{O}(|\eigB(z)\eigB(\zeta)|^{-1})$, and it follows from \eqref{eig_estimate} that $|\eigB(z,\zeta)| \lesssim \Delta_0^{1/3}(|\dM| +|\other{\dM}|) + |\dM|^2 + |\other{\dM}|^2$, hence \eqref{eq:K_est} is established. 

The bound \eqref{eq:K_crude_bound} follows from \eqref{dM_bound}, \eqref{eq:sigma_comp}, and the estimate \eqref{eq:K_est}.
\end{proof}
\subsection{Integration by Parts}
For the remainder of Section \ref{sec:variance_lemma_proof}, some implicit constants in the $\lesssim$ and $\sim$ relations may depend on the $L^2$ norms of the test function $g$ and its derivatives, as well as on the choice of the function $\chi$ in \eqref{QA_f}.
\begin{lemma} \label{lemma:V*} Let $0 < \eta_* \le N^{-\varepsilon_0/2}\eta_0$, and define $\Omega_* := \{z\in\mathbb{C} : \eta_* < |\im z| < 1\}$, then 
	\begin{equation} \label{eq:V*estimate}
		V(f) = \frac{1}{\pi^2}\int\limits_{\Omega_*}\int\limits_{\Omega_*}\frac{\partial\other{f}(\zeta)}{\partial\bar\zeta}\frac{\partial\other{f}(z)}{\partial\bar{z}}\mathcal{K}(z,\zeta)\mathrm{d}\bar\zeta\mathrm{d}\zeta\mathrm{d}\bar{z}\mathrm{d}z + \mathcal{O}(N^{-\varepsilon_0}).
	\end{equation}
\end{lemma}
\begin{proof}[Proof of Lemma \ref{lemma:V*}]
	Let $V_*(f)$ denote the integral on the right-hand side of \eqref{eq:V*estimate}. For all $z\in\supp{f}$, we have $|\re z - \sng| \le C\eta_0$. If $\widehat{\sng}\in \D$, we split the domains of integration over $z$ and $\zeta$ by the perpendicular bisector of $[\sng, \widehat{\sng}]$ to facilitate a polar change of coordinates around $\sng$ and $\widehat{\sng}$ in their respective half-planes. In each part of the domain we consider two separate regimes: $\re z\in\supp{\rho}$ and $\re z \notin \supp{\rho}$, therefore $|V(f)-V_*(f)| = I_{\mathrm{in,in}} + I_{\mathrm{in,out}} + I_{\mathrm{out,in}}+ I_{\mathrm{out,out}}$.
	
	To bound the eigenvalue $\eigB(z,\zeta)$ we use the following technical estimate, which is an immediate consequence of \eqref{eigF_1} and Lemma 4.4 in \cite{Landon2021Wignertype}.
	\begin{equation} \label{eq:eigB_crude_bound}
		|\eigB(z,\zeta)| \gtrsim \rho(z)^{-1}|\im z| + \rho(\zeta)^{-1}|\im\zeta|.
	\end{equation}
	
	It follows from estimates \eqref{eq:K_crude_bound}, \eqref{eq:eigB_crude_bound} and the scaling relations \eqref{eq:eig_comp} that
	\begin{equation}
		\begin{split}
			|I_{\mathrm{in,in}}| &\lesssim \eta_0^{-4} \int\limits_{\Omega_{\mathrm{in}}}\int\limits_{\Omega_{\mathrm{in}}}  \frac{(\dis^{3/2}\sin\omega) (\other{\dis}^{3/2} \sin\other{\omega}) (\Delta_0+\dis)^{-1/6}(\Delta_0 + \other{\dis})^{-1/6}(\Delta_0 + \dis + \other{\dis} )^{2/3}}{ \bigl((\Delta_0+\dis)^{1/6}\dis^{1/2}\sin\omega + (\Delta_0+\other{\dis})^{1/6}\other{\dis}^{1/2}\sin\other{\omega})^2}\mathrm{d}\dis\mathrm{d}\omega\mathrm{d}\other{\dis}\mathrm{d}\other{\omega}\\ 
			&\lesssim \int\limits_0^{C\eta_0}\int\limits_0^{C\eta_0} \frac{(\Delta_0 + \dis)^{1/3}}{(\Delta_0+\other{\dis})^{1/3}}\dis\other{\dis}\min\{1,\dis^{-1}\eta_0N^{-\varepsilon_0/2}\}\min\{1,\other{\dis}^{-1}\eta_0N^{-\varepsilon_0/2}\}\mathrm{d}\dis\mathrm{d}\other{\dis} \lesssim N^{-\varepsilon_0},
		\end{split}
	\end{equation}
	where $\Omega_{\mathrm{out}} := \{(\dis,\omega) : 0 \le \dis \le C\eta_0, 0 \le \omega \le \pi/2, \sin\omega \le 2\eta_0N^{-\varepsilon_0/2} \}$, and we used the uniform bound $|f''(x)| \lesssim \eta_0^{-2}$. The last inequality follows by direct computation.
	
	Similarly, in the case  $z\in\supp{\rho}$, $\zeta\notin\supp{\rho}$, we obtain the bound
	\begin{equation}
		\begin{split}
			|I_{\mathrm{in,out}}| & \le \frac{C}{\eta_0^4}\int\limits_{\Omega_{\mathrm{out}}}\int\limits_{\Omega_{\mathrm{in}}}  \frac{\dis^{3/2}\sin\omega\,\other{\dis}^{3/2}(\Delta_0+\dis)^{-1/6} (\Delta_0 + \dis + \other{\dis} )^{2/3} \, \mathrm{d}\dis\mathrm{d}\omega\mathrm{d}\other{\dis}\mathrm{d}\other{\omega} }{(\Delta_0 + \other{\dis})^{1/6} \bigl((\Delta_0+\dis)^{1/6}\dis^{1/2}\sin\omega + (\Delta_0+\other{\dis})^{1/6}\other{\dis}^{1/2})^2} \lesssim N^{-\varepsilon_0},
		\end{split}
	\end{equation}
	where $\Omega_{\mathrm{in}} := \{(\dis,\omega) : 0 \le \dis \le \min\{\Delta/2, C\eta_0\}, 0 \le \omega \le \pi/2, \sin\omega \le 2\eta_0N^{-\varepsilon_0/2} \}$. By exchanging $z$ and $\zeta$, the same estimate holds for $I_{\mathrm{out,in}}$. The case $z,\zeta\notin\supp{\rho}$ follows analogously. This concludes the proof of Lemma \ref{lemma:V*}.
\end{proof}
\begin{lemma}  \label{lemma:V_good}
	Let $\eta_*\in (0,N^{-100})$.
	There exists and interval $\mathcal{R}$ with length $|\mathcal{R}|\sim N^{-\varepsilon_0/3}$, which satisfies $\sng\in\mathcal{R}$, and $\dist(\partial\,\mathcal{R}, \Sng) \sim N^{-\varepsilon_0/3}$, such that 
	\begin{equation} \label{eq:V_eta*}
		V(f) = \frac{1}{4\pi^2} \iint\limits_{\mathcal{R}^2} \bigl(f(y)-f(x)\bigr)^2\mathcal{K}_{\eta_*}(x,y)\mathrm{d}x\mathrm{d}y + \mathcal{O}(\eta_0^{1/3}+N^{-\varepsilon_0}),
	\end{equation}
	where the kernel $\mathcal{K}_{\eta_*}(x,y)$ is defined as
	\begin{equation} \label{eq:K_eta}
		\mathcal{K}_{\eta_*}(x,y) := \re\bigl[\mathcal{K}(x+\I\eta_*,y+\I\eta_*) - \mathcal{K}(x+\I\eta_*,y-\I\eta_*)\bigr].
	\end{equation}
\end{lemma}
\begin{proof}[Proof of Lemma \ref{lemma:V_good}]
	We introduce an auxiliary functions $\mathcal{L}(z,\zeta)$, defined by
	\begin{equation}
		\begin{split}
			\mathcal{L}(z,\zeta) &:= \mathcal{L}_{\log}(z,\zeta) + \mathcal{L}_1(z,\zeta),\\
			\mathcal{L}_{\log}(z,\zeta) &:= -2\log\det\{\stab(z,\zeta)\}, \quad \mathcal{L}_1(z,\zeta) := \Tr[\m\other{\m}S] +\frac{1}{2}\bigl\langle\overline{\m\other{\m}},\mathcal{C}^{(4)}[\m\other{\m}]\bigr\rangle
		\end{split}
	\end{equation}
	Note that 
	\begin{equation} \label{eq:L''=K}
		\frac{\partial^2}{\partial z\partial\zeta} \mathcal{L}(z,\zeta) = \mathcal{K}(z,\zeta)
	\end{equation}
	Furthermore,  the following estimates hold uniformly in $z,\zeta \in \D$,
	\begin{equation} \label{eq:L_bounds}
		|\mathcal{L}(z,\zeta)| \lesssim N\bigl(1+|\log (|\im z|)| + |\log(|\im\zeta|)|\bigr) ,\quad|\partial_\zeta{L}(z,\zeta)| \lesssim  \frac{(\Delta_0+\dis+\other{\dis})^{1/3}}{|\eigB(\zeta,\zeta)\eigB(z,\zeta)|}.
	\end{equation}
	We write $z:= x+\I\eta, \zeta := y+\I\other{\eta}$ and plug \eqref{eq:L''=K} into the expression \eqref{eq:V*estimate} for $V(f)$. Using the fact that $\partial_z u = -i\partial_\eta u$ for any holomorphic function $u(z)$, and integrating by parts in $\eta$, we obtain
	\begin{equation} \label{by_parts_in_eta}
		\begin{split}
			V(f) =& \frac{\I}{\pi^2}\iint\limits_{\mathbb{R}^2}\mathrm{d}x\mathrm{d}y\int\limits_{|\other{\eta}|>\eta_*} \frac{\partial \other f(\zeta)}{\partial \bar \zeta}\int\limits_{|\eta|>\eta_*} \frac{\partial^2 \other f(z)}{\partial \eta\partial \bar z} \frac{\partial}{\partial\zeta} \mathcal{L}(z,\zeta) \mathrm{d}\other{\eta}\mathrm{d}\eta\\
			&-\frac{\eta_*}{2\pi^2}\iint\limits_{\mathbb{R}^2}\mathrm{d}x\mathrm{d}yf''(x)\int\limits_{|\other{\eta}|>\eta_*} \frac{\partial \other f(\zeta)}{\partial \bar \zeta}\frac{\partial}{\partial\zeta}\bigl( \mathcal{L}(x+\I\eta_*,\zeta) + \mathcal{L}(x-\I\eta_*,\zeta) \bigr)\mathrm{d}\other{\eta}+ \bigO{N^{-\twoalp}}.
		\end{split}
	\end{equation}
	Here we used the fact that $\partial_{\bar{z}}\other{f}(x+\I\eta) = \I\eta f''(x)/2$, for $|\eta| \ll 1$. Combining the second estimate in \eqref{eq:L_bounds} with \eqref{eq:eigB_crude_bound}, we assert that $|\partial_{\zeta}\mathcal{L}(z,\zeta)| \lesssim |\other{\eta}|^{-4/3}$, therefore, invoking Lemma \ref{lemma:int}, we estimate the absolute value of the boundary term (the second line in \eqref{by_parts_in_eta}) by $\mathcal{O}(\eta_*\eta_0^{-4/3})$, which is smaller than $\bigO{N^{-\twoalp}}$.
	
	Similarly, integrating the first term on the right hand side of \eqref{by_parts_in_eta} by parts in $\other{\eta}$ yields
	\begin{equation} \label{by_parts_in_eta'}
		\begin{split}
			V(f) =& -\frac{1}{\pi^2}\int\limits_{\Omega_*}\int\limits_{\Omega_*}\frac{\partial^2 \other f(z)}{\partial \bar z\partial \eta} \frac{\partial^2 \other f(\zeta)}{\partial \bar \zeta\partial\other{\eta}} \mathcal{L}(z,\zeta) \mathrm{d}\bar\zeta\mathrm{d}\zeta \mathrm{d}\bar z\mathrm{d}z\\
			&+\frac{\I\eta_*}{2\pi^2}\iint\limits_{\mathbb{R}^2}\mathrm{d}x\mathrm{d}yf''(y)\int\limits_{|\eta|>\eta_*}  \frac{\partial^2 \other f(z)}{\partial \eta\partial \bar z}\bigl(\mathcal{L}(z,y+\I\eta_*)+\mathcal{L}(z,y-\I\eta_*)\bigr) \mathrm{d}\eta+ \bigO{N^{-\twoalp}}.
		\end{split}
	\end{equation}
	It follows from the first estimate in  \eqref{eq:L_bounds}, and the uniform bound $|\partial_\eta\partial_{\bar{z}}\other{f}| \lesssim \eta_0^{-2}$, %and the expression \eqref{quasi_analytic_deriv}
	that the boundary term (the second line of \eqref{by_parts_in_eta'}) is again dominated by $\mathcal{O}(N^{-\twoalp})$.
	
	We apply Stokes' theorem to \eqref{by_parts_in_eta'} twice: once in $z$ and once in $\zeta$. Considering that $\partial_\eta \other{f}(z)$ vanishes on the boundary of $\Omega_*$ except for the lines $\{\im z = \pm\eta_*\}$, this results in
	\begin{equation} \label{V(f)_f'f'_form}
		\begin{split}
			V(f) =& \frac{1}{4\pi^2}\iint\limits_{\mathbb{R}^2}\sum\limits_{\eta,\other{\eta} = \pm\eta_*}\sign\left(\eta\other{\eta}\right)\frac{\partial \other f(x+\I\eta)}{\partial \eta} \frac{\partial \other f(y+\I\other{\eta})}{\partial\other{\eta}} \mathcal{L}(x+\I\eta,y+\I\other{\eta}) \mathrm{d}x\mathrm{d}y + \bigO{N^{-\twoalp}}\\
			=&-\frac{1}{2\pi^2}\iint\limits_{\mathbb{R}^2}f'(x)f'(y)\other{\mathcal{L}}(x,y) \mathrm{d}x\mathrm{d}y + \bigO{N^{-\twoalp}},
		\end{split}
	\end{equation}  
	where 
	\begin{equation}
		\other{\mathcal{L}}(x,y):=\re\left[\mathcal{L}(x+\I\eta_*,y+\I\eta_*)-\mathcal{L}(x+\I\eta_*,y-\I\eta_*)\right]
	\end{equation}
	
	It follows from (4.4) of Theorem 4.1 in \cite{Ajanki2016Univ}, that the length of the support intervals of $\rho$ is order one, and the spacing between the endpoints of the intervals and the cusp points is order one. Therefore, by possibly shrinking the size $c\sim 1$ of $\D$, we can assume that $\Sng\cap\D \subset \{\sng, \widehat{\sng}\}$. If $\Delta < \frac{1}{2}c_0N^{-\varepsilon_0/3}$, we set $E_\pm := \sng \pm \frac{3}{4}c_0N^{-\varepsilon_0/3}$, and if $\Delta > \frac{1}{2}c_0N^{-\varepsilon_0/3}$, we set $E_\pm := \sng \pm \frac{1}{4}c_0N^{-\varepsilon_0/3}$. Then $\mathcal{R}:=[E_-,E_+]$ has length $|\mathcal{R}| \sim N^{-\varepsilon_0/3}$, and satisfies $\sng\in \mathcal{R}$, $\dist(\partial\mathcal{R}, \Sng)\sim N^{-\varepsilon_0/3}$.
	
	For all $y\in\supp{f}$, $y-E_0\lesssim \eta_0$, hence $\supp{f}\subset \mathcal{R}$, and $|y-E_\pm|\sim N^{-\varepsilon_0/3}$. By symmetry of the function $\mathcal{L}(z,\zeta)$ and the second estimate in \eqref{eq:L_bounds}, it follows that  
	\begin{equation} \label{boundary_L'_bound}
		\frac{\partial}{\partial y}\other{\mathcal{L}}(E_\pm,y) \lesssim N^{\varepsilon_0/2}, \quad y \in \supp{f}.
	\end{equation} 
	We write $f'(y) = \partial_y\left(f(y)-f(x)\right)$, perform integration by parts in $y$ and integrate the boundary term by parts in $x$ to obtain
	\begin{equation} \label{by_parts_in_y}
		\begin{split}
			V(f)
			%&-\frac{1}{2\pi^2}\int\limits_\mathbb{R} f'(x)f(x) \left(\other{\mathcal{L}}(x,E_0+\eps) -\other{\mathcal{L}}(x,E_0-\eps) \right) \mathrm{d}x + \bigO{N^{-\twoalp}}\\
			=&\frac{1}{2\pi^2}\iint\limits_{\mathcal{R}^2}f'(x)\left(f(y)-f(x)\right)\frac{\partial}{\partial y} 
			\other{\mathcal{L}}(x,y)
			\mathrm{d}x\mathrm{d}y\\
			&+\frac{1}{4\pi^2}\int\limits_{\mathcal{R}} f(x)^2 \frac{\partial}{\partial x}\bigl(\other{\mathcal{L}}(x,E_+) -\other{\mathcal{L}}(x,E_-) \bigr) \mathrm{d}x + \bigO{N^{-\twoalp}}.
		\end{split}
	\end{equation}
	Since $\norm{f}_2^2 \lesssim \eta_0$,  it follows from the first estimate in \eqref{eq:L_bounds} that the second integral in \eqref{by_parts_in_y} is $\mathcal{O}(N^{\varepsilon_0/2}\eta_0)$. 
	Similarly, integrating \eqref{by_parts_in_y} by parts in $x$ and using \eqref{by_parts_in_y} to substitute one of the emerging integrals for $-V(f) + \mathcal{O}(N^{-\twoalp}+N^{\varepsilon_0/2}\eta_0)$, we get
	\begin{equation} \label{by_parts_in_x}
		\begin{split}
			2V(f)
			=&\frac{1}{2\pi^2}\iint\limits_{\mathcal{R}^2}\bigl(f(y)-f(x)\bigr)^2\frac{\partial^2}{\partial x\partial y} 
			\other{\mathcal{L}}(x,y)
			\mathrm{d}x\mathrm{d}y
			+ \bigO{N^{\varepsilon_0/2}\eta_0+N^{-\twoalp }},
		\end{split}
	\end{equation}
	where we used \eqref{boundary_L'_bound} to estimate the boundary term. 
	For any holomorphic function $u(z)$ of $z = x + i\eta$, we have $\partial_x u = \re[\partial_z u]$, hence $\partial_x\partial_y\other{\mathcal{L}}(x,y) = \re\left[\mathcal{K}(x+\I\eta_*,y+\I\eta_*)-\mathcal{K}(x+\I\eta_*,y-\I\eta_*)\right]$. This concludes the proof of Lemma \ref{lemma:V_good}.
\end{proof}

\subsection{Dominated Convergence} \label{sec:dom_conv}
\begin{lemma} \label{lemma:K_dom}
	For all $x,y \in \mathcal{R}$, the estimate
	\begin{equation} \label{eq:Kdomi}
		|\mathcal{K}_{\eta_*}(x,y)| \lesssim \mathcal{K}_{\mathrm{bound}}(x,y) := \frac{\bigl(M(x,y-x)^2+\rho(x)^2\bigr)\bigl(\Delta_0 + \dis(x) + \dis(y) \bigr)^{2/3} + N^{-100/3}
		}{|x-y|^2\dis(x)^{1/2}(\Delta_0+\dis(x))^{1/6}\dis(y)^{1/2}(\Delta_0+\dis(y))^{1/6}}
	\end{equation}
	holds uniformly in $\eta_* \in [0, N^{-100}]$, where the control quantity $M(z,w)$ is defined in \eqref{eq:better_dM_bound}.
\end{lemma}
\begin{proof}[Proof of Lemma \ref{lemma:K_dom}]
	It follows from \eqref{eq:eigB}, \eqref{eq:eigBbar}, and \eqref{eq:K_crude_bound} that for all $z:=x+\I \eta$, $\zeta:=y+\I\eta \in \D\cap\mathbb{H}$, we have 
	\begin{equation}
		|\mathcal{K}_{\eta_*}(x,y)| \lesssim \frac{\bigl(M(x+\I\eta_*,y-x)+\rho(x)+\eta_*^{1/3}\bigr)^2\bigl(\Delta_0 + \dis(x) + \dis(y) + \eta_* \bigr)^{2/3}}{|x-y|^2\dis(x)^{1/2}(\Delta_0+\dis(x))^{1/6}\dis(y)^{1/2}(\Delta_0+\dis(y))^{1/6}},
	\end{equation}
	where we used $\dis(z) \gtrsim \dis(\re z)$, $|\dM(z,\zeta)| \lesssim M(x+\I\eta,y-x)$ by \eqref{eq:better_dM_bound}, and $\rho(x+\I\eta) \lesssim \rho(x)+\eta^{1/3}$, which follows from $1/3$-H\"{o}lder continuity of $\rho$. Observe that by definition of $M(z,y-x)$ in \eqref{eq:better_dM_bound}, the quantity $M(x+\I\eta,y-x)$ is monotonically non-increasing in $\eta \ge 0$. The estimate \eqref{eq:Kdomi} now follows from the fact that both factors in the numerator are $\mathcal{O}(1)$.
\end{proof}

\begin{lemma} \label{lemma:domi}
	The quantity $\mathcal{K}_{\mathrm{bound}}(x,y)$ defined in \eqref{eq:Kdomi} satisfies
	\begin{equation} \label{eq:Kdom_int}
		\iint\limits_{\mathcal{R}^2}\bigl(f(x)-f(y)\bigr)^2\mathcal{K}_{\mathrm{bound}}(x,y)\mathrm{d}x\mathrm{d}y \lesssim 1.
	\end{equation}
\end{lemma}
\begin{proof}[Proof of Lemma \ref{lemma:domi}]
	Observe that $|f(x) - f(y)|^2 \lesssim \eta_0^{-2} |x-y|^2$, hence the $N^{-100/3}$ term contributes at most $N^{-100/3}\eta_0^{-2}$ to the integral, and is therefore negligible.
	
	We divide the interval on integration $\mathcal{R}$ into parts in the following way:
	$\mathcal{R}_{\mathrm{supp}} := \mathcal{R}\cap \supp{\rho}$, $\mathcal{R}_{\mathrm{gap}} := \mathcal{R}\backslash \supp{\rho}$, $\mathcal{R}_{\sng} := \{x \in \mathcal{R} : \dis(x) = |x-\sng|\}$, $\mathcal{R}_{\widehat{\sng}} := \{x \in \mathcal{R} : \dis(x) = |x-\widehat{\sng}|\}$. 
	For $E \in \{\sng, \widehat{\sng}\}$, in each of the segments $\mathcal{R}_{\mathrm{supp}}\cap\mathcal{R}_{E}$, $\mathcal{R}_{\mathrm{gap}}\cap\mathcal{R}_{E}$, we can perform a change of variable $w := \dis(x) = |E + w|$.
	If $x,y \in \mathcal{R}_{E}$, then we have 
	\begin{equation} \label{eq:M_same_side}
		M(x,y-x)^2 +\rho(x)^2\lesssim (\dis(x)+\dis(y))(\Delta_0+\dis(x)+\dis(y))^{-1/3}.
	\end{equation}
	In the opposite case, for $x \in \mathcal{R}_{E}$, $y \in \mathcal{R}_{\widehat{E}}$, we have the estimate
	\begin{equation} \label{eq:M_opp_sides}
		M(x,y-x)^2 +\rho(x)^2 \lesssim (\Delta_0 + \dis(x)+\dis(y))^{2/3}.
	\end{equation}
	
	Furthermore, define $\mathcal{S} := \bigcap_{[a,b]\supset \supp{f}}[a-c\eta_0,b+c\eta_0] \subset \mathcal{R}$. 
	If both $x,y \notin \mathcal{S}$, then $f(x)=f(y) =0$, and the estimate is trivial. Therefore, without loss of generality, we always assume that $x\in \mathcal{S}$.
	
	Let $I(\mathcal{R}_1,\mathcal{R}_2)$ denote the integral of $\mathcal{K}_{\mathrm{bound}}(x,y)$ against $(f(x)-f(y))^2$ over the domain $\mathcal{R}_1\times \mathcal{R}_2 \subset \mathcal{R}^2$.  
	
	First, we tackle the case $y \in \mathcal{S}$. In this regime we use the estimate $|f(x) - f(y)|^2 \lesssim \eta_0^{-2} |x-y|^2$, which is a consequence of $g\in C^2_0(\mathbb{R})$. Furthermore, we only consider $x\in \mathcal{R}_{\sng}\cap \mathcal{S}$, since the other case is analogous due to symmetry.
	First, using \eqref{eq:M_same_side} we estimate
	\begin{equation} \label{eq:in_same}
		I(\mathcal{R}_{\sng}\cap\mathcal{S},\mathcal{R}_{\sng}\cap\mathcal{S}) \lesssim 
		\eta_0^{-2}\int\limits_{0}^{\eta_0}\int\limits_{0}^{\eta_0} \frac{(w+\other{w})\bigl((\Delta_0 + w)^{1/3}+(\Delta_0 + \other{w})^{1/3}\bigr)}{w^{1/2}(\Delta_0+w)^{1/6}\other{w}^{1/2}(\Delta_0+\other{w})^{1/6}}\mathrm{d}w\mathrm{d}\other{w} \lesssim 1,
	\end{equation}
	where the final inequality is obtained by expanding the brackets and using the fact that $(\Delta_0+w)^{-1} \lesssim \eta_0(\Delta_0+\eta_0)^{-1}w^{-1}$ for all $w \in [0,\eta_0]$. Similarly, using \eqref{eq:M_opp_sides}, we obtain 
	\begin{equation}
		I(\mathcal{R}_{\sng}\cap\mathcal{S},\mathcal{R}_{\widehat{\sng}}\cap\mathcal{S}) \lesssim 
		\eta_0^{-1}(\Delta_0 + \eta_0) \lesssim 1,
	\end{equation}
	since $\mathcal{R}_{\widehat{\sng}}\cap\mathcal{S}$ is not empty only if $\Delta \lesssim \eta_0$.
	
	Next, we consider the case $y\notin \mathcal{S}$, then $x\in \supp{\rho}$, hence $|x-y|\gtrsim \eta_0$. In this regime, we use the estimate $|f(x)-f(y)|^2 \lesssim 1$. First, if $x,y$ belong to the same $\mathcal{R}_E$, for $E\in\{\sng, \widehat{\sng}\}$, then $|x-y| \gtrsim |\dis(x) - \dis(y)|$, and using the estimate \eqref{eq:M_same_side}, we deduce that 
	\begin{equation} \label{eq:out_same}
		I(\mathcal{R}_{\sng}\cap\mathcal{S},\mathcal{R}_{\sng}\backslash\mathcal{S}) \lesssim 
		\int\limits_{0}^{\eta_0}\frac{\mathrm{d}w}{w^{1/2}(\Delta_0+w)^{1/6}}\int\limits_{\eta_0}^{1} (\Delta_0 + q)^{1/6}q^{-3/2}\mathrm{d}q \lesssim 1,
	\end{equation}
	where we performed a change of variable $q:=\other{w}-w$, and used the fact that $q \gtrsim \eta_0 \gtrsim w$.	Similarly, we obtain the estimate $I(\mathcal{R}_{\widehat{\sng}}\cap\mathcal{S},\mathcal{R}_{\widehat{\sng}}\backslash\mathcal{S}) \lesssim 1$.
	
	It remains to estimate the contribution of $x\in \mathcal{R}_{E}\cap\mathcal{S}$, $y \in \mathcal{R}_{\widehat{E}}\backslash\mathcal{S}$, for $E \in \{\sng, \widehat{\sng}\}$. We split the proof into two cases: $\widehat{\sng} \in \mathcal{S}$, and $\widehat{\sng} \notin \mathcal{S}$.
	
	If $\widehat{\sng} \in \mathcal{S}$, then $\Delta \lesssim \eta_0$, and $\mathcal{R}\backslash\mathcal{S} \subset \mathcal{R}_{\mathrm{supp}}$. Therefore, for all $x \in \mathcal{R}_{\sng}\cap \mathcal{S}$, $y\in \mathcal{R}_{\widehat{\sng}}\backslash\mathcal{S}$, we have $|x-y| \gtrsim \eta_0+\dis(y)$, hence \eqref{eq:M_opp_sides} yields
	\begin{equation}
		I(\mathcal{R}_{\sng}\cap\mathcal{S},\mathcal{R}_{\widehat{\sng}}\backslash\mathcal{S}) \lesssim (\Delta_0+\eta_0)^{1/3}\int\limits_{0}^{\eta_0}\frac{\mathrm{d}w}{w^{1/2}(\Delta_0+w)^{1/6}}\int\limits_{\other{c}\eta_0}^{1}\frac{\mathrm{d}\other{w}}{(\eta_0+\other{w})^{5/6}\other{w}^{1/2}}\lesssim 1.
	\end{equation} 
	Under the assumption $\widehat{\sng} \in \mathcal{S}$, the integral $I(\mathcal{R}_{\widehat{\sng}}\cap\mathcal{S},\mathcal{R}_{\sng}\backslash\mathcal{S})$ is estimated analogously.
	
	If $\widehat{E} \notin \mathcal{S}$, then $\Delta_0 \gtrsim \eta_0$, and $\mathcal{R}_{\widehat{\sng}}\cap\mathcal{R}_{\mathrm{supp}}\cap\mathcal{S} = \varnothing$. It follows trivially from \eqref{eq:M_opp_sides} that
	\begin{equation}
		I(\mathcal{R}_{E}\cap\mathcal{S},\mathcal{R}_{\widehat{E}}\cap\mathcal{R}_{\mathrm{supp}}\backslash\mathcal{S}) \lesssim \eta_0^{1/2}\Delta_0^{-1/6}\int\limits_{0}^{1}\frac{\mathrm{d}\other{w}}{(\Delta_0+\other{w})^{5/6}  \other{w}^{1/2}} \lesssim 1, \quad E\in \{\sng, \widehat{\sng}\},
	\end{equation}  
	and
	\begin{equation}
		I(\mathcal{R}_{\sng}\cap\mathcal{R}_{\mathrm{supp}}\cap\mathcal{S},\mathcal{R}_{\widehat{\sng}}\cap\mathcal{R}_{\mathrm{gap}}\backslash\mathcal{S}) \lesssim \int\limits_{0}^{\eta_0}w^{-1/2}\Delta_0^{-1/6}\mathrm{d}w\int\limits_{0}^{\Delta/2}\mathrm{d}\other{w}\frac{(\Delta_0+\other{w})^{7/6}}{(\Delta-\other{w})^{2}  \other{w}^{1/2}} \lesssim 1.
	\end{equation} 
	Finally, we consider $x\in\mathcal{R}_{E}\cap\mathcal{R}_{\mathrm{gap}}\cap\mathcal{S}$, $y\in\mathcal{R}_{\widehat{E}}\cap\mathcal{R}_{\mathrm{gap}}\backslash\mathcal{S}$. Note that this the contribution of this regime is non-zero if $\eta_0 \lesssim \Delta \lesssim 1$. Furthermore, by \eqref{eq:better_dM_bound}, in this regime we have the estimate
	\begin{equation} \label{eq:M_gap_gap}
		M(x,y-x)^2 + \rho(x)^2 \lesssim |x-y|(\Delta_0 + \dis(x)+\dis(y))^{-1/3}.
	\end{equation}
	\begin{equation}
		I(\mathcal{R}_{E}\cap\mathcal{R}_{\mathrm{gap}}\cap\mathcal{S},\mathcal{R}_{\widehat{E}}\cap\mathcal{R}_{\mathrm{gap}}\backslash\mathcal{S}) \lesssim \int\limits_{0}^{\Delta/2}\int\limits_{0}^{\Delta/2}\frac{\mathds{1}_{w\le C \eta_0}\mathrm{d}\other{w}\mathrm{d}w}{w^{1/2}(\Delta-w-\other{w})  \other{w}^{1/2}} \lesssim \frac{\Delta^{1/6}}{\Delta_0^{1/6}}\min\{1, \eta_0/\Delta\}^{1/2},
	\end{equation}  
	which concludes the proof of Lemma \ref{lemma:domi}
\end{proof}

Using dominated convergence theorem and Lemma \ref{lemma:domi}, we can take the $\eta_* \to 0$ limit on both sides of \eqref{eq:V_eta*}, and conclude that
\begin{equation} \label{eq:V_limit}
	V(f) = \frac{1}{4\pi^2} \iint\limits_{\mathcal{R}^2} \bigl(f(y)-f(x)\bigr)^2\mathcal{K}_0(x,y)\mathrm{d}x\mathrm{d}y + \mathcal{O}(\eta_0^{1/3}+N^{-\varepsilon_0}),
\end{equation}
where $\mathcal{K}_0(x,y) := \lim_{\eta_*\to+0}\mathcal{K}_{\eta_*}(x,y)$.

\subsection{Computing the Limiting Kernel}

\begin{proof}[Proof of Lemma \ref{lemma:variance}]
	Observe that for all $x := \sng +w \in \mathcal{R}\backslash\supp{\rho}$, and all $y:=\sng+\other{w}\in\mathcal{R}\backslash\{x\}$, we have 
	\begin{equation} \label{eq:K_out}
		\mathcal{K}_0(x,y) = 0,
	\end{equation}
	Identity \eqref{eq:K_out} follows from \eqref{eq:K_eta} and the fact that the vector $\m(x\pm\I0)$ has real entries for $x\notin\supp{\rho}$.
	
	For all $z:= x+\I0$, with $x := \sng+w \in \mathcal{R}$, we have $\langle \vect{b}^{\ell}(z), \vect{b}(z) \rangle	 = \langle \vect{b}^{\ell}(\sng), \vect{b}(\sng) \rangle(1+ \mathcal{O}(|w|^{1/3}))$, as well as  $\lVert\vect{b}^\ell(z)-\vect{b}^\ell(\sng)\rVert_\infty \lesssim |w|^{1/3}$ by \eqref{dm_bound} and \eqref{Pi_pert}. Therefore, for all $z:=x+\I0, \zeta:=y+\I0$ with $x,y \in \mathcal{R}$, 
	\begin{equation} \label{eq:dMdiff_est}
		\dM(z,\zeta) = (\other{\dM}-\dM)\bigl(1+ \mathcal{O}(|w|^{1/3}+|\other{w}|^{1/3})\bigr), \quad \dM(z,\zeta) + 2\I\rho(x) = (\other{\dM}-\overline{\dM})\bigl(1+\mathcal{O}(|w|^{1/3}+|\other{w}|^{1/3})\bigr),
	\end{equation}
	where we additionally used \eqref{eq:dM+rho} to obtain the second estimate. Here $\dM := \dM(\sng, z)$, $\other{\dM} := \dM(\sng, \zeta)$.

	Let $\mathcal{R}_{\Delta} := \{y \in \mathcal{R} \cap \supp{\rho} : |y-\widehat{\sng}| \le \Delta_0^{4/3}\}$ then for all $z:= x+\I0$ with $x\in \supp{\rho}\cap\mathcal{R}\backslash\mathcal{R}_{\Delta}$, estimates \eqref{eq:eig_comp}, \eqref{eq:dM_approx}, and \eqref{eig_estimate} imply 
	\begin{equation} \label{eq:eigB_lower}
		\frac{1}{\eigB(z,z)} = \frac{-\langle\vect{f}^2\rangle}{2\sigma\sol(w) + 3\psi\sol(w)^2 }\bigl(1+\mathcal{O}(|w|^{1/3})\bigr).
	\end{equation}
	Therefore, combining estimates \eqref{eq:eigB}, \eqref{eq:eigBbar}, \eqref{eq:K_est}, \eqref{eq:dM_approx}, \eqref{eq:dMdiff_est}, and \eqref{eq:eigB_lower}, we obtain
	\begin{equation} \label{eq:K_formula}
		\mathcal{K}_0(x,y) = \other{\mathcal{K}}(x-\sng,y-\sng) 
		+\mathcal{O}\bigl(\mathcal{K}_{\mathrm{bound}}(x,y)(|x-\sng|^{1/3}+|y-\sng|^{1/3}) \bigr), \quad x,y \in \mathcal{R}\backslash\mathcal{R}_{\Delta},
	\end{equation}
	where the function $\other{\mathcal{K}}$ is given by \eqref{eq:K(h)_kernel}.
	Since $|w|,|\other{w}| \lesssim N^{-\varepsilon_0/3}$, it follows from \eqref{eq:Kdom_int}, that the contribution of the second term in \eqref{eq:K_formula} to the integral in \eqref{eq:V_limit} is bounded by $\mathcal{O}(N^{-\varepsilon_0/9})$. %Therefore, to establish \eqref{eq:V_sol}, it remains to bound the contribution of $\mathcal{R}_{\Delta}$ to the integrals in \eqref{eq:V_limit} and \eqref{eq:V_sol}.
	
	Moreover, using arguments analogous to the proof of Lemma \ref{lemma:domi}, we deduce that
	\begin{equation}
		\iint\limits_{\mathcal{R}\times\mathcal{R}_{\Delta}}\bigl(f(x)-f(y)\bigr)^2\mathcal{K}_{\mathrm{bound}}(x,y)\mathrm{d}x\mathrm{d}y \lesssim \eta_0^{1/9}.
	\end{equation}
	Similarly, it can be checked using the explicit expressions for $\mathcal{K}_0(x,y)$ and $\sol(x)$ given in \eqref{eq:K(h)_kernel} and Lemma \ref{lemma:sol}, respectively, that the contribution of $\mathcal{R}_{\Delta}$ to the integral on the right-hand side of \eqref{eq:V_sol} is also bounded by $\mathcal{O}(\eta_0^{1/9})$. This concludes the proof of Lemma \ref{lemma:variance}.
\end{proof}

\section{Computation of the Bias} \label{sec:bias}
\begin{proof}[Proof of \eqref{eq:main2_bias} and \eqref{eq:Bias_converge} of Proposition \ref{prop:main2}]
	Let $\stab := \stab(z,z)$ be the stability operator defined in \eqref{stab_def}, and let $\m := \m(z)$.
	Applying the cumulant expansion formula to $-zm_j(z)\mathbb{E}[G_{jj}(z)]$ and using estimates \eqref{eq:new_term_est}, \eqref{eq:curlyTlaw} yields
	\begin{equation} \label{eq:ExpvG}
		\mathbb{E}\bigl[\Tr G\bigr] = N m + \Tr\biggl[ \frac{\m'}{\m}\frac{(1-\stab)^2}{\stab}\biggr] + \bigl\langle\overline{\m}^2,\mathcal{C}^{(4)}[\m\m']\bigr\rangle + \mathcal{O}\biggl(N \frac{\Psi^3 + ( \Delta_0 + \dis)^{1/3}\Theta^2}{|\eigB|}\biggr),
	\end{equation}
	uniformly for all $z \in \Omega_{1} := \{z\in\mathbb{C} : \eta_1\le|\im z| \le 1\}$, where $\eta_1 := \eta_{\mathfrak{f}}N^{\varepsilon_0/6}$. 
	Using the Helffer-Sj\"{o}strand formula, we assert that 
	\begin{equation} \label{eq:bias_int}
		\begin{split}
			\mathbb{E}\bigl[\Tr f(H)\bigr] =&~ \frac{1}{\pi}\int\limits_{\Omega_1} \frac{\partial\other{f}}{\partial 	\bar{z}}\biggl(N m +\Tr\biggl[ \frac{\m'}{\m}\frac{(1-\stab)^2}{\stab}\biggr] + \bigl\langle\overline{\m}^2,\,\mathcal{C}^{(4)}[\m\m']\bigr\rangle\biggr)(z)\mathrm{d}\bar z\mathrm{d}z
			+ \mathcal{O}\bigl(N^{-\varepsilon_0/6}\bigr),
		\end{split}
	\end{equation}
	where we used Lemma \ref{lemma:int} to estimate the integral of the error term in \eqref{eq:ExpvG}, and the fact that the ultra-local scales $|\eta| \le \eta_{1}$ contribute at most $N^{-\varepsilon_0/6}$ to $\Expv[\Tr f(H)]$.
	
	Observe that by Assumption \eqref{cond_B} the $\mathcal{C}^{(4)}$ term in the integrand of \eqref{eq:bias_int} is $\mathcal{O}(|\eigB|^{-1})$, hence by Lemma \ref{lemma:int}, its contribution to the integral on the right-hand side of \eqref{eq:bias_int} is bounded by $\mathcal{O}(\eta_0^{1/3})$.
	
	Next, we estimate the integral of the $N m(z)$ term.
	%\begin{equation} \label{eq:Exp_Nm_term}
	%	\frac{1}{\pi}\int\limits_{\Omega_1} \frac{\partial\other{f}}{\partial 	\bar{z}}\,N m(z)\mathrm{d}\bar z\mathrm{d}z.
	%\end{equation}
	By using the estimates $\norm{f'}_\infty \lesssim \eta_0^{-1}$ and $\norm{\m'}_\infty \lesssim |\eigB|^{-1}$ after integrating by parts in $x$, we find that
	\begin{equation}
		\int\limits_{\mathbb{R}}\mathrm{d}x \int\limits_{\eta_*}^{\eta_{1}}\mathrm{d}\eta \frac{\partial\other{f}}{\partial \bar{z}}(x+\I \eta)\, N m(x+\I\eta) = \mathcal{O}(N^{-\varepsilon_0/6}),
	\end{equation}
	for all $0 < \eta_* \le \eta_{1}$, which implies that the ultra-local scales $|\eta| \in [\eta_*, \eta_{1}]$ can be reintroduced to the integral \eqref{eq:Exp_Nm_term} at the cost of $\mathcal{O}(N^{-\varepsilon_0/6})$. Hence, by choosing $\eta_* := N^{-100}$, applying Stokes' theorem, and the estimate \eqref{dm_bound}, we obtain
	\begin{equation} \label{eq:Exp_Nm_term}
		\frac{1}{\pi}\int\limits_{\Omega_1} \frac{\partial\other{f}}{\partial 	\bar{z}}\,N m(z)\mathrm{d}\bar z\mathrm{d}z = \int\limits_{\mathbb{R}} f(x)\rho(x)\mathrm{d}x + \mathcal{O}\bigl(N^{-\varepsilon_0/6}\bigr).
	\end{equation}
	
	We proceed to estimate the contribution of the second term in the integrand of \eqref{eq:bias_int}. First, we consider the case $N^{-\varepsilon_0/3}\eta_{\mathfrak{f}} < \Delta < \Delta_*$. It follows by Stokes' theorem that
	\begin{equation} \label{eq:Exp_Tr_term}
		\frac{1}{\pi}\int\limits_{\Omega_1} \frac{\partial\other{f}}{\partial 	\bar{z}}\Tr\biggl[\frac{\m'}{\m}\frac{(1-\stab)^2}{\stab}\biggr](z)\mathrm{d}\bar z\mathrm{d}z = \frac{1}{\pi}\int\limits_{\mathbb{R}} f(x)\im\Tr\biggl[ \frac{\m'}{\m}\frac{(1-\stab)^2}{\stab}\biggr](x+\I\eta_*)\mathrm{d}x + \mathcal{O}\bigl(N^{-\varepsilon_0/6}\bigr),
	\end{equation}
	where we used estimates \eqref{eq:eig_comp}, \eqref{K_term1} to deduce that the ultra-local scales  $|\eta| \in [\eta_*, \eta_{1}]$ with $0 < \eta_* \le \eta_{1}$ can be added back to the integral on the left-hand side of \eqref{eq:Exp_Tr_term} at the cost of $\mathcal{O}(N^{-\varepsilon_0/6})$ error.
	
	We split the integration domain $\mathbb{R}$ into two rays starting at the midpoint of the gap $(\sng + \widehat{\sng})/2$. We only present the full computation for the ray containing $\sng$, as the other case is analogous. Note that the ray containing $\widehat{\sng}$ needs to be considered only if it intersects the support of $f$, that is, if $\Delta\lesssim \eta_0$.
	It follows from the decomposition $\stab^{-1} = \eigB^{-1}\Pi + \stab^{-1}Q$, and estimates \eqref{eq:eig_comp}, \eqref{eq:inbstabQ}, \eqref{eig_estimate}, \eqref{K_term1} that 
	\begin{equation} \label{eq:bias_k_int}
		\frac{1}{\pi}\int\limits_{-\Delta/2}^{+\infty} f(x_{\sigma})\im\Tr\biggl[ \frac{\m'}{\m}\frac{(1-\stab)^2}{\stab}\biggr](x_{\sigma}+\I\eta_*)\mathrm{d}x = \int\limits_{-\Delta/2}^{+\infty} f(x_{\sigma}) \im \bigl[k(x_{\sigma}+\I\eta_*)\bigr]\mathrm{d}x  +\mathcal{O}\bigl(\eta_0^{1/3}\bigr),
	\end{equation}
	where $x_{\sigma}:= \sng +x \sign{\sigma}$, the function $k(z)$ is defined as
	\begin{equation}
		k(z) := \frac{\sigma
			+ 3\psi\dM(z)}{\bigl(2\sigma\dM(z) + 3\psi\dM(z)^2 \bigr)^2},
	\end{equation}
	and $\dM(z) := \dM(\sng,z)$ is defined in \eqref{eq:dM_def}.
	
	Using the cubic equation \eqref{dm_cubic}, we find
	\begin{equation} \label{eq:bias_k}
			k(z) = -\frac{1}{4}\frac{1}{(z-\sng)\mu_0(\sng,z)} - \frac{1}{4}\frac{1}{(z-\sng)\mu_0(\sng,z) + \pi\widehat{\Delta}\sign{\sigma}} + k_2(z),
	\end{equation}
	where the function $\mu_0(\sng,z)$ satisfies \eqref{eq:cubic_mus}, and the function $k_1(z)$ is defined as
	\begin{equation}
		k_1(z) := -\frac{3}{2}\frac{\sigma\psi + 3\psi^2\dM(z)}{\dM(z)\bigl(2\sigma+3\psi\dM(z)\bigr)\bigl((3\psi\dM(z)+\sigma)^2-4\sigma^2\bigr)}.
	\end{equation}
	Note that the imaginary parts of the first two terms on the right-hand side of \eqref{eq:bias_k} behave like approximate delta functions on the scale of $|\im z|$. Indeed, estimates \eqref{eq:cubic_mus} and \eqref{dM_bound} imply that
	\begin{equation} \label{eq:im_delta_func}
		\im \biggl[\frac{-1}{(z-\sng)\mu_0(\sng,z)}\biggr] = \frac{\im z}{\pi|z-\sng|^2} + \mathcal{O}\bigl(|z-\sng|^{-2/3}\bigr).
	\end{equation}
	Therefore, for all $\eta_* \le N^{-\varepsilon_0}\eta_{\mathfrak{f}}$, we obtain
	\begin{equation} \label{eq:bias_delta_part}
		\int\limits_{-\Delta/2}^{+\infty} f(x_{\sigma}) \im \biggl[-\frac{1}{4}\frac{1}{(x_{\sigma}+\I\eta_*-\sng)\mu_0(\sng,x_{\sigma}+\I\eta_*)}\biggr]\mathrm{d}x = \frac{f(\sng)}{4} +\mathcal{O}\bigl(N^{-\varepsilon_0/6} + \eta_0^{1/3}\bigr).
	\end{equation}
	Similarly, we show that for for all $\eta_* \le N^{-\varepsilon_0}\eta_{\mathfrak{f}}$, the second term in the integrand on the right-hand side of \eqref{eq:bias_k} contributes at most $\mathcal{O}(N^{\varepsilon_0/6}+\eta_0^{1/3})$ to the integral in \eqref{eq:bias_k_int}.
	
	We now analyze the imaginary part of the remaining $k_1(z)$ term on the right-hand side of \eqref{eq:bias_k}. Using the estimates \eqref{eq:eig_comp} and \eqref{dm_bound}, together with the cubic equation \eqref{dm_cubic}, and arguments analogous to Section \ref{sec:dom_conv}, allows us to take the $\eta_* \to 0$ limit, and obtain
	\begin{equation} \label{eq:bias_gap_part}
		\int\limits_{-\Delta/2}^{+\infty} f(x_{\sigma}) \im \bigl[k_1(x_{\sigma}+\I\eta_*)\bigr]\mathrm{d}x = -\frac{\sqrt{3}}{2\pi}\int\limits_{1}^{+\infty}\frac{g_{\widehat{\alpha}}(t)}{\sqrt{t^2-1}(4t^2-1)}\mathrm{d}t + \mathcal{O}(\eta_0^{1/3}),
	\end{equation}
	where we applied the estimates \eqref{eq:dM_approx}, the expression \eqref{eq:sol_gap}, equation \eqref{eq:gapsol_eq}, and a change of variable $\sign(\sigma)\re\gapsol(\sign(\sigma)+2\widehat{\Delta}^{-1}(x_\sigma-\sng)) \mapsto t$. Here $g_{\widehat{\alpha}}$ is defined in \eqref{main_gap}, and $\widehat{\alpha} = \tfrac{\widehat{\Delta}}{2\eta_0}$.
	
	Combining \eqref{eq:bias_delta_part}, \eqref{eq:bias_gap_part}, and their counterparts for the ray containing $\widehat{\sng}$, we deduce from \eqref{eq:Exp_Tr_term}, \eqref{eq:Exp_Tr_term} and \eqref{eq:bias_k_int} that
	\begin{equation} \label{eq:bias_gap}
		\mathbb{E}\bigl[\Tr f(H)\bigr] - \int\limits_{\mathbb{R}} f(x)\rho(x)\mathrm{d}x = \frac{f(\sng)}{4} +  \frac{f(\widehat{\sng})}{4} - \frac{\sqrt{3}}{2\pi}\int\limits_{|t|>1}\frac{g_{\widehat{\alpha}}(t)}{\sqrt{t^2-1}(4t^2-1)}\mathrm{d}t+  \mathcal{O}\bigl(N^{-\varepsilon_0/6}\bigr).
	\end{equation}
	It follows from the definition of $g_{\widehat{\alpha}}$ in \eqref{main_gap}, the estimate in \eqref{def:delta_hat}, and the continuity of $g$, that $f(\widehat{\sng}) = g_{\widehat{\alpha}}(-1) + \mathcal{O}(\eta_0^{1/3})$ and $f(\sng) =g(\widehat{\alpha}(1))$. Therefore, in the case $\eta_{\mathfrak{f}} \le \Delta< \Delta_*$, the desired estimate \eqref{eq:main2_bias} follows form \eqref{eq:bias_gap}.
	
	Next, we analyze the case $\Delta \le N^{-\varepsilon_0/3}\eta_{\mathfrak{f}}$. Under this assumption, for all $z := \sng + w +\I \eta_1$, we have $|z-\sng| \gtrsim \Delta N^{\varepsilon_0/2}$, hence by \eqref{eq:dM_lower_bound},  $|\dM(z)| \gtrsim N^{\varepsilon_0/6}|\sigma|$. In particular, the function $k(z)$ defined in \eqref{eq:bias_k} admits the estimate
	\begin{equation}
		k(z) = \frac{1+\mathcal{O}\bigl(\Delta_0^{1/3}|z-\sng|^{-1/3}\bigr)}{3\psi \dM(z)^3} = -\frac{1+\mathcal{O}\bigl(\Delta_0^{1/3}|z-\sng|^{-1/3}\bigr)}{3 (z-\sng)\mu_0(\sng,z)},
	\end{equation} 
	where we used the cubic equation \eqref{dm_cubic} to obtain the second estimate. Using \eqref{eq:Exp_Nm_term}, \eqref{eq:Exp_Tr_term}, \eqref{eq:bias_k_int} with $\eta_* := \eta_{1}$, and \eqref{eq:im_delta_func}, we deduce that 
	\begin{equation}\label{eq:bias_cusp}
		\mathbb{E}\bigl[\Tr f(H)\bigr] - \int\limits_{\mathbb{R}} f(x)\rho(x)\mathrm{d}x = \frac{f(\sng)}{3} + \mathcal{O}\bigl(N^{-\varepsilon_0/6}\bigr).
	\end{equation}
	
	Similarly, for $\Delta \ge \Delta_*$, we have for all $z := \sng+w+\I \eta$
	\begin{equation}
		k(z) = \frac{1+\mathcal{O}\bigl(|z-\sng|^{1/3}\bigr)}{4\sigma \dM(z)^2} = -\frac{1+\mathcal{O}\bigl(|z-\sng|^{1/3}\bigr)}{4 (z-\sng)\mu_0(\sng,z)},
	\end{equation} 
	which yields
	\begin{equation} \label{eq:bias_edge}
		\mathbb{E}\bigl[\Tr f(H)\bigr] - \int\limits_{\mathbb{R}} f(x)\rho(x)\mathrm{d}x = \frac{f(\sng)}{4} + \mathcal{O}\bigl(N^{-\varepsilon_0/6}\bigr).
	\end{equation}
	Observe that for any $g\in C^2_c$ supported inside an interval of order one length, we have
	\begin{equation} \label{eq:bias_cont}
		\frac{\sqrt{3}}{2\pi}\int\limits_{|t|>1}\frac{g_{\widehat{\alpha}}(t)}{\sqrt{t^2-1}(4t^2-1)}\mathrm{d}t = \frac{g(0)}{6} + \mathcal{O}\bigl(\widehat{\alpha}^{2/3}\bigr), \quad \frac{\sqrt{3}}{2\pi}\int\limits_{|t|>1}\frac{g_{\widehat{\alpha}}(t)}{\sqrt{t^2-1}(4t^2-1)}\mathrm{d}t = \mathcal{O}\bigl(\widehat{\alpha}^{-1/2}\bigr),
	\end{equation}
	which proves the continuity of the bias functional $\mathrm{Bias}_{\widehat{\alpha}}(g)$ in the parameter $\widehat{\alpha} \in [0, +\infty]$ in \eqref{eq:Bias_converge}.
	The estimate \eqref{eq:main2_bias} in the cases $\Delta< \eta_{\mathfrak{f}}$ and $\Delta \ge \Delta_*$ follows from \eqref{eq:bias_cusp}, \eqref{eq:bias_edge}, and \eqref{eq:bias_cont}. This concludes the proof of \eqref{eq:main2_bias} and \eqref{eq:Bias_converge}.
\end{proof}
\section{Small Local Minima} \label{sect:smallrho}
In this section, we briefly outline the key differences in the case of $\sng \in \mathcal{M}_{\rho_*}$ being a small local minimum of $\rho$ with $0 <\rho(\sng) < \rho_*$, compared to the case of $\sng \in \Sng$ presented in detail above.

Let $\rho_0:=\rho(\sng)$. First of all, the comparison relations for the quantities $\rho(z)$ and $\eigB(z)$ for $z$ in the vicinity of $\sng$ collected in Proposition \ref{prop_scaling} take the following form.
\begin{lemma} (Proposition 3.2 of \cite{Erdos2018CuspUF} and Lemma 7.14 of \cite{Alt2020energy})
	For all $z \in  \mathbb{D}_{E, c}$, the quantities $\rho:=\rho(z)$, $\eigB:=\eigB(z,z)$ satisfy	
	\begin{equation} \label{eq:comp_smallrho}
		\begin{split}
			\rho &\sim \rho_0 + \dis^{1/3},\\
			|\eigB| &\sim (\rho_0 + \dis^{1/3})^2,
		\end{split}
	\end{equation}
	where $\dis\equiv\dis(z) := |z-\sng|$. Furthermore, the quantity $\sigma(\sng)$ admits the upper bound`
	\begin{equation}
		|\sigma(\sng)| \lesssim \rho_0^2.
	\end{equation}
\end{lemma}

It follows from the proof of Lemmas \ref{eq:new_term_est} and \eqref{eq:curlyT_error}, that the key probabilistic estimates \eqref{eq:new_term_est} and \eqref{eq:curlyTlaw} hold with $\Delta_0$ replaced by $\rho_0^3$.

The first line of \eqref{eq:comp_smallrho}, together with the definition \eqref{eq:etaf_supp}, imply that the local fluctuation scale around a non-zero local minimum is given by $\eta_{\mathfrak{f}}(\sng) = \min\{N^{-3/4}, \rho_0^{-1}N^{-1}\}$.

The second line of \eqref{eq:comp_smallrho} implies that the modulus of the destabilizing eigenvalue $|\eigB|$ is bounded from below by $\rho(\sng)^2$. This lower bound combined with the absence of the gap in the support close to $\sng$ makes the integral estimates presented in Section \ref{sec:variance_lemma_proof} significantly easier. On the other hand, the perturbative expansion of $\eigB(z,\zeta)$ (see Proposition \ref{prop_eigval}) has to be altered to consider the cases of $z$ and $\zeta$ lying in same or opposite upper and lower half-planes separately.
\begin{prop} \label{prop:eigB_smallrho}
	Let $z, \zeta \in \D\cap\mathbb{H}$ be two spectral parameters, then the eigenvalue of $\stab(z,\zeta)$ and $\stab(\bar z,\zeta)$ with the smallest modulus admits the following estimates
	\begin{equation}
		\begin{split}
			\eigB(z,\zeta) &= \eigB_0 - \frac{3\I\psi\rho_0}{\langle\vect{f}^2\rangle}\bigl(\dM+\other{\dM}\bigr) - \frac{\psi}{\langle\vect{f}^2\rangle} \bigl(\dM^2+\dM\other{\dM}+\other{\dM}^2\bigr)
			+ \mathcal{O}\bigl(\rho_0(|\dM|+|\other{\dM}|)(\rho_0+|\dM|+|\other{\dM}|) + \dis +\other{\dis}\bigr),\\
			\eigB(\bar z,\zeta) &= 
			- \frac{\I\psi\rho_0}{\langle\vect{f}^2\rangle}\bigl(\overline{\dM}-\other{\dM}\bigr) - \frac{\psi}{\langle\vect{f}^2\rangle}\bigl(\overline{\dM}^2+\overline{\dM}\other{\dM}+\other{\dM}^2\bigr) 
			+ \mathcal{O}\bigl(\rho_0(|\dM|+|\other{\dM}|)(\rho_0+|\dM|+|\other{\dM}|) + \dis +\other{\dis}\bigr),
		\end{split}
	\end{equation}
	where $\eigB_0 := \eigB(\sng,\sng)$, and the functions $\dM := \dM(\sng, z)$, $\other{\dM} := \dM(\sng,\zeta)$ are defined in \eqref{eq:dM_def}, the quantities $\psi$, $\sigma$ and the vector $\vect{f}$ are evaluated at $\sng$.
	
	Moreover, the estimates collected in Proposition \ref{prop:eigB} hold without change.
\end{prop}
\begin{proof} The proof of Proposition \ref{prop:eigB_smallrho} is analogous to that of Proposition \ref{prop_eigval}. \end{proof}

The final ingredient required to express the kernel $\mathcal{K}(z,\zeta)$ (defined in \eqref{K_def}) in terms of the function $\dM$, as in Lemma \ref{lemma:K_crude}, is the following equivalent of Lemma \ref{lemma_m'}.
\begin{lemma}  \label{lemma:m'_smallrho}
	Let $z$ be a spectral parameter in $\D\backslash\mathbb{R}$, and let $\beta(z)$ be as defined in \eqref{beta_est}, then 
	\begin{equation}
		\begin{split}
			\langle \vect{f}^2\rangle\eigB\frac{\m'}{\m} =& \vect{p}\vect{f}\bigl( \pi + 2(\dM+\I\rho_0)\gamma_1 %\bigl\langle|\m_0|, \frac{1-\vv\vv^*}{1-F}[\vect{f}^2\vect{p}]  
			\bigr\rangle \bigr)+ (\dM+\I\rho)\bigl(2\pi\vect{p} \frac{1-\vv\vv^*}{1-F}[\vect{f}^2\vect{p}] - \pi\vect{f}^2 \bigr) +\mathcal{O}\bigl( (\rho_0+ |\dM|)^2
			\bigr),
		\end{split}
	\end{equation}
	where $\eigB := \eigB(z,z)$, $\dM:=\dM(\sng,z)$, and the quantity $\gamma_1$ is independent of $z$ and satisfies $|\gamma_1| \lesssim 1$. Here $\vect{b},\vect{p},\vv, F$, and $\sigma$ are evaluated at $\sng$. 
\end{lemma}
\begin{proof}
	The proof of Lemma \ref{lemma:m'_smallrho} is analogous to that of Lemma \ref{lemma_m'}.
\end{proof}

Furthermore, the following lemma gives an explicit approximation to the function $\dM(\sng,\cdot)$.
\begin{lemma} (Proposition 7.10 in \cite{Alt2020energy})  \label{lemma:sol_smallrho}
	For any local minimum $\sng$ of $\rho$ satisfying $0<\rho(\sng)<\rho_*$, there exists a threshold $c_*\sim 1$, such that for all $w\in[-c_*,c_*]$, the function $\dM(\sng,\sng+w)$ admits the estimate
	\begin{equation}
		\dM(\sng, \sng+w) = \sol(w) + \mathcal{O}\bigl(\min\{|w|^{2/3},|w|\rho_0^{-1}\}\bigr)
	\end{equation}
	with the function $\sol(w)$ defined as
	\begin{equation}
		\sol(w) := \frac{\rho_0}{\sqrt{3}}\bigl(\mathfrak{q}\bigl(\sqrt{27}\pi w/(2\psi\rho_0^3)\bigr) - \I\sqrt{3} \bigr),
	\end{equation}
	where the function $\mathfrak{q}(\lambda)$ is given by
	\begin{equation}
		\mathfrak{q}(\lambda) := e^{\I\pi/3}\bigl(\sqrt{1+\lambda^2}+\lambda\bigr)^{1/3} +e^{2\I\pi/3}\bigl(\sqrt{1+\lambda^2}-\lambda\bigr)^{1/3}.
	\end{equation}
	Furthermore, the function $\mathfrak{q}(\lambda)$ satisfies the cubic equation
	\begin{equation}
		\mathfrak{q}(\lambda)^3 + 3\mathfrak{q}(\lambda) +2\lambda = 0.
	\end{equation}
\end{lemma}

Using Proposition \ref{prop:eigB_smallrho}, Lemma \ref{lemma:m'_smallrho}, and Lemma \ref{lemma:sol_smallrho}, and following the arguments laid out in Section \ref{sec:variance_lemma_proof}, we establish that an analog of Lemma \ref{lemma:variance} holds with the kernel $\other{\mathcal{K}}(w,\other{w})$ given by
\begin{equation}
	\other{\mathcal{K}}(w,\other{w}) =2\re\biggl[ \frac{\bigl(\mathfrak{q}^2 + 4\mathfrak{q}\other{\mathfrak{q}}+ \other{\mathfrak{q}}^2 - 3\bigr)(\mathfrak{q}-\other{\mathfrak{q}})^2}{(\mathfrak{q}^2+1)(\other{\mathfrak{q}}^2+1)(w-\other{w})^2} 
	- \frac{\bigl(\overline{\mathfrak{q}}^2 + 4\overline{\mathfrak{q}}\other{\mathfrak{q}}+ \other{\mathfrak{q}}^2 - 3\bigr)(\overline{\mathfrak{q}}-\other{\mathfrak{q}})^2}{(\overline{\mathfrak{q}}^2+1)(\other{\mathfrak{q}}^2+1)(w-\other{w})^2} \biggr],
\end{equation}
from which the claim of Theorem \ref{th:rho>0} follows analogously to the proof of Proposition \ref{prop:main2}. We leave the remaining details to the reader. 

\vspace{15pt}

\textbf{Acknowledgments.}
%TODO:: this
	I would like to express my gratitude to L\'aszl\'o Erd\H{o}s for his careful guidance and supervision of my work. I am also thankful to Jana Reker and Joscha Henheik for many helpful discussions.
	
\vspace{15pt}

\textbf{Funding.}
%TODO:: this
 The author was supported by the ERC Advanced Grant "RMTBeyond" No.~101020331.

\printbibliography
\end{document}